\documentclass{article}

\usepackage{algorithmicx}
\usepackage{algpseudocode}
\usepackage{arxiv}
\usepackage{stmaryrd}
\usepackage{appendix}
\usepackage[utf8]{inputenc} % allow utf-8 input
\usepackage[T1]{fontenc}    % use 8-bit T1 fonts
\usepackage{hyperref}       % hyperlinks
\usepackage{url}            % simple URL typesetting
\usepackage{booktabs}       % professional-quality tables
\usepackage{amsfonts}       % blackboard math symbols
\usepackage{nicefrac}       % compact symbols for 1/2, etc.
\usepackage{microtype}      % microtypography
\usepackage{doi}
\usepackage{fix-cm}
\usepackage{enumitem}
\usepackage{style}
\usepackage{graphicx}
\usepackage{listings}
\usepackage{tikz}
\usepackage{doi}
\usepackage{amsmath}

\numberwithin{equation}{section}

\usepackage[normalem]{ulem}
\newcommand{\Sinit}{\mathcal{S}^\text{i}}
\newcommand{\SegmentInitial}{S^\text{i}}
\newcommand{\Sdiscarded}{\mathcal{S}^{\text{d}}}

\title{A Fictitious Domain Formulation with Bubble Enrichment}

\author{ \href{https://orcid.org/0000-0001-8642-4258}{\includegraphics[scale=0.06]{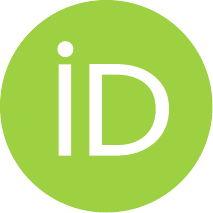}\hspace{1mm}Stefano~Berrone} \\
	Dipartimento di Scienze Matematiche\\
	``G. L. Lagrange''\\
	Politecnico di Torino, TO, 10129 \\
	\texttt{stefano.berrone@polito.it} \\
	\And
	\href{https://orcid.org/0009-0009-6203-9596}{\includegraphics[scale=0.06]{orcid.pdf}\hspace{1mm}Lorenzo~Neva} \\
	Dipartimento di Scienze Matematiche\\
	``G. L. Lagrange''\\
	Politecnico di Torino, TO, 10129 \\
	\texttt{lorenzo.neva@polito.it} \\
	\And
	\href{https://orcid.org/0000-0001-9976-6556}{\includegraphics[scale=0.06]{orcid.pdf}\hspace{1mm}Stefano~Scialò} \\
	Dipartimento di Scienze Matematiche\\
	``G. L. Lagrange''\\
	Politecnico di Torino, TO, 10129 \\
	\texttt{stefano.scialo@polito.it}  \\
	\And
	\href{https://orcid.org/0000-0001-7123-9199}{\includegraphics[scale=0.06]{orcid.pdf}\hspace{1mm}Fabio~Vicini} \\
	Dipartimento di Scienze Matematiche\\
	``G. L. Lagrange''\\
	Politecnico di Torino, TO, 10129 \\
	\texttt{fabio.vicini@polito.it} \\
}

\hypersetup{
pdftitle={A Fictitious Domain Formulation with Bubble Enrichment},
pdfsubject={},
pdfauthor={Stefano~Berrone, Lorenzo~Neva, Stefano~Scialò, Fabio~Vicini},
pdfkeywords={},
}

\begin{document}
\maketitle

\begin{abstract}
The fictitious domain method offers a powerful alternative to circumvent complex mesh generation process by embedding the complex physical domain into a simpler, non-matching background mesh. Among the various strategies to enforce the boundary condition on such a mesh, one possibility is to impose it weakly through a Lagrange multiplier. Existing approaches following this strategy typically require two independent uniform meshes for the domain and the boundary, whose mesh sizes are related by a prescribed ratio. In this work, we relax this requirement by constructing the boundary mesh directly from the trace of the background triangulation. To recover the uniform discrete inf-sup condition under this weaker mesh assumption, we enrich the discrete space of the solution with bubble functions. Our main result is the construction of a suitable operator that allows us to prove this uniform discrete inf-sup condition, thereby establishing the stability of the resulting scheme. As a consequence, we derive optimal a priori error estimates and provide a numerical experiment to validate the theoretical results. 
\end{abstract}

% keywords can be removed
\keywords{Complex Geometries, Fictitious Domain Method, Lagrange Multiplier, Bubble Function, Error Analysis}

\section{Introduction}
In engineering applications, efficiently solving partial differential equations on complex domains is a critical task. The Finite Element Method (FEM) requires the domain to be discretized by a mesh conforming to its boundary, and generating such a boundary-fitted mesh can be highly challenging when the domain has a complex or evolving geometry, often representing a major computational bottleneck. For this reason, alternative approaches have been developed in the literature to overcome this limitation. Among these, there exists a broad class of unfitted methods, for which representative examples include \cite{burman2025cut,hansbo2002unfitted}, whose main advantage lies in drastically simplifying the grid generation process by relaxing the boundary-conforming constraint.
This latter approach includes the framework of so-called fictitious domain method \cite{Girault1995, GLOWINSKI1994283}. Our approach employs the fictitious domain method for elliptic problems with Dirichlet boundary conditions by combining weak mesh requirements with a stabilized discretization strategy. The resulting formulation satisfies a uniform discrete inf-sup condition and preserves the expected convergence properties through a rigorous error analysis.

The fictitious domain strategy is applied to a wide set of problems, e.g. involving time-dependent or evolving geometries, such as fluid-structure interactions, free-boundary problems, and shape optimization. The central idea behind this method is to embed the domain into a background mesh that is simple to generate, without requiring alignment between the mesh elements and the domain boundary. While this greatly facilitates the handling of moving domains, the trade-off is that enforcing boundary conditions requires specialized techniques. In fact, in order to ensure that the obtained solution also satisfies the original problem, the fulfillment of the boundary conditions must be somehow imposed. Since an exhaustive review of all available approaches for the weak imposition of boundary conditions is beyond the scope of this work, we refer the reader to the recent literature \cite{ chouly2026review} for a comprehensive overview. Among these strategies, the pioneering work of Babu\v{s}ka \cite{Babuska1972/73} laid the foundation for the imposition of boundary conditions by means of Lagrange multipliers.  This idea is adopted in the context of fictitious domain method in \cite{Girault1995,GLOWINSKI1994283}, where elliptic Dirichlet problems were solved using conforming finite elements of degree one on a regular grid. In particular, in \cite{Girault1995} a uniform inf-sup condition was established, albeit under strong mesh assumptions: the mesh of the extended domain is required to be uniform, the boundary mesh must be suitably discretized with respect to the background mesh, with a ratio of about $3$-to-$1$, and the boundary is assumed to have no excessively sharp angles. Under these three conditions, the authors prove the uniform inf-sup stability of the discrete scheme. As a consequence, their method relies on two distinct uniform meshes, governed by two independent parameters that must be coordinated according to this ratio. The objective of our work is to introduce a fictitious domain method for elliptic problems with boundary Dirichlet conditions, relaxing the theoretical restrictive mesh assumptions required in \cite{Girault1995}. Although starting from the same mixed variational formulation and discrete problem structure, this study distinguishes itself through the specific construction of the meshes. First, we allow the background triangulations to be locally quasi-uniform. Furthermore, the mesh defined on the embedded boundary is derived from the trace of the background triangulation and subsequently processed via a coarsening algorithm. The second primary innovation of the proposed work lies in the local enhancement of the standard $\Poly{1}{}$ discrete space, a strategy specifically designed to recover the uniform discrete inf-sup condition. Our approach leverages bubble functions, originally introduced in \cite{arnold1984stable} as stabilizing terms, by adapting them to the fictitious domain framework. The use of these functions allows to prove the uniform discrete inf-sup condition, followed by an a priori error analysis.
In conclusion, a key advantage of the proposed approach is that the construction of the boundary mesh is entirely governed by a single parameter, namely the mesh size of the two-dimensional bulk mesh, thus eliminating the need for any additional independent mesh generation on the boundary. Moreover, despite this simplified meshing strategy, the resulting discrete problem retains a sound mathematical foundation, in the sense that the uniform discrete inf-sup condition is satisfied.

The paper is structured as follows.  In Section~\ref{sec:modelproblem} we present the model problem and its mixed formulation, obtained using the fictitious domain method. Section~\ref{sec:discretization} is dedicated to the discrete problem, including how the domain is discretized and which discrete spaces are used. The well-posedness of our method and error estimates are presented in Section~\ref{sec:stabilityErrorEstimates}. Section~\ref{sec:numericalresults} and Section~\ref{sec:conclusion} are dedicated to numerical experiments and conclusions, respectively. In the \hyperref[appendix]{Appendix}, a detailed proof of Lemma~\ref{Lemma:LowerBound} is presented.

\subsection{Notations}
\label{sec:notations}

We employ standard notation for Sobolev spaces throughout the paper. Let $\genericset \subset \mathbb{R}^2$ be a bounded open domain. For scalar functions $f,g \in L^2(\genericset)$ and vector-valued functions $\mathbf{s}, \mathbf{t} \in [L^2(\genericset)]^2$, the corresponding $L^2(\genericset)$ inner products are denoted by $\scal[\genericset]{f}{g}$ and $\scal[\genericset]{\mathbf{s}}{\mathbf{t}}$, respectively:

\begin{equation*}
    \scal[\genericset]{f}{g} := \int_\genericset fg \operatorname{d} \genericset, \quad \quad \scal[\genericset]{\mathbf{s}}{\mathbf{t}} := \int_\genericset \mathbf{s} \cdot \mathbf{t} \operatorname{d} \genericset .
\end{equation*}
Furthermore, $\norm[\sob{s}{\genericset}]{\cdot}$ and $\seminorm[\sob{s}{\genericset}]{\cdot}$ denote the norm and seminorm of functions in $\sob{s}{\genericset}$, for $s \geq 0$, respectively. By definition, $\sob{0}{\genericset} \equiv \leb{2}{\genericset}$. The dual space of $\sob[0]{1}{\genericset}$ is $\sob{-1}{\genericset}$, and the duality pairing will be denoted by $\langle \cdot, \cdot \rangle$. Finally, we employ fractional Sobolev spaces on the boundary of $\genericset$, adopting the standard notation; the dual space of $\sob{\frac{1}{2}}{\partial \genericset}$ is $\sob{-\frac{1}{2}}{\partial \genericset}$ and the duality pairing is indicated by $\dual{\cdot}{\cdot}$. A possible norm for the dual space is:
\begin{equation*}
\label{eq:NormasuH-1/2}
\norm[\sob{-\frac{1}{2}}{\partial \genericset}]{\mu}   := \sup_{0 \ne\theta \in \sob{\frac{1}{2}}{\partial \genericset}} \frac{\dual{\mu}{\theta}}{\norm[\sob{\frac{1}{2}}{\partial \genericset}]{\theta}}, \quad \forall \mu \in \sob{-\frac{1}{2}}{\partial \genericset}.
\end{equation*}
Throughout this work, for any sufficiently regular subset or curve $\Xi \subseteq \genericset$, the notation $v|_{\Xi}$ denotes either the trace operator or the standard restriction of $v$ to $\Xi$, mapping $H^1(\omega)$ to $L^2(\Xi)$ or $H^1(\Xi)$ respectively, depending on the codimension of the target set.

Given a generic set $E \subset \R^{n}$, $n=1,2$, let us introduce the space $\Poly{k}{E}$ of two-dimensional polynomials of degree up to $k \in \N$ on $E$. 

In the following, every generic triangle $T \subset \mathbb{R}^2$ is regarded as a closed set. The symbol $\mathring{T}$ denotes its interior, while $\partial T$ denotes its boundary. Whenever Sobolev spaces are involved, the standard notation $\sob{k}{T}$ is adopted, with the convention that it actually refers to $\sob{k}{\mathring{T}}$. Moreover, the superscript $\hat{T}$ will be used to indicate the image of $T$ in the reference configuration, following the classical finite element framework. The standard affine mapping $F_T: \hat{T} \to T$ is defined by:
\begin{equation}
\label{eq:affinetransf}
    \bm{x} =F_T (\hat{\bm{x}}) :=\bm{x}_0 + \begin{bmatrix}
        x_1 - x_0 & x_2 - x_0 \\
        y_1 - y_0 & y_2 - x_0 \\
    \end{bmatrix} 
    \begin{bmatrix}
        \hat{x} \\
        \hat{y} \\
    \end{bmatrix} \quad \forall \bm{x} \in T, \forall \bm{\hat{x}} \in \hat{T},
\end{equation}
where $\bm{x}_0:=(x_0, y_0)$, $\bm{x}_1:= (x_1, y_1)$, and $\bm{x}_2:= (x_2, y_2)$ denote the vertices of $T$, ordered in a counterclockwise sense. Finally, the symbol $C$ denotes a generic positive constant independent of the mesh size, whose value may vary in different contexts.

\section{Model problem and Fictitious Domain Formulation}
\label{sec:modelproblem}

Let $D \subset \R^2$ be a bounded domain with Lipschitz boundary $\gamma$, as in \cite{Girault1995, GLOWINSKI1994283}. With $f \in \sob{-1}{D}$ and $ g \in \sob{\frac{1}{2}}{\gamma}$, we consider the following model problem with Dirichlet boundary conditions:
\begin{equation}
    \begin{cases}
       \label{eq:modelproblem}
        \text{Find $u$ in $\sob{1}{D}$ such that} \\
        - \nabla \cdot \left( \nu \nabla u  \right) + \alpha u = f & \text{in } D,\\
        u = g &\text{on } \gamma,\\
    \end{cases}
\end{equation}
where $\nu \in \left[\leb{\infty}{D}\right]^{2 \times 2}$ is a uniformly positive diffusion coefficient, and $\alpha \in \leb{\infty}{D}$ is a non-negative reaction coefficient. For notational simplicity, we shall assume in the sequel that $\nu$ and $\alpha$ are constant. The extension to more general cases follows similar arguments. If we define $V_g := \{ v \in \sob{1}{D} \text{ } | \text{ } v = g \text{ on } \gamma \}$, the variational formulation of Problem~\eqref{eq:modelproblem} reads as:
\begin{equation}
\label{eq:varproblem}
\begin{cases}
\text{Find $u \in V_g$ such that} \\
    \dbilin[D]{u}{v} = \langle f, v \rangle & \forall v \in \sob[0]{1}{D},
\end{cases}
\end{equation}
where $\dbilin[D]{}{} : \sob{1}{D} \times \sob{1}{D} \to \R$ is the following symmetric, continuous, and coercive bilinear form
\begin{equation*}
\dbilin[D]{u}{v} = \scal[D]{ \nu\nabla u}{\nabla v} + \scal[D]{\alpha u}{v} \quad \forall u,v \in \sob{1}{D}.
\end{equation*}
According to \cite{BoffiBrezziFortin2013}, the Problem~\eqref{eq:varproblem} admits a unique solution.

Following \cite{Girault1995,GLOWINSKI1994283}, in order to obtain the fictitious domain formulation of Problem~\eqref{eq:modelproblem}, we embed $D$ into a simpler \textit{box} $\Omega \subset \R^2$ with boundary $\Gamma$, we consider an extension of the forcing term $f$ such as $\tilde{f} \in \leb{2}{\Omega}$, a closed subspace $X$ of $\sob{1}{\Omega}$, equipped with the same norm, and the following bilinear forms:
\begin{equation*}
\dbilin[\Omega]{v}{w}:= \nu \scal[\Omega]{\nabla v}{\nabla w} + \alpha\scal[\Omega]{ v}{w} \quad \forall u,v \in X,
\end{equation*}
and
\begin{equation*}
    \advbilin{w}{   \mu}:=  \dual{w}{\mu} \quad \forall w \in \sob{\frac{1}{2}}{\gamma}, \forall \mu \in  M:=\sob{-\frac{1}{2}}{\gamma}.
\end{equation*}
Typical choices for $X$ are: $\sob{1}{\Omega}$, $\sob[0]{1}{\Omega}$, or $\sob[p]{1}{\Omega} = \{ v \mid v \in \sob{1}{\Omega}, \ v \text{ is periodic on } \partial\Omega \}$ if $\Omega$ is a cartesian product of intervals. With all the previous tools, we are able to consider the following mixed problem:  
\begin{equation}
\label{eq:MixedSystem}
\begin{cases}
    \text{Find a pair $\left( \tilde{u}, \lambda \right)$ in $ X \times M$ such that } \\
     \dbilin[\Omega]{\tilde{u}}{v} + \advbilin{v|_{\gamma}}{\lambda}= \scal[\Omega]{\tilde{f}}{v}   & \forall v \in X, \\
  \advbilin{\tilde{u}|_{\gamma}}{\mu} = \advbilin{g}{\mu} &  \forall \mu \in M.
\end{cases}
\end{equation}

Due to classical results on abstract saddle-point problems \cite{BoffiBrezziFortin2013}, $(\tilde{u}, \lambda) \in X \times M$ is the only solution of Problem~\eqref{eq:MixedSystem} and satisfies $\tilde{u}|_D = u$ and $\lambda = \left[ \nu \frac{\partial \tilde{u}}{\partial n} \right]_\gamma$, i.e. the jump of the co-normal derivative of $\tilde{u}$ across $\gamma$. Moreover, the bilinear form $b$ shall satisfies the so-called continuous inf-sup condition: there exists a constant $\beta > 0$ such that  
\begin{equation}
\label{eq:InfSupContinua}
\sup_{v \in X} \frac{\advbilin{v|_{\gamma}}{\mu}}{\norm[\sob{1}{\Omega}]{v}} \geq \beta \norm[M]{\mu} \quad \forall \mu \in M.
\end{equation}

\section{Discretization}
\label{sec:discretization}
Let $h>0$, $\eta >0$ be two real parameters and $X_h \subset X$, $ M_\eta \subset M$ be two finite dimensional spaces. As in \cite{Girault1995}, assume $M_\eta$ contains at least the constant functions. The discrete version of the Problem~\eqref{eq:MixedSystem} is the following
\begin{equation}
\label{eq:DiscreteMixedProblem}
\begin{cases}
    \text{Find a pair $\left( u_h, \lambda_\eta \right)$ in $ X_h \times M_\eta$ such that } \\
     \dbilin[\Omega]{u_h}{v_h} + \advbilin{v_h|_{\gamma}}{\lambda_\eta}= \scal[\Omega]{\tilde{f}}{v_h}   & \forall v_h \in X_h, \\
  \advbilin{u_h|_{\gamma}}{\mu_\eta} = \advbilin{g}{\mu_\eta} &  \forall \mu_\eta \in M_\eta.
\end{cases}
\end{equation}

It follows from the abstract discretization theory of mixed problems \cite{BoffiBrezziFortin2013} that good error estimates can be established for the solution of Problem~\eqref{eq:DiscreteMixedProblem} if two conditions are fulfilled. Let 
\begin{equation*}
V_h = \{ v_h \in X_h : \advbilin{v_h|_\gamma}{\mu_\eta} = 0 \quad \forall \mu_\eta \in M_\eta \}.
\end{equation*}  
Then $ a_{\Omega} $ must be uniformly coercive on $ V_h $, i.e., must exists a constant $ \kappa^* > 0 $, independent of $ h $ and $ \eta $, such that  
\begin{equation}
\label{eq:ellipticity}
a_{\Omega}(v_h, v_h) \geq \kappa^* \norm[\sob{1}{\Omega}]{v_h}^2 \quad \forall v_h \in V_h,
\end{equation}  
and $ b $ must satisfies a uniform discrete inf-sup condition, i.e., must exists a constant $ \beta^* > 0 $, independent of $ h $ and $ \eta $, such that  
\begin{equation}
\label{eq:inf-sup}
\sup_{v_h \in X_h} \frac{\advbilin{v_h|_{\gamma}}{\mu_\eta}}{\norm[\sob{1}{\Omega}]{v_h}} \geq \beta^* \norm[M]{\mu_\eta} \quad \forall \mu_\eta \in M_\eta.
\end{equation}
The condition in Equation~\eqref{eq:ellipticity} for all nonnegative $\alpha$ and positive $\nu$ is proved in \cite{Girault1995}, while Equation~\eqref{eq:inf-sup} is not inherited by the continuous counterpart. This fact makes the choice of the discrete spaces $X_h$ and $M_\eta$ crucial for well-posedness. Given the difficulty to handle a dual norm, we will use a results from \cite{Fortin1977}: 

\begin{lemma}
\label{Lemma:OperatorediFortin}
    Assume that $ b $ satisfies the continuous inf-sup condition~\eqref{eq:InfSupContinua}. Then, the discrete inf-sup condition~\eqref{eq:inf-sup} holds if and only if there exists a bounded linear operator $ \Pi_h : X \to X_h $ such that: \begin{equation}
\label{eq:FortinOperator1}
\norm[\sob{1}{\Omega}]{ \Pi_h(v)} \leq C \norm[\sob{1}{\Omega}]{ v} \quad \forall v \in X,
\end{equation}  
where $ C > 0 $ is a constant independent of $ h $ and $ \eta $, and  
\begin{equation}
\label{eq:FortinOperator2}
\advbilin{\Pi_h(v)|_{\gamma} - v|_{\gamma}}{ \mu_\eta} = 0 \quad \forall \mu_\eta \in M_\eta.
\end{equation}   
\end{lemma} 

The following subsections are devoted to the construction of the meshes 
for $\Omega$ and $\gamma$, as well as to the definition of the discrete 
spaces $X_h$ and $M_\eta$. As highlighted in the introduction, the proposed approach presents two key 
advantages with respect to~\cite{Girault1995}: it extends the framework 
to a broader class of mesh configurations, and it removes the need for 
an independently generated mesh on $\gamma$, which would introduce an 
additional free parameter requiring careful tuning. This is realized by coupling the 
boundary discretization directly to the bulk mesh on $\Omega$. 
This construction exploits the intersections between the background 
triangulation and the embedded boundary to define the discretization 
of $\gamma$, and, combined with a different choice of discrete spaces, 
allows us to prove stability without requiring the specific $1{:}3$ 
mesh ratio assumed in~\cite{Girault1995}.

\subsection{Domain discretization}
\label{subsec:domaindiscretization}

Let us consider a sequence of triangulations $ \left( \mathcal{T}_h \right)_{h>0}$ on $\Omega$ that is shape-regular in the sense of \cite{ciarlet1978finite}. Let us also assume that each triangulation $\mathcal{T}_h$ is locally quasi-uniform in the sense of \cite{CanutoNochetto2024}. To formalize this, for any element $T \in \mathcal{T}_h$, we introduce the discrete neighborhood
\begin{equation}
\label{eq:discreteneighborhood}
     \omega_{\mathcal{T}_h}(T) := \bigcup_{\substack{T' \in \mathcal{T}_h \\ T' \cap T \neq \emptyset}} T'.
\end{equation}
The local quasi-uniformity property then requires that
\[
    \max_{T \in \mathcal{T}_h} \#\omega_{\mathcal{T}_h}(T) \leq C(\sigma), \qquad \max_{T' \subset \omega_{\mathcal{T}_h}(T)} \frac{|T|}{|T'|} \leq C(\sigma), 
    \]
    where $\sigma$ denotes the shape regularity coefficient. We denote by $\mathcal{E}_h$ the set of edges $e$ for all $T \in \mathcal{T}_h$ and by $\mathcal{V}_h$ the set of vertices of $\mathcal{T}_h$. For each triangle $T \in \mathcal{T}_h$, we denote by $h_T$ its diameter, i.e.\ the length of its longest edge. The value $h$ denotes the length of the longest edge $e \in \mathcal{E}_h$, i.e.\ $h := \max_{T \in \mathcal{T}_h} h_T$.

\subsection{Boundary discretization}
    \label{sec:boundarydiscret}
    We assume that the boundary $\gamma$ is polygonal,  composed of a union of $N$ straight segments $ \{  \xi_1, \dots, \xi_N \}$, called \textit{boundary segments}, i.e. $\gamma = \bigcup_{i=1}^N \xi_i. $  See Figure~\ref{fig:boundary_segme}(a) for an example. The extension of our work in case of curved boundary $\gamma$ can be done as in \cite{Girault1995}, see Remark~\ref{rm:curvedboundary}.
Furthermore, for simplicity of presentation, we assume that the corner points of $\gamma$ are not too sharp, i.e., that the angles between consecutive boundary segments are uniformly bounded away from zero, as in \cite{Girault1995}. This assumption is introduced only to simplify the analysis; the results of this paper can be extended to configurations where it is violated. In particular, Remark~\ref{rem:sharp} outlines how the proposed theory can be adapted to handle such cases without loss of validity. 

\begin{figure}
    \centering
    \begin{subfigure}[b]{0.35\linewidth}
        \includegraphics[width=\linewidth]{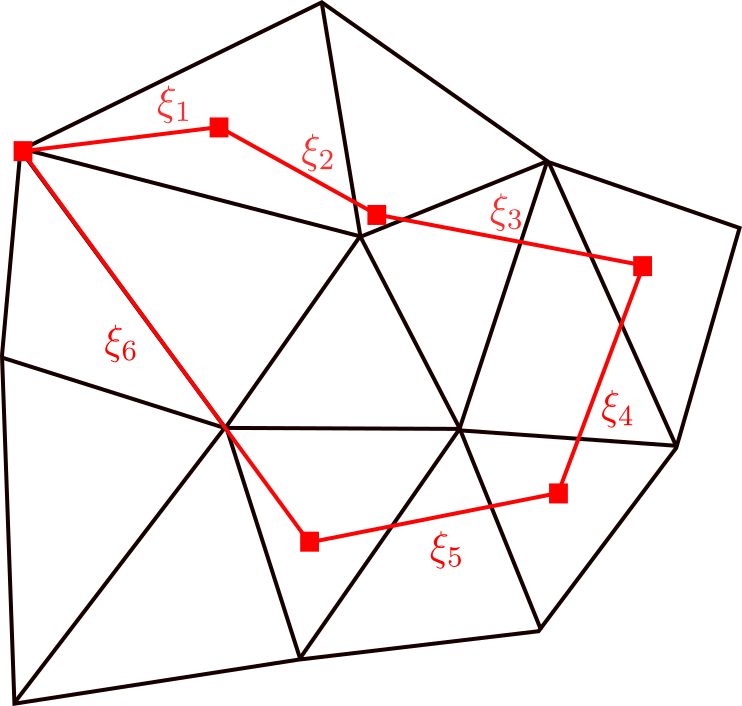}
        \caption{}
    \end{subfigure}
    \hspace{1cm}
    \begin{subfigure}[b]{0.35\linewidth}
        \includegraphics[width=\linewidth]{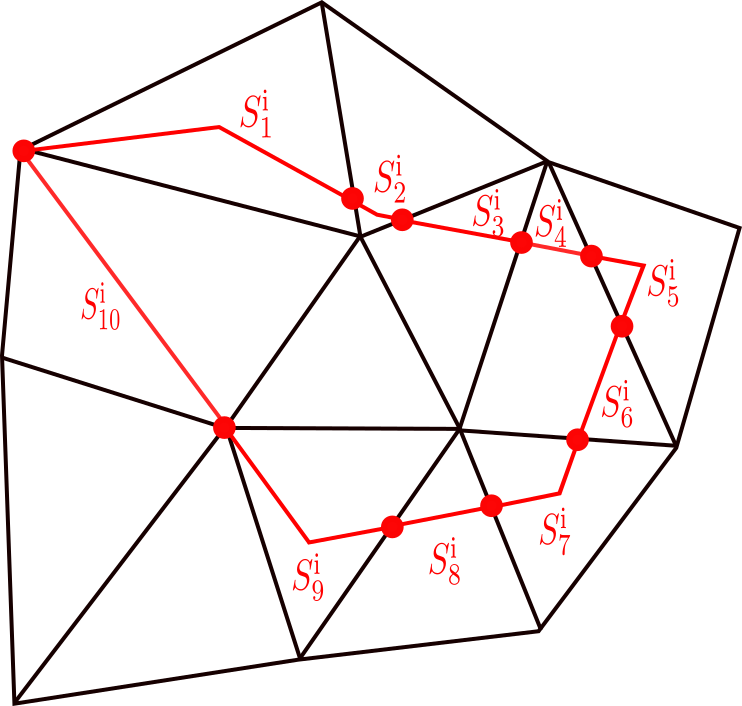}
        \caption{}
    \end{subfigure}
    \caption{An illustration of the initial boundary discretization, where a red polygon boundary is embedded in a black triangulation. Figure~(a) shows the initial segmentation of the boundary, consisting solely of the boundary segments $\xi_1, \dots, \xi_6$. Figure~(b) displays the discretization $\Sinit$ of the boundary induced by the triangulation. The endpoints of \textit{boundary segments} are marked with red squares, those of elements of $\Sinit$ with red circles.} 
    \label{fig:boundary_segme}
\end{figure}

To discretize $\gamma$, we start from the discretization naturally \emph{induced} $\Sinit$ by the triangulation $\mathcal{T}_h$: the available geometric information consists of the \textit{boundary segments}, and $\mathcal{T}_h$ induces a first tentative discretization of $\gamma$ by grouping together all the portions of the boundary lying in the same triangle, as shown in Figure~\ref{fig:boundary_segme}(b). Each element of $\Sinit$ is therefore a finite union of portion of \textit{boundary segments}. In the following, the elements of $\Sinit$ will be denoted with a superscript $^\text{i}$. Moreover, the following property holds: 
\begin{equation}
\label{eq:OgniScontenutoinunT}
    \forall \SegmentInitial \in \Sinit, \quad  \exists !  \ T \in \mathcal{T}_h : \SegmentInitial \cap \mathring{T} \ne \emptyset \quad  \oplus \quad \exists!e \in \mathcal{E}_h: \SegmentInitial \equiv e   .
\end{equation}
The symbol $\oplus$ indicates the XOR. This induced discretization may contain very small elements that can cause instabilities in the discrete formulation by degrading the conditioning of the resulting linear system. For this reason, the induced discretization is refined through a coarsening procedure, which first classifies its elements and then merges together those that are too small, so that every element of the resulting discretization satisfies the desired minimum size constraint. We now turn to the description of this classification.

\begin{figure}
    \centering
    \includegraphics[width=0.4\linewidth]{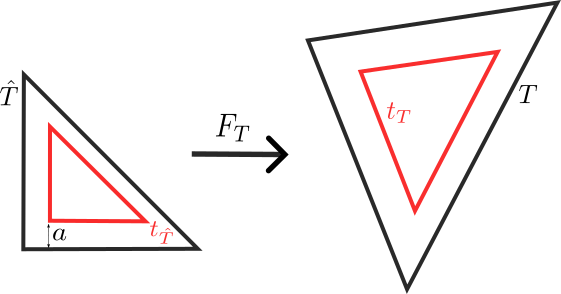}
    \caption{Affine map $F_T:\hat{T}\to T$ from the reference element $\hat{T}$ to a generic physical triangle $T$. The sub-triangle $t_{\hat{T}}\subset\hat{T}$, related to parameter $a\in [0, \frac{1}{3}$), is mapped onto $t_T = F_T(t_{\hat{T}})\subset T$, illustrating how the sub-triangulation is created.}
\label{fig:reference_triangle}
    \label{fig:TriangRef}
\end{figure}

We classify each element $\SegmentInitial \in \Sinit$ as \textit{definitive} or \textit{discarded}, according to whether it is already large enough or not. To this end, to each triangle $T \in \mathcal{T}_h$ we associate a smaller, scaled copy $t_T \subset T$, as illustrated in Figure~\ref{fig:reference_triangle}, and we use the intersection between the elements $\SegmentInitial \in \Sinit$ and these sub-triangles to decide the classification of $\SegmentInitial$. For each triangle $T \in \mathcal{T}_h$, the sub-triangle $t_T$ is defined as
\begin{equation}
\label{eq:subtriangle}
    t_T := \{ \bm{x} = F_T(g(\hat{\bm{x}})) \mid \hat{\bm{x}} \in \hat{T} \} \subset T,
\end{equation}
where $F_T$ is the map from the reference triangle $\hat{T}$ to $T$, and $g : \hat{T} \to \hat{T}$ is the scaling transformation
\begin{equation}
\label{eq:TrasformScaling}
    g(\hat{\bm{x}}) :=(1-3a)\hat{\bm{x}}+\begin{bmatrix}
        a \\
        a \\
    \end{bmatrix} ,
\end{equation}
with $a \in \R$ chosen in the interval $[0, \frac{1}{3})$. With this definition, $t_T$ is a scaled copy of $T$, with size depending on $a$, that lies inside $T$, is similar to $T$, and shares the same barycenter. The elements of $\Sinit$ are processed sequentially: an element is accepted if it either intersects the boundary of a subtriangle not considered yet or coincides with an edge of $\mathcal{T}_h$; otherwise, the classification is updated by discarding the conflicting elements. To keep track of this process, we introduce three sets, namely $\mathcal{T}^I_h$, $\mathcal{T}^e_h$, and $\mathcal{E}_h^e$, initially empty and progressively filled, as each element of $\Sinit$ is processed, with the triangles and edges of the background mesh. The complete classification procedure is summarized in Algorithm~\ref{alg:classification}.
The output of the algorithm consists of a partition of the elements $S$ of $\Sinit$ into \emph{definitive} and \emph{discarded} elements, together with the sets $\mathcal{T}^I_h$ of triangles containing exactly one \emph{definitive} element in their interior, $\mathcal{T}^e_h$ of triangles having a \emph{definitive} element coinciding with one of their edges, and $\mathcal{E}_h^e$ of mesh edges that coincide with a \emph{definitive} element of $\Sinit$.

\begin{algorithm}[h]
\caption{Classification of the elements of $\Sinit$}\label{alg:classification}

\KwIn{The set of elements $\Sinit$, together with the empty sets $\mathcal{T}_h^I$, $\mathcal{T}_h^e$, and $\mathcal{E}_h^e$.}

\For{each $\SegmentInitial \in \Sinit$}{
    \eIf{$\SegmentInitial$ coincides with an edge $e \in \mathcal{E}_h$ and the neighboring triangles do not belong to $\mathcal{T}_h^I$}{
        Classify $\SegmentInitial$ as \emph{definitive}\;
        Add $e$ to $\mathcal{E}_h^e$\;
        Add the triangles sharing $e$ to $\mathcal{T}_h^e$\;
    }{
        \eIf{$\SegmentInitial \cap \partial t_T \neq \emptyset$ and $T \notin \mathcal{T}_h^I \cup \mathcal{T}_h^e$}{
            Classify $\SegmentInitial$ as \emph{definitive}\;
            Add $T$ to $\mathcal{T}_h^I$\;
        }{
            Classify $\SegmentInitial$ as \emph{discarded}\;
            \ForAll{$T \in \mathcal{T}_h^I \cup \mathcal{T}_h^e$ intersected by $\SegmentInitial$}{
                Let $\tilde{\SegmentInitial}$ be the \textit{definitive} element that caused $T$ to be included in $\mathcal{T}_h^I\cup \mathcal{T}_h^e$\;
                Reclassify $\tilde{\SegmentInitial}$ as \emph{discarded}\;
                Remove $T$ from $\mathcal{T}_h^I$ or $\mathcal{T}_h^e$\;
                \If{$T \in \mathcal{T}_h^e$}{
                    Remove the corresponding edge from $\mathcal{E}_h^e$\;
                }
            }
        }
    }
}

\KwOut{The partition of $\Sinit$ into \emph{definitive} and \emph{discarded} elements, together with the sets $\mathcal{T}_h^I$, $\mathcal{T}_h^e$, and $\mathcal{E}_h^e$.} 

\end{algorithm}

\begin{figure}
    \centering 
    \begin{subfigure}[b]{0.35\linewidth}
        \includegraphics[width=\linewidth]{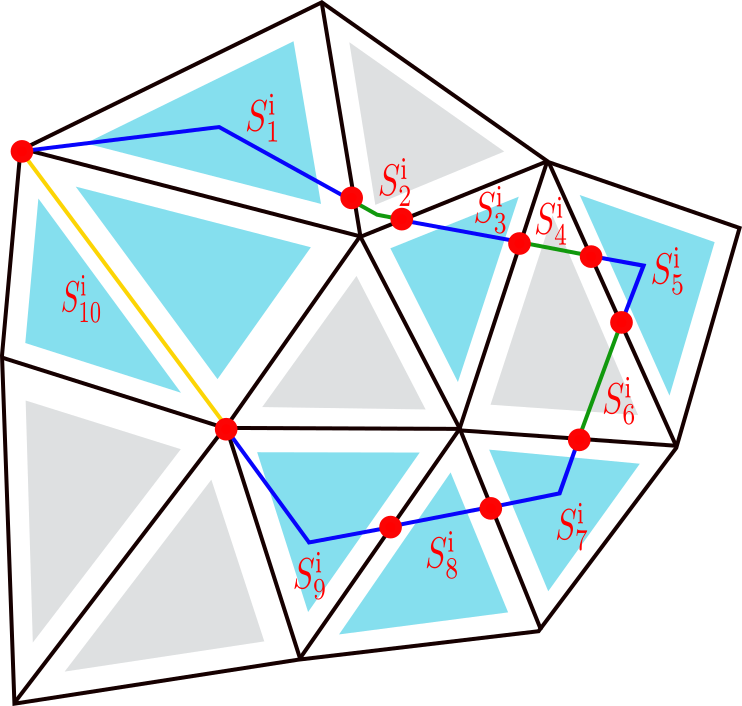}
        \caption{}
    \end{subfigure}
 \hspace{1cm}
    \begin{subfigure}[b]{0.35\linewidth}
        \includegraphics[width=\linewidth]{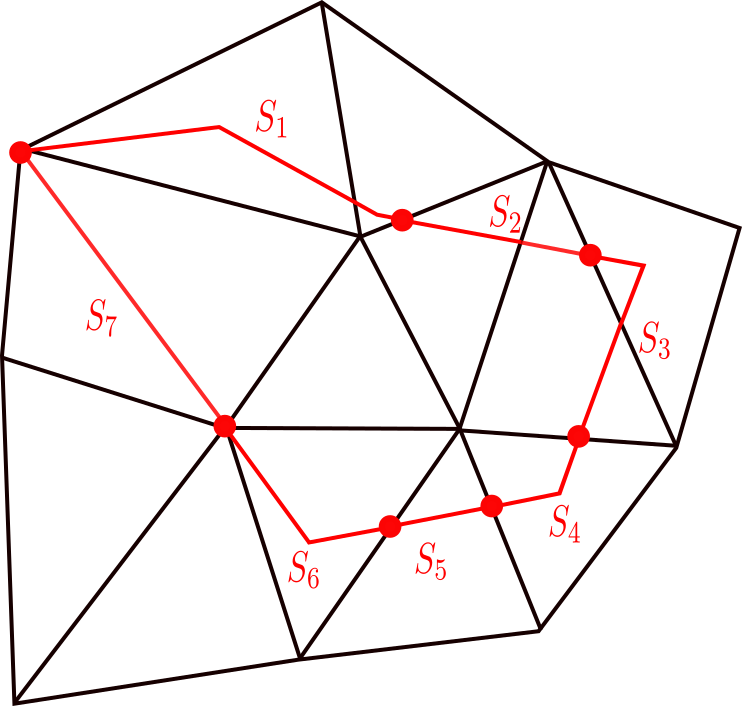}
        \caption{}
    \end{subfigure}
    \caption{An illustration of the final steps of the boundary discretization procedure, applied to the same example as in Figure~\ref{fig:boundary_segme}. In each figure, the endpoints of the 1D elements are marked with red dots. Figure~(a) illustrates the outcome of the classification algorithm: the elements of $\Sinit$ classified as \textit{definitive} are shown in blue, those classified as \textit{discarded} in green, and those classified as \textit{definitive} but coinciding with an edge of the triangulation in yellow. The internal subtriangles belonging to $\mathcal{T}^I_h \cup \mathcal{T}_h^e$ are colored in light blue, while the remaining ones are shown in grey. Finally, Figure~(b) presents the final discretization $\mathcal{S}_\eta$.}
    \label{fig:boundary_segme2}
\end{figure}

Let us analyze the example reported in Figure~\ref{fig:boundary_segme2}(a). The internal subtriangles are colored in blue if they belong to $\mathcal{T}^I_h \cup \mathcal{T}_h^e$, and in grey otherwise. The \textit{definitive} elements are $\SegmentInitial_1, \SegmentInitial_3, \SegmentInitial_5, \SegmentInitial_7, \SegmentInitial_8, \SegmentInitial_9$, and $\SegmentInitial_{10}$, where the latter is an example of an element coinciding with an edge of the triangulation $\mathcal{T}_h$. On the other hand, $\SegmentInitial_2$ is \textit{discarded} as it does not intersect any internal subtriangle, while $\SegmentInitial_4$ and $\SegmentInitial_6$ are \textit{discarded} despite intersecting an internal subtriangle, since they both share the same triangle. This situation illustrates the application of lines 12--19 of Algorithm~\ref{alg:classification}, which is also described in detail in the following two examples.

\begin{figure}[ht]
    \centering
    \begin{subfigure}[b]{0.23\linewidth}
        \includegraphics[width=\linewidth]{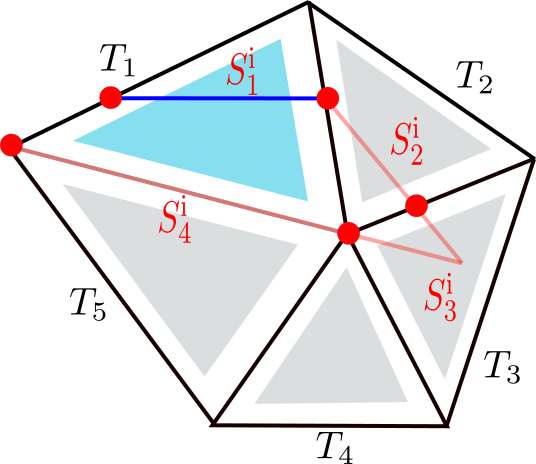}
        \caption{$\mathcal{T}^I_h = \{ T_1 \}$.}
    \end{subfigure}
    \hfill
    \begin{subfigure}[b]{0.23\linewidth}
        \includegraphics[width=\linewidth]{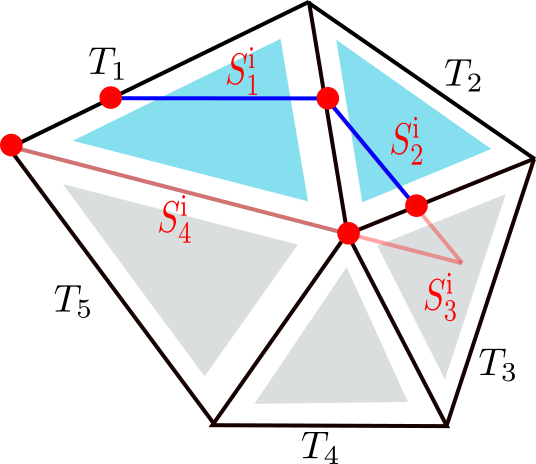}
        \caption{$\mathcal{T}^I_h = \{ T_1, T_2 \}$.}
    \end{subfigure}
    \hfill
    \begin{subfigure}[b]{0.23\linewidth}
        \includegraphics[width=\linewidth]{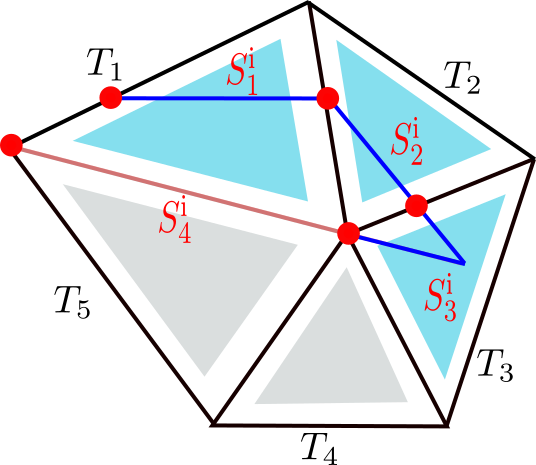}
        \caption{$\mathcal{T}^I_h = \{ T_1, T_2, T_3 \}$.}
    \end{subfigure}
    \hfill
    \begin{subfigure}[b]{0.23\linewidth}
        \includegraphics[width=\linewidth]{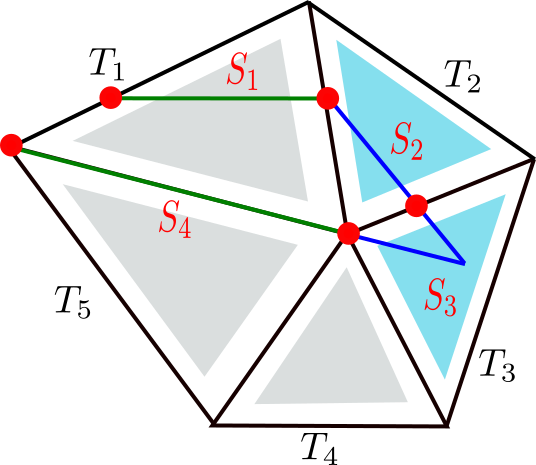}
        \caption{$\mathcal{T}^I_h = \{ T_1, T_2 \}$.}
    \end{subfigure}
    \caption{Illustration of Example~\ref{ex:Example2}. The elements of $\Sinit$ are processed sequentially, and the evolution of this process is described from Figure (a) to Figure (d). In each figure, the endpoints of the 1D elements are marked with a red dot. The elements currently classified as \emph{definitive} are shown in blue, those not yet processed are shown in semi-transparent red, while those classified as \emph{discarded} are shown in green. The triangles belonging to $\mathcal{T}_h^I \cup \mathcal{T}_h^e$ are highlighted in light blue, whereas the remaining triangles are shown in grey.}
    \label{fig:Esempio1}
\end{figure}

\begin{example}
\label{ex:Example2}
    The following example illustrates a situation in which a previously accepted element is reclassified as \emph{discarded}. Figure~\ref{fig:Esempio1} shows the successive steps of the algorithm, with each panel corresponding to one iteration. In Figure~\ref{fig:Esempio1}(a), the element $\SegmentInitial_1$ is processed first: since it intersects the boundary of an internal subtriangle that is not associated with any previously accepted element, $\SegmentInitial_1$ is classified as \emph{definitive}, and the corresponding triangle is added to $\mathcal{T}_h^I$. In Figure~\ref{fig:Esempio1}(b), the element $\SegmentInitial_2$ is processed and classified as \emph{definitive}, leading to the inclusion of the corresponding triangle in $\mathcal{T}_h^I$. Next, as shown in Figure~\ref{fig:Esempio1}(c), the element $\SegmentInitial_3$ is processed, classified as \emph{definitive} and its corresponding triangle is included in $\mathcal{T}_h^I$. When $\SegmentInitial_4$ is processed (Figure~\ref{fig:Esempio1}(d)), it coincides with an edge of the background triangulation for which one of the neighboring triangle belongs to $\mathcal{T}_h^I$. According to Algorithm~\ref{alg:classification} (lines 12--19), $\SegmentInitial_4$ is classified as \emph{discarded}, the conflicting triangle is removed from $\mathcal{T}_h^I$, and the previously accepted element $\SegmentInitial_1$ is also reclassified as \emph{discarded}. This example illustrates how the algorithm resolves conflicts by removing both the newly processed element and the previously accepted one.
\end{example}

\begin{figure}
    \centering 
    \begin{subfigure}[b]{0.15\linewidth}
        \includegraphics[width=\linewidth]{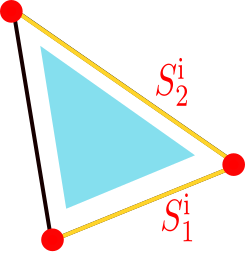}
        \caption{}
    \end{subfigure}
 \hspace{1cm}
    \begin{subfigure}[b]{0.15\linewidth}
        \includegraphics[width=\linewidth]{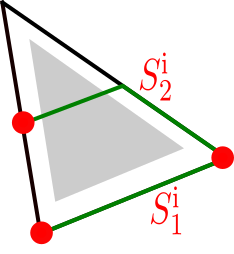}
        \caption{}
    \end{subfigure}
    \caption{Illustration of Example~\ref{ex:Example1}. The elements of $\Sinit$ classified as \emph{definitive} coinciding with an edge of the triangulation are shown in yellow, those classified as \emph{discarded} in green. The triangles belonging to $\mathcal{T}_h^I \cup \mathcal{T}_h^e$ are highlighted in light blue, while the remaining triangles are shown in grey.}
    \label{fig:Esempio2}
\end{figure}

\begin{example}
    \label{ex:Example1}
    The example in Figure~\ref{fig:Esempio2} illustrate two special configurations, where elements of $\Sinit$ are related with edges of the background triangulation. In Figure~\ref{fig:Esempio2}(a), the two elements $\SegmentInitial_1$, $\SegmentInitial_2$ of the induced mesh each coincide with an edge of the same background triangle. Both elements are accepted as \emph{definitive}, and the involved triangle is stored in $\mathcal{T}^e_h$. Conversely, Figure~\ref{fig:Esempio2}(b) shows a different configuration: the element $\SegmentInitial_1$ is processed first and classified as \emph{definitive}, causing the corresponding edge and its neighboring triangles to be added to $\mathcal{E}_h^e$ and $\mathcal{T}_h^e$, respectively. When $\SegmentInitial_2$ is subsequently processed, since the background triangle already belongs to $\mathcal{T}_h^e$, according to Algorithm~\ref{alg:classification} (lines 12--19), $\SegmentInitial_2$ is classified as \emph{discarded}. Moreover, the previously accepted element $\SegmentInitial_1$ is reclassified as \emph{discarded}, and the corresponding edge and neighboring triangles are removed from $\mathcal{E}_h^e$ and $\mathcal{T}_h^e$, respectively.
\end{example}

After the classification process, we are able to define the definitive discretization of $\gamma$, denoted by $\mathcal{S}_\eta$, where $\eta$ is its mesh-size. Each group of consecutive \textit{discarded} elements of $\Sinit$ is first merged together, and the resulting element is treated in one of two ways: if its length exceeds $3h$, where $h$ is the local mesh size of the intersected triangles, it is included in $\mathcal{S}_\eta$ as a new element, and also collected in the subset $\Sdiscarded \subseteq \mathcal{S}_\eta$, used later for classification purposes; otherwise, it is further merged into the preceding \textit{definitive} element, forming a single element of $\mathcal{S}_\eta$. All remaining \textit{definitive} elements of $\Sinit$ that are not involved in such a merging are included in $\mathcal{S}_\eta$ unchanged.
The reason for adopting this agglomeration criterion will be clarified in the proof of Proposition~\ref{Proposition:proprietàfunzionibolla}. Furthermore, some elements $S \in \mathcal{S}_{\eta}$ may now have a nonempty intersection with more than one triangle $T \in \mathcal{T}_h$. Consequently, Equation~\eqref{eq:OgniScontenutoinunT} is replaced by:
\begin{equation}
\label{eq:Scontenutonelsottotriangolo}
    \forall S \in \mathcal{S}_{\eta} \Big( \exists !   T \in \mathcal{T}^I_h \ \text{ such that } S \cap \partial t_T \ne \emptyset \quad \oplus \quad \exists! e \in \mathcal{E}^e_h : e \subset S \quad \oplus \quad S \in \Sdiscarded \Big).
\end{equation}

Let us focus on the example proposed in Figure~\ref{fig:boundary_segme2}(b). It shows various elements $S \in \mathcal{S}_\eta$ with different configurations: those entirely contained within a single triangle (e.g., $S_5$); those resulting from the agglomeration of multiple elements where at least one was previously categorized as a \textit{definitive} element, such as $S_3$; and those that coincide with an edge, such as $S_7$. Furthermore, we can deduce how the choice of $a$ impacts onto the discretization of $\gamma$: increasing the value of $a$ correspond to coarsening the mesh for boundary.

It is important to note that, in most cases, the set $\Sdiscarded$ is not actually generated by the algorithm, as it arises only for sufficiently large values of $a$, i.e.\ when the sub-triangles $t_T$ are small enough that several consecutive elements $\SegmentInitial$ fail to intersect any of them and are therefore all classified as \textit{discarded}.

\subsection{Discrete spaces}
\label{subsec:DiscreteSpaces}
We assume that $X = \sob{1}{\Omega}$ as in \cite{Girault1995}. The discrete space $X_h$ is built upon three different spaces: the objective is to enrich the classical $\Poly{1}{}$-finite element space with additional functions that will be used to define the operator introduced in the Lemma~\ref{Lemma:OperatorediFortin}. Notably, this 
enrichment is restricted to a specific subset of triangles, specifically 
those intersected by the boundary $\gamma$. It is precisely for this reason 
that the sets $\mathcal{T}^I_h, \mathcal{E}^e_h,$ and $\mathcal{T}^e_h$ were formally defined in Section~\ref{subsec:domaindiscretization}

Let us consider $T \in \mathcal{T}^I_h$ and the set $\{\varphi_i^T \}_{i=1}^3$ of local linear finite element basis function on $T$. Then, we can define the \textit{internal bubble function} in $T$ as \begin{equation} \label{eq:Bubble3} b_{h}^T(\bm{x}) = 27 \varphi_1^T(\bm{x})\varphi_2^T (\bm{x})\varphi_3^T(\bm{x}) \quad \forall \bm{x} \in T. \end{equation}
This function belongs to $ \Poly{3}{T} \cap \sob[0]{1}{T}$.
The space of \textit{internal bubble functions} is \begin{equation*} B^I_h = \bigoplus_{T\in \mathcal{T}^I_h} \operatorname{span}{ b_h^T } \subset X. \end{equation*} 

Now, let us consider an edge $e \in \mathcal{E}_h^e$, with $T_1, T_2 \in \mathcal{T}^e_h$ the two triangles sharing $e$, and let $v_1, v_2 \in \mathcal{V}_h$ denote its two endpoints. For $k=1,2$, let $\{\varphi_i^{T_k}\}_{i=1}^3$ denote, as above, the set of local linear finite element basis functions on $T_k$, and let $\varphi_1^{T_k}$ and $\varphi_2^{T_k}$ be the two basis functions associated with $v_1$ and $v_2$, respectively. The \textit{edge bubble function} $b_h^e$ is then defined element-wise as
\begin{equation}
\label{eq:Bubble2}
b_h^e(\bm{x}) =
\begin{cases}
4\, \varphi_1^{T_1}(\bm{x}) \, \varphi_2^{T_1}(\bm{x}), & \bm{x} \in T_1, \\[4pt]
4\, \varphi_1^{T_2}(\bm{x}) \, \varphi_2^{T_2}(\bm{x}), & \bm{x} \in T_2, \\[4pt]
0, & \text{otherwise},
\end{cases}
\end{equation}
This function belongs to $C^0\left( T_1 \cup T_2 \right)$. The space of \textit{edge bubble functions} is \begin{equation*}
  B^e_h = \bigoplus_{ e \in \mathcal{E}^e_h} \operatorname{span}\{ b^e_h \} \subset X .
\end{equation*}

The final discrete space is
\begin{equation*}
    X_h = \{ v_h \in C^0(\Omega):v_{h }|_T \in \Poly{1}{T} \quad \forall T \in \mathcal{T}_h \}    \oplus  B^I_h \oplus B^e_h.
\end{equation*}

Then, we set
\begin{equation}
\label{eq:M_eta}
M_\eta = \{\mu_\eta \in \leb{2}{\gamma}: \mu_{\eta }|_S \in \Poly{0}{S} \quad \forall S \in \mathcal{S}_\eta \}.
\end{equation}

\begin{remark}
    \label{rm:curvedboundary}
    The extension of our method to curved boundaries is straightforward. We use a polygonal line to approximate the boundary, in order to apply the fictitious domain method described in our work. This approximation is created connecting the points of intersection between $\gamma$ and $\mathcal{E}_h$. The analysis in \cite{Girault1995} shows that replacing the curved boundary with a suitably chosen polygonal approximation does not affect the convergence order, since we look for a (enhanced) linear approximation of the solution.
\end{remark}

\section{Stability and error estimates}
\label{sec:stabilityErrorEstimates}
    The objective is to define the operator involved in Lemma~\ref{Lemma:OperatorediFortin}, following the construction in \cite{Girault1995}. Let $R_h$ be the Cl\'ement quasi-interpolation operator, associated with $X_h$, introduced in \cite{clement1975approximation}. Recall that for any $ v $ in $ X $, $ R_h(v) $ belongs to $ X_h $, and $ R_h $ satisfies the following local error estimates for $ m = 1,2 $:
\begin{equation}
\label{eq:estimates on the clement operator}
\begin{aligned}
    \forall T \in \mathcal{T}_h, \forall v \in \sob{m}{\omega_{\mathcal{T}_h}(T)}, \quad & \norm[\leb{2}{T}]{ R_h(v) - v } \leq C h_T^m \seminorm[\sob{m}{\omega_{\mathcal{T}_h}(T)}]{v}, \\
    & \seminorm[\sob{1}{T}]{R_h(v)-v} \leq C h_T^{m-1} \seminorm[\sob{m}{\omega_{\mathcal{T}_h}(T)}]{v},
\end{aligned}
\end{equation}
where $\omega_{\mathcal{T}_h}(T)$ is defined in Equation~\eqref{eq:discreteneighborhood}. We derive the following result:
\begin{proposition}
\label{Proposition:proprietàfunzionibolla}
    For each $S \in \mathcal{S}_\eta$, exists $\psi_S \in X_h$ such that:
    \begin{equation}
\label{eq:positiviàintegrale}
    \int_S\psi_S|_S  >0,
\end{equation}
and simultaneously
\begin{equation}
\label{eq:BolleSupportoNulloInAltriSegmenti}
   \int_S \psi_{S'}|_S =0 \quad \forall S' \in \mathcal{S}_\eta : S' \ne S.
\end{equation}
\end{proposition}

\begin{proof}
Let us consider $S \in \mathcal{S}_\eta $. Using Equation~\eqref{eq:Scontenutonelsottotriangolo}, we know that $S$ must be related to a triangle, an edge or $S \in \Sdiscarded$. The key idea relies on the use of different functions:
\begin{itemize}
    \item if exists $T \in \mathcal{T}^I_h \ \text{ such that } S \cap \partial t_{T} \ne \emptyset$, then $\psi_S$ is chosen in $B^I_h$ as the \textit{internal bubble function} related to $T$. This triangle is now denoted as $T_S$, and $\psi_S:=b_h^{T_S}$. The Equation~\eqref{eq:positiviàintegrale} is satisfied since the function is positive in $\mathring{T}_S$, equal to zero in $\Omega \setminus \mathring{T}_S $ and $S \cap \mathring{T}_S \ne \emptyset$;
    \item if exists $e \in \mathcal{E}^e_h $ such that $ e \subset S $, $\psi_S$ is chosen in $B^e_h$ as $b_h^e$, the \textit{edge bubble function} related to $e$. The Equation~\eqref{eq:positiviàintegrale} is satisfied since the function is positive in $\mathring{e}$, non-negative elsewhere and $\mathring{e} \subset S$. In this case, one of the two triangles that shares $e$ is denoted as $T_S$;
    \item if $S \in \Sdiscarded$, adapting the reasoning of \cite{Girault1995} to our setting, we can find a node $v_S \in \mathcal{V}_h$ such that the macro-element $\Delta_S$ consisting of the multiple triangles of $\mathcal{T}_h$ satisfies the following properties:
\begin{enumerate}[label={\textbf{P.\arabic*}}]
    \item\label{property:bolleGG1} $S$ intersects at least one interior segment $s$ of $\Delta_S$ in a point distant from $v_S$ less than the half of the length of $S$;
    \item\label{property:bolleGG2} the end points of $S$ do not belong to the interior of $\Delta_S$;
    \item\label{property:bolleGG3} the intersection between $\Delta_S$ and any other $\Delta_{S'}$ (for $S' \in \Sdiscarded$) or any $T \in \mathcal{T}^I_h \cup \mathcal{T}^e_h$ is empty or reduced to a node or an edge $e' \in \mathcal{E}_h \setminus \mathcal{E}_h^e$. 
\end{enumerate}
In this case, $\psi_S$ is chosen as the $\Poly{1}{}$-FEM global basis function associated with the node $v_S \in \mathcal{V}_h$. Using Properties~\ref{property:bolleGG1}-\ref{property:bolleGG3}, $\psi_S$ satisfies Equation~\eqref{eq:positiviàintegrale}, as in \cite{Girault1995}. In this case, one of the triangles that shares $v_S$ is denoted as $T_S$. An example of two such macro-elements, $\Delta_S$ and $\Delta_{S'}$, is shown in Figure~\ref{fig:GiraultGlowinski}. 
\end{itemize} 
The symbol $T_S$ introduced in this proof will be used in Lemma~\ref{Lemma:LowerBound}.
 Then, Equation~\eqref{eq:BolleSupportoNulloInAltriSegmenti} follows by construction. 
\end{proof}

\begin{remark}
\label{rem:sharp}
In complex engineering applications, the assumption on the sharpness of the angles of $\gamma$~\cite{Girault1995} may be restrictive. Our construction offers some flexibility in this regard: this hypothesis, can be removed by simply excluding the subset $\Sdiscarded$ from the coarsening strategy proposed here, without affecting the rest of the discretization procedure. Indeed, the sharpness assumption is used exclusively to guarantee that properties~\ref{property:bolleGG1}--\ref{property:bolleGG3} in the proof of Proposition~\ref{Proposition:proprietàfunzionibolla} hold: it plays no role elsewhere in the construction. This provides a practical trade-off: relaxing the angle condition allows the method to be applied to a broader class of domains, at the cost of a possibly coarser discretization of $\gamma$, depending on the tessellation $\mathcal{T}_h$.
\end{remark}

\begin{remark}
Other choices for the enrichment functions of the discrete spaces could have been considered instead of those introduced in Section~\ref{subsec:DiscreteSpaces}, such as the edge bubble functions proposed in \cite{berrone2006robust}. Unlike the edge bubble functions adopted in this work, whose support coincides with the union of two neighboring triangles, the edge bubble functions introduced in \cite{berrone2006robust} are supported only on portions of these triangles. Consequently, they could also be used to establish the counterpart of Proposition~\ref{Proposition:proprietàfunzionibolla}, even for a finer discretization $\mathcal{S}_\eta$. However, this choice would shift the main difficulty to the identification of the enrichment function associated with a given element $S$, which would significantly complicate the classification procedure described in Algorithm~\ref{alg:classification}.
\end{remark}

Finally, we are able to define the operator $ \Pi_h $ as:
\begin{equation}
\label{eq:Pi_h}
\Pi_h(v) := R_h(v) + \sum_{S \in \mathcal{S}_\eta} c_S \psi_S \quad \forall v \in X,
\end{equation}
where $ \psi_S$ is the function associated with $S$ as in the previous proof and each constant $ c_S $ is chosen so that
\begin{equation}
\label{eq:condizioneequiv}
\int_S \Pi_h(v)|_S  = \int_S v|_S \quad \forall S \in \mathcal{S}_\eta.
\end{equation}
Substituting Equation~\eqref{eq:Pi_h} into Equation~\eqref{eq:condizioneequiv} and using~\eqref{eq:BolleSupportoNulloInAltriSegmenti}, we obtain the following explicit expression for the coefficients $c_S$:
\begin{equation}
\label{eq:CostantiCs}
c_S = -\frac{1}{\int_S \psi_S|_S  \, } \int_S \left(R_h(v) - v\right)|_S  .
\end{equation}
This operation is well-defined since the denominator is non-zero by virtue of Equation~\eqref{eq:positiviàintegrale}. Now, in order to demonstrate Equations~\eqref{eq:FortinOperator1} and~\eqref{eq:FortinOperator2}, a preliminary result is showed.

\begin{figure}
    \centering
    \includegraphics[width=0.7\linewidth]{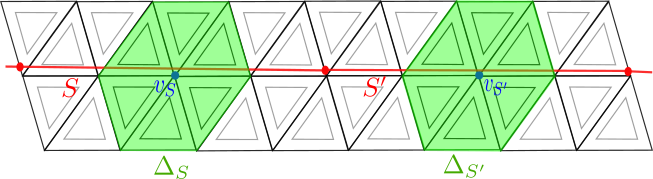}
    \caption{Example of two elements $S, S' \in \Sdiscarded$ (highlighted in red). For each element, $v_S$ and $v_{S'}$ (in blue) denote the associated vertex, while $\Delta_S$ and $\Delta_{S'}$ (in green) represent the corresponding patch of triangles surrounding them.}
    \label{fig:GiraultGlowinski}
\end{figure}

\begin{lemma}
\label{Lemma:rapportosegmenti}
    Since $ \left( \mathcal{T}_h \right)_{h>0}$ is shape-regular, for each segment $s$ lying inside a triangle $T \in \mathcal{T}_h$ or on its boundary, the ratio between its length and the length of its mapping $\hat{s}$ onto the reference triangle satisfies
    \begin{equation*}
        \frac{|s|}{|\hat{s}|} \geq C h_T.
    \end{equation*}
\end{lemma}
\begin{proof}
    If $s$ coincides with an edge of $T$, the proof is immediate. Let us consider a segment $s \subset T$ and its mapping $\hat{s}$. Without loss of generality, we may assume that $\hat{s} := \{ \hat{x}_0 + t\hat{v} : t \in [0,\ell] \}$, with $\hat{x}_0 \in \mathbb{R}^2$, $0<\ell = |\hat{s}| $ and $\hat{v}$ a unitary vector, so that $s = \{ F_T(\hat{x}_0) + t B_T \hat{v} : t \in [0,\ell] \}$, and
    \begin{equation*}
        \frac{|s|}{|\hat{s}|} = |B_T \hat{v}|.
    \end{equation*}
    Using classical linear algebra results, we have $|B_T \hat{v}| \geq \sigma_{\min} |\hat v|$, where $\sigma_{\min}$ denotes the smallest singular value of $B_T$, and $\sigma_{\min} = 1/\norm{B_T^{-1}}$. Thanks to~\cite[Theorem~3.1.3]{ciarlet1978finite} and mesh-regularity assumptions:
    \begin{equation*}
    \frac{|s|}{|\hat{s}|}  \geq \frac{1}{\norm{ B_T^{-1}} } \geq \frac {\rho_T}{h_{\hat{T}}} \geq C h_T,
\end{equation*}
where $h_{\hat{T}}$ is the longest edge of $\hat{T}$.
\end{proof}

Now we derive a lower bound for the denominator in Equation~\eqref{eq:CostantiCs}. This result will be used in Theorem~\ref{teo:FinaleTheorem}.

\begin{lemma}
\label{Lemma:LowerBound}
For each \( S \in \mathcal{S}_\eta \), there exists a constant \( C(a)\), depending only on the parameter \( a \) defined in Section \ref{sec:boundarydiscret}, such that
    \begin{equation*}
        \int_S\psi_S|_S  \geq C(a) h_{T_S},
    \end{equation*}
    where $\psi_S$ and $T_S$ are chosen as described in the proof of Proposition~\ref{Proposition:proprietàfunzionibolla}.
\end{lemma}
\begin{proof}
Let us consider an element $S \in \mathcal{S}_\eta$. Equation~\eqref{eq:Scontenutonelsottotriangolo} identifies three distinct configurations that must be addressed separately in the analysis:
\begin{itemize}
    \item if \( S \in \tilde\Sdiscarded \), we can adapt the proof of Lemma~4 in \cite{Girault1995}, since the use of a locally quasi-uniform mesh makes this case a straightforward extension of that result. 
    \item If exists $e \in \mathcal{E}^e_h$ such that $e \subset S$, as in the proof of Proposition~\ref{Proposition:proprietàfunzionibolla} $\psi_S$ is an \textit{edge bubble function}. Since $b_h^e$ is non-negative, it follows that
\begin{equation*}
    \int_S \psi_S|_S  \geq \int_e \psi_S|_e  = \frac{|e|}{|\hat{e}|}\int_{\hat{e}} \hat{\psi}_S|_{\hat{e}} \geq C \frac{|e|}{|\hat{e}|}.
\end{equation*}
Using Lemma~\ref{Lemma:rapportosegmenti}, the thesis is obtained.
\item The last case happens when we suppose that exists $T \in \mathcal{T}^I_h$ such that $S \cap \partial t_{T} \ne \emptyset$. This triangle is exactly $T_S$ from the proof of Proposition~\ref{Proposition:proprietàfunzionibolla}. Since the chosen \textit{internal bubble function} $\psi_S$ is null outside $T$, it suffices to consider the portion  $S_{T} \subseteq S$ that is entirely contained in $T$. Now, switching to the reference element, we obtain:
\begin{equation}
\label{eq:IntegraleClassicCaseRiferimento}
\int_{S_T} \psi_S|_{S_T}  = \frac{|S_T|}{|\hat{S}_T|}\int_{\hat{S}_T} \hat{\psi}_S|_{\hat{S}_T}:=  \frac{|S_T|}{|\hat{S}_T|} I_{\hat{S}}.
        \end{equation}
Depending on how and where the intersection between $S_T$ and the subtriangle $t_T$ occurs, the computation of the integral $I_{\hat{S}}$ will differ. In the \hyperref[appendix]{Appendix} we have proved that exists a constant $C(a)$ independent of $h$ and dependent on the parameter $a$ such that 
\begin{equation}
\label{eq:stimexappendix}
   I_{\hat{S}} \geq C(a).
\end{equation}
As $a$ is fixed, and since $S_T$ consists of a finite union of segments, using Lemma~\ref{Lemma:rapportosegmenti} we obtain the thesis.
\end{itemize}
\end{proof}

A quick inspection of the proof in the \hyperref[appendix]{Appendix} may lead the reader to observe that the parameter \( a \) introduced in Equation~\eqref{eq:TrasformScaling} plays a crucial role in ensuring stability. Values of $a$ close to zero lead to very small values of the constant $C(a)$, which may subsequently affect the stability results of the Theorem~\ref{teo:FinaleTheorem}. However, it should be noted that the parameter $a$ influences also the geometric discretization: choosing $a$ near its upper bound of $1/3$ can result in excessively coarse meshes for the boundary (see Section~\ref{sec:discretization}). The influence of the parameter $a$ on the stability of the method will be studied numerically in Section~\ref{sec:numericalresults}. For the remainder of this section, $a$ is considered fixed, and the dependence of the constant $C(a)$ on this parameter is omitted.

In the next theorem, we use will two results in \cite{Girault1995}, where the reader can find the proofs.
\begin{lemma}
\label{Lemma:StimaSuTriangRef}
    Let $\hat{T}$ denote the reference unit triangle and let $\hat{l}$ be any segment that intersects $\hat{T}$. Then, there exists a constant $\hat{C}$, independent of $\hat{l}$, such that
    \begin{equation*}
        \norm[\leb{2}{\hat{l}}]{ \hat{w}} \leq \hat{C} \norm[\sob{1}{\hat{T}}]{ \hat{w}} \quad \forall \hat{w} \in \sob{1}{\hat{T}}.
    \end{equation*}
\end{lemma}

\begin{lemma}
\label{Lemma:Confrontotralunghezzesegmenti}
    Let $l$ be a segment that intersects a non degenerate triangle $T$ and let $\hat{l}$ be its image on the reference unit triangle $\hat{T}$ by the affine transformation that maps $\hat{T}$ onto $T$. Let $B_T$ denote the matrix of this transformation and let $\norm{B_T}$ be its euclidean norm. Then,
\[
\frac{|l|}{|\hat{l}|} \leq \norm{B_T}.
\]
\end{lemma}

The assertions of the two preceding Lemmata remain valid even in the case 
where the endpoints of $l$ and $\hat{l}$ do not lie on the boundaries 
of $T$ and $\hat{T}$, respectively. Now, we are able to prove the main theorem.

\begin{theorem}
\label{teo:FinaleTheorem}
    Assume that $\eta \leq Ch$. There exists a constant $\beta^* >0$, independent of $h$ and $\eta$ such that the discrete inf-sup condition~\eqref{eq:inf-sup} holds true.
\end{theorem}

\begin{proof}
The discrete inf-sup condition is ensured by the construction of a Fortin operator, as stated in Lemma~\ref{Lemma:OperatorediFortin}. In particular, if there exists an operator $\Pi_h$ satisfying the Equations~\eqref{eq:FortinOperator1}–\eqref{eq:FortinOperator2}, then the thesis follows. The operator introduced in Equation~\eqref{eq:Pi_h} satisfies Equation~\eqref{eq:FortinOperator2} by construction, thanks to the specific choice of the constant $c_S$ and the fact that $M_\eta$ contains constant functions. It remains therefore to prove that Equation~\eqref{eq:FortinOperator1} holds. 

For any $v \in X$:
\begin{equation*}
\norm[X]{\Pi_h(v)} \leq \norm[X]{R_h(v)} + \norm[X]{\sum_{S \in \mathcal{S}_n} c_S \psi_S}.
\end{equation*}
We denote by $\Delta_S$ the support of $\psi_S$, which can be a single triangle, a couple of triangles sharing an edge, or a macro-element consisting of multiple triangles, depending on $S$. Since these supports may overlap, the above sum reduces to:
$$
\norm[X]{\sum_{S \in \mathcal{S}_n} c_S \psi_S} 
\;\le\; C_1 \left( \sum_{S \in \mathcal{S}_n} |c_S|^2 \norm[\sob{1}{\Delta_S}]{\psi_S}^2 \right)^{\frac{1}{2}},
$$
where the constant $C_1>0$ depends only on the maximum number of supports $\Delta_S$ overlapping at a single point, that remains bounded due to the mesh quality assumptions. Let $T$ be any triangle in $\Delta_S$ and $B_T$ be the matrix of the affine transformation that maps $\hat{T}$ onto $T$ as in Equation~\eqref{eq:affinetransf}. 
Using~\cite[Theorem~3.1.3]{ciarlet1978finite}, thanks to mesh regularity assumptions, there exists a constant $C$ independent of $h$ such that:
\begin{equation}
\label{eq:stimeconmeshregularity}
    \det(B_T) \simeq C h_T^2, \quad \norm{B_T^{-1}} \leq C h_T^{-1}, \quad \norm{B_T} \leq C h_T.
\end{equation}
These estimates allows us to obtain the following upper bounds:
\begin{equation*}
\norm[\leb{2}{T}]{\psi_S}  =|\det(B_T)|^{\frac{1}{2}} \norm[\leb{2}{\hat{T}}]{\hat{\psi}_S} \leq C h_T,
\end{equation*}
and
\begin{equation*}
\begin{aligned}
    \seminorm[\sob{1}{T}]{\psi_S} &= \scal[T]{\nabla \psi_S}{\nabla \psi_S}^{\frac{1}{2}} =|\det(B_T)|^{\frac{1}{2}} \scal[\hat{T}]{B_T^{-T}\nabla \hat{\psi}_S}{B_T^{-T}\nabla \hat{\psi}_S}^{\frac{1}{2}} \leq |\det(B_T)|^{\frac{1}{2}} \norm{ B_T^{-1}} \seminorm[\sob{1}{\hat{T}}]{\hat{\psi}_S} \leq C.
\end{aligned}
\end{equation*}
Hence, we obtain
\begin{equation}
\label{eq:ProofMainTeoremUpperBoundbS}
\norm[\sob{1}{\Delta_S}]{\psi_S} \leq C.
\end{equation}
Next, let us find a bound for $c_S$. For each $S \in \mathcal{S}_\eta$, let $\{\ell_i\}_{i \in \mathcal{I}_S}$ denote the collection of segments composing $S$, where $\mathcal{I}_S$ is the corresponding set of indices, so that $S = \bigcup_{i \in \mathcal{I}_S} \ell_i$. Let us call $T_i$ the element of $\mathcal{T}_h$ intersected by $\ell_i$. From Lemma~\ref{Lemma:LowerBound}, we have
\begin{equation*}
\begin{aligned}
    |c_S| &\leq \frac{C}{h_{T_S}} \sum_{i \in \mathcal{I}_S} \left| \int_{\ell_i} (R_h(v) - v)|_{\ell_i}  \right| \leq \frac{C}{h_{T_S}} \sum_{i \in \mathcal{I}_S}|\ell_i|^{\frac{1}{2}} \norm[\leb{2}{\ell_i}]{R_h(v) - v} \\
    &\leq \frac{C}{h_{T_S}} \sum_{i \in \mathcal{I}_S} |\ell_i|^{\frac{1}{2}} \left(\frac{|\ell_i|}{|\hat{\ell}_i|} \right)^{\frac{1}{2}} \norm[\leb{2}{\hat{\ell_i}}]{\hat{R}_h(v) - \hat{v}},
\end{aligned}
\end{equation*}
where we switched to the reference element. Applying Lemma~\ref{Lemma:StimaSuTriangRef}, Lemma~\ref{Lemma:Confrontotralunghezzesegmenti} and Properties~\eqref{eq:stimeconmeshregularity}, we obtain
\begin{equation}
\label{eq:ProofMainTheoremIntermediateCs}
\begin{aligned}
   |c_S| & \leq \frac{C}{h_{T_S}}\sum_{i \in \mathcal{I}_S} |\ell_i|^{\frac{1}{2}} \norm{B_{T_i}}^{\frac{1}{2}} \norm[\sob{1}{\hat{T}_i}]{\hat{R}_h(v) - \hat{v}} \\
   &\leq \frac{C}{h_{T_S}} \sum_{i \in \mathcal{I}_S} |\ell_i|^{\frac{1}{2}} \norm{B_{T_i}}^{\frac{1}{2}}|\det(B_{T_i})|^{-\frac{1}{2}} \left( \norm[\leb{2}{T_i}]{R_h(v) - v}^2 + \norm{B_{T_i}}^2 \seminorm[\sob{1}{T_i}]{R_h(v) - v}^2 \right)^{\frac{1}{2}} \\
   & \leq \frac{C}{h_{T_S}} \sum_{i \in \mathcal{I}_S} \left( \norm[\leb{2}{T_i}]{R_h(v) - v}^2 + h_{T_i}^2 \seminorm[\sob{1}{T_i}]{R_h(v) - v}^2 \right)^{\frac{1}{2}}. \\
\end{aligned}
\end{equation}
Combining~\eqref{eq:ProofMainTeoremUpperBoundbS} and~\eqref{eq:ProofMainTheoremIntermediateCs}, we obtain
\begin{equation*}
\left( \sum_{S \in \mathcal{S}_n} |c_S|^2 \norm[\sob{1}{\Delta_S}]{\psi_S}^2 \right)^{\frac{1}{2}}
\leq \left( \sum_{S \in \mathcal{S}_n} \frac{C}{h_{T_S}^2}   \sum_{i \in \mathcal{I}_S} \left( \norm[\leb{2}{T_i}]{R_h(v) - v}^2 + h_{T_i}^2 \seminorm[\sob{1}{T_i}]{R_h(v) - v}^2 \right) \right)^{\frac{1}{2}}.
\end{equation*}
From~\eqref{eq:estimates on the clement operator} and local quasi-uniformity, the thesis follows.
\end{proof}

The following result is obtained by following the methodology in \cite{Girault1995}.
\begin{theorem}
\label{theo:TeoremaConvergenzaH1}
    Under mesh assumptions on $\left( \mathcal{T}_h \right)_{h>0}$, construction of $\mathcal{S}_\eta$ and discrete spaces $X_h$ and $M_\eta$, Problem~\eqref{eq:DiscreteMixedProblem} has a unique solution $(u_h, \lambda_\eta)$ such that
    \begin{equation*}
        \norm[X]{\tilde{u}-u_h} + \norm[M]{\lambda - \lambda_\eta}  \leq C \left( \inf_{v_h \in X_h} \norm[X]{\tilde{u}-v_h} + \inf_{\mu_\eta \in M_\eta} \norm[M]{\lambda - \mu_\eta}  \right).
    \end{equation*}
    If $\tilde{u} \in \sob{\frac{3}{2}-\epsilon}{\Omega}$ and $\lambda \in \leb{2}{\gamma}$, with $\epsilon > 0$:
    \begin{equation*}
        \norm[X]{\tilde{u}-u_h} + \norm[M]{\lambda - \lambda_\eta}  \leq C \left( h^{\frac{1}{2} - \epsilon} \norm[\sob{\frac{3}{2}-\epsilon}{\Omega}]{\tilde{u}} + \sqrt{\eta} \norm[\leb{2}{\gamma}]{\lambda}\right);
    \end{equation*}
    if $\tilde{u} \in \sob{\frac{3}{2}-\epsilon}{\Omega}$ and $\lambda \in \sob{\frac{1}{2}}{\xi_i}$, with $\epsilon > 0$:
     \begin{equation*}
        \norm[X]{\tilde{u}-u_h} + \norm[M]{\lambda - \lambda_\eta}  \leq C \left( h^{\frac{1}{2} - \epsilon} \norm[\sob{\frac{3}{2}-\epsilon}{\Omega}]{\tilde{u}} + \eta \left( \sum_{i=1}^N \norm[\sob{\frac{1}{2}}{\xi_i}]{\lambda}^2 \right)^{\frac{1}{2}} \right);
    \end{equation*}
    if $\tilde{u} \in \sob{2}{\Omega}$ (then $\lambda=0$):
     \begin{equation*}
        \norm[X]{\tilde{u}-u_h} + \norm[M]{\lambda - \lambda_\eta}  \leq C h \norm[\sob{2}{\Omega}]{\tilde{u}}.
    \end{equation*}
\end{theorem}

\begin{remark}
\label{remark:smoothness_preserving}
Unless we choose the extension to $\Omega$ of the right-hand side $f$ in a compatible way with the Dirichlet boundary condition on $\gamma$, we will not in general have $\tilde{u} \in \sob{2}{\Omega}$, since the normal derivative will jump across $\gamma$. The maximum overall regularity that we can in general expect is $\tilde{u} \in \sob{t}{\Omega}$, for $t < 3/2$. Note that $\tilde{u}$ will have a higher regularity if and only if $\lambda = 0$. More sophisticated methods can be adopted to overcome this limitation, as the one presented in \cite{mommer2006smoothness}.
\end{remark}

\section{Numerical results}
\label{sec:numericalresults}

All numerical experiments presented in this work have been performed using the library \texttt{PolyDiM} (POLYtopal DIscretization Methods)~\cite{Polydim}. The analysis presented in previous sections is validated by the following experiment, inspired by the numerical test introduced in \cite{burman2010fictitious,burman2012fictitious}, modified through the cut-off function proposed in \cite{berrone2016adaptive,berrone2019optimal}. 
  The goal is to establish the convergence of the method with the expected rate, and, at the same time, to show that the lack of enrichment of the discrete space with bubble functions leads to a discrete problem for which the inf-sup stability condition is not uniform, resulting in numerical issues. For the continuous problem, the domain $D$ is defined as a circle of radius $\rho_0 = \frac{1}{2}$ centered at $ \left(x_c, y_c\right):=\left(0.1, 0\right) $. The corresponding governing equation in $D$ is formulated as Problem~\eqref{eq:modelproblem} with $\nu = 1$ and $\alpha = 0$, where $f, g$ are computed based on the exact solution expressed in polar coordinates
  \[
u(r, \phi) = \chi(r)\frac{1}{9} \left( \rho_0^3 - r^3 \right), \quad r \in \left[0, \rho_0\right],\phi \in \left[0,2\pi\right].
\]
The cut-off function $\chi(r)$ has the following expression:
\[
\chi(r) = \frac{w(3/4 - r)}{w(r - 1/4) + w(3/4 - r)}, \quad 
w(r) = \begin{cases} 
r^2 & \text{if } r > 0, \\ 
0 & \text{else,} 
\end{cases}
\quad r \in \left[0, \rho_0\right].
\]

Using the notation of Section~\ref{sec:modelproblem}, we consider the problem posed on the extended domain $\Omega := \left[-1, 1\right]^2$, and we denote by $\tilde{u}$ its exact solution, i.e. the extension of $u$ computed over the whole domain $\Omega$. The extension $\tilde{f}$ is computed based on $\tilde{u}$, and consequently, $\tilde{f}|_D = f$. Boundary conditions are Neumann (zero) condition over all $\partial \Omega$. A numerical solution is shown in Figure~\ref{fig:solutionnumerical}. Since $\tilde{u} \in \sob{2}{\Omega}$, using Theorem~\ref{theo:TeoremaConvergenzaH1}, we expected the same convergence rate of FEM employing first order polynomial on a conform shape-regular triangulation on $D$. Note that it is also possible to extend the function $f$ in an arbitrary way so as to obtain a different exact solution on $\Omega$, together with corresponding boundary conditions. See Remark~\ref{remark:smoothness_preserving} for a discussion about this situation. 

In the following, Section~\ref{subsec:convergencerates} is devoted to 
the analysis of the convergence rates for the solution of the numerical experiment described above, which we will be denoted in what follows as the test problem, along with a brief investigation 
of the influence of the parameter $a$, introduced in 
Section~\ref{sec:boundarydiscret}, on the numerical errors. Section~\ref{subsec:StabilityNumericalResults}, investigates the role of 
bubble functions: we show that solving the test problem with the 
enrichment framework proposed in this work guaranties numerical 
stability, whereas its absence leads to instabilities.
\begin{figure}
    \centering
    \includegraphics[width=0.4\linewidth]{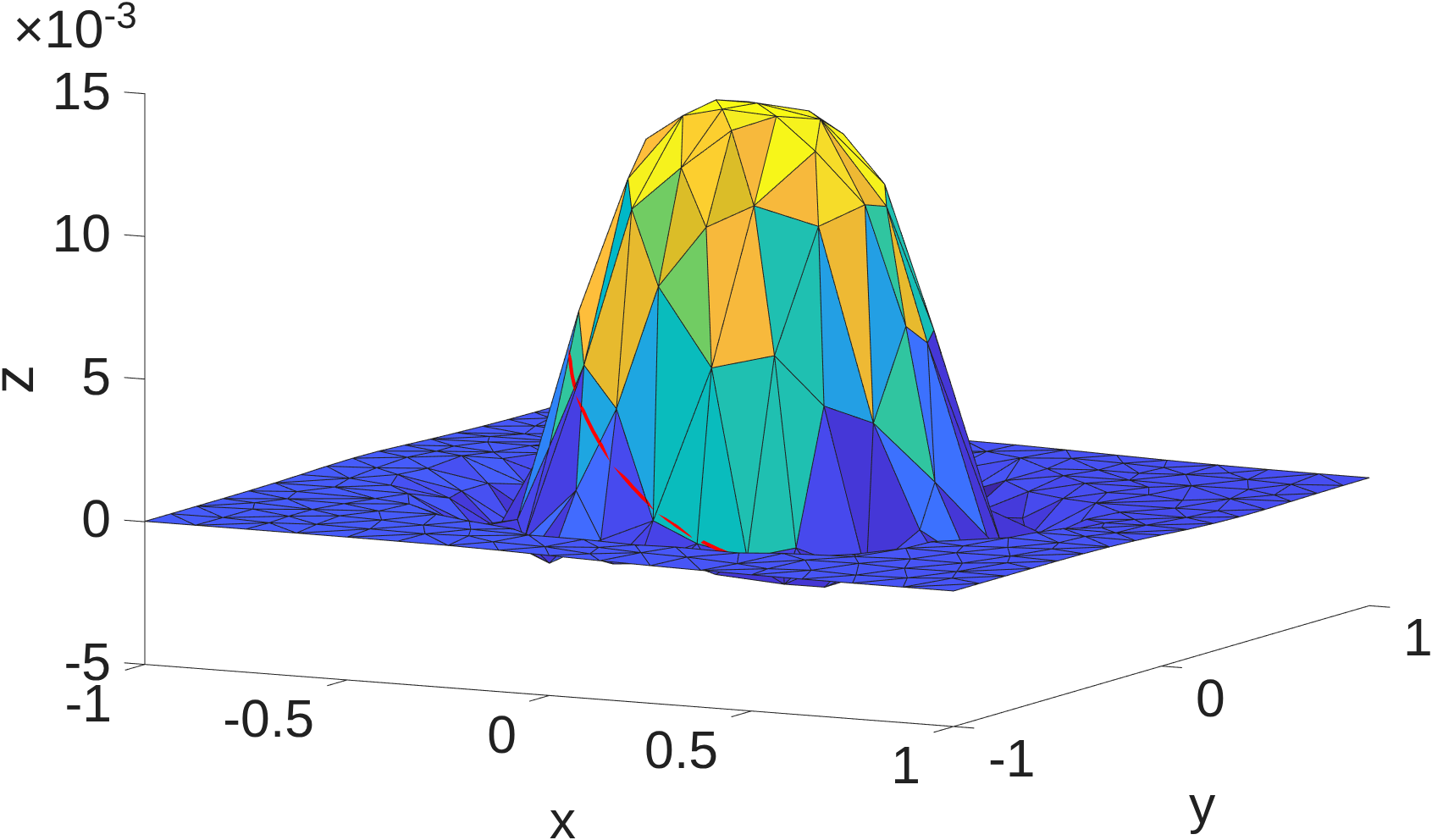}
    \hspace{1cm}
    \includegraphics[width=0.4\linewidth]{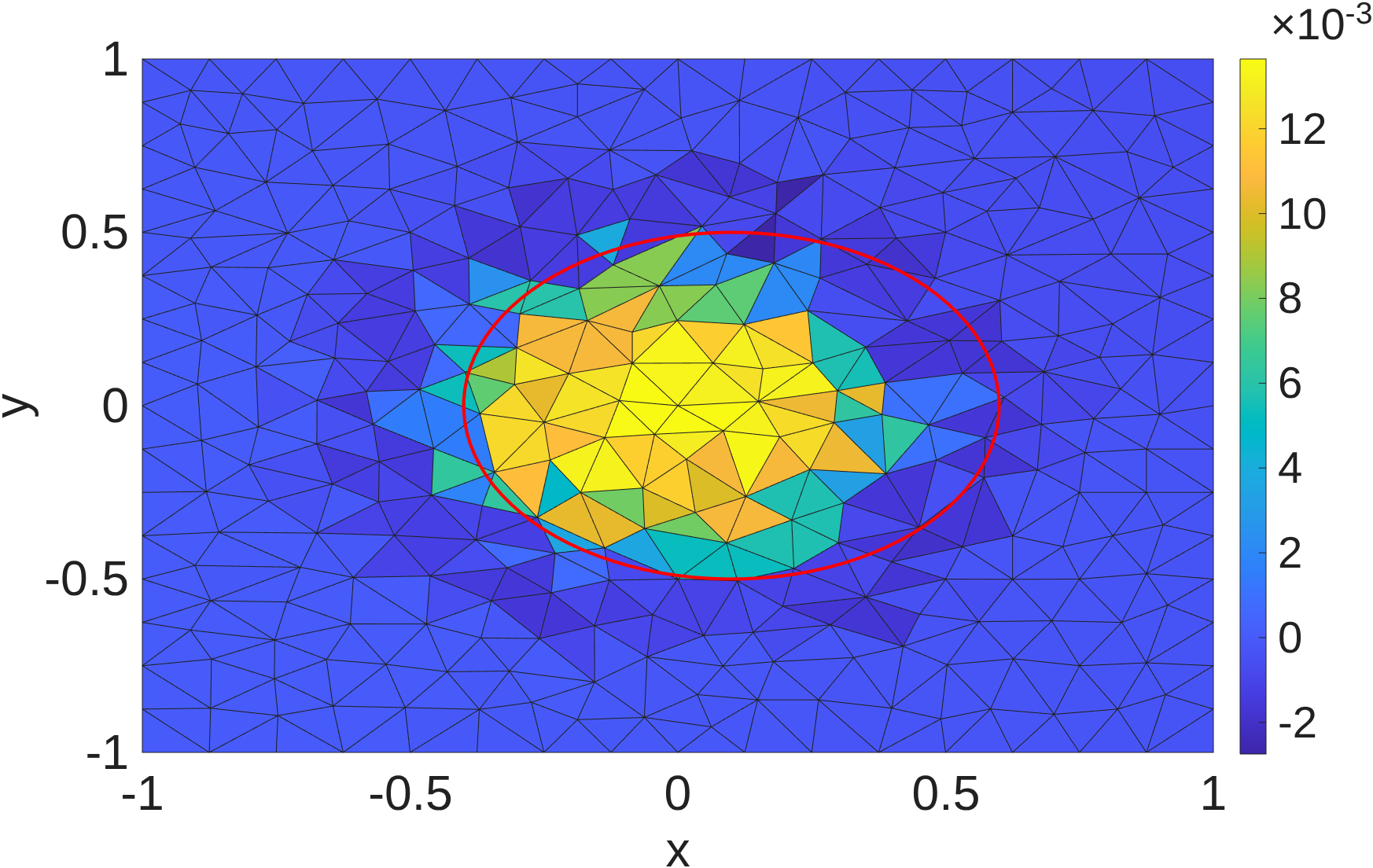}
    \caption{Numerical solution of the fictitious domain formulation~\eqref{eq:DiscreteMixedProblem} of the test problem, computed on a quasi-uniform shape-regular mesh with parameters $h=0.1$ and $a=0.1$. The left panel displays a three-dimensional surface plot of the discrete solution, while the right panel shows its top view over the domain $\Omega$. In both panels, the red circle denotes the boundary $\gamma$ before linearization.
}
    \label{fig:solutionnumerical}
\end{figure}

\subsection{Convergence rates}
\label{subsec:convergencerates}
Firstly, we solve the discrete problem using quasi-uniform shape-regular 
meshes with decreasing mesh size $h$. To assess the robustness of the 
method with respect to the parameter $a$, the experiments are performed for 
three representative values: a large value $a=0.25$, an intermediate 
value $a=0.1$, and a small value $a=0.0025$. The results are essentially 
identical for all three choices, confirming that the numerical 
performance of the method is robust with respect to this parameter. 
For each mesh, we compute the following standard error measures for 
$\tilde{u}$:
\begin{equation}
\label{eq:Errors}
    \mathrm{Err}_0 = \frac{\norm[\leb{2}{\Omega}]{\tilde{u} - u_h}}{\norm[\leb{2}{\Omega}]{\tilde{u}}},\qquad \mathrm{Err}_{\nabla} = \frac{\norm[\sob{1}{\Omega}]{\tilde{u} -  u_h}}{\norm[\sob{1}{\Omega}]{\tilde{u} }}.
\end{equation}
For the Lagrange multiplier $\lambda$, we approximate the 
$\sob{-\frac{1}{2}}{\gamma}$ error as
\begin{equation}
\label{eq:errorH-1/2}
    \mathrm{Err}_\gamma := \left( \sum_{S \in \mathcal{S}_\eta} |S| \norm[\leb{2}{S}]{\lambda-\lambda_\eta}^2 \right)^{\frac{1}{2}} \approx  \norm[\sob{-\frac{1}{2}}{\gamma}]{\lambda-\lambda_\eta}.
\end{equation}
\begin{figure}
    \centering
    \begin{subfigure}[b]{0.4\textwidth}
        \includegraphics[width=\linewidth]{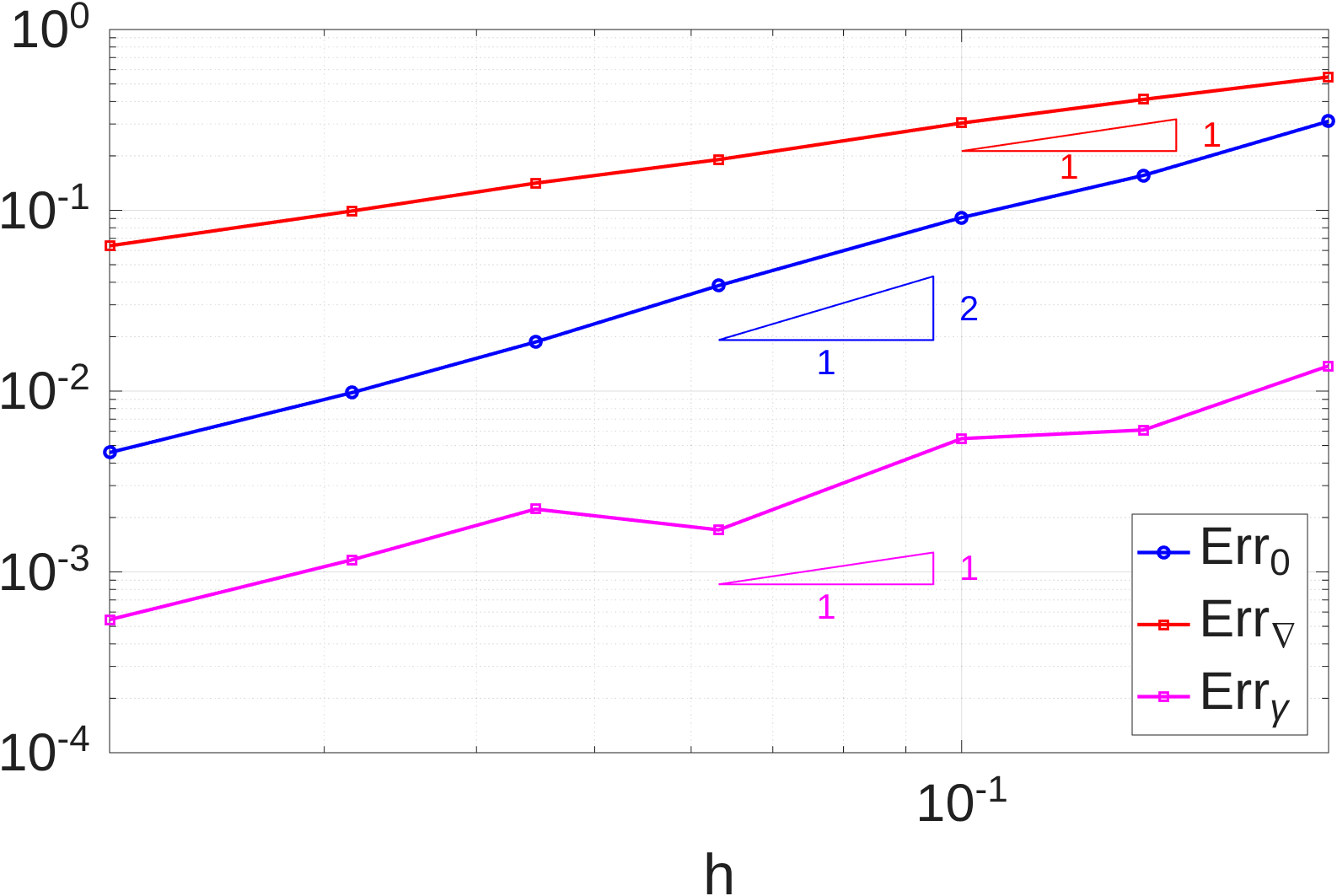}
        \caption{}
        \label{fig:convergence_curves_1}
    \end{subfigure}
    \hspace{1cm}
    \begin{subfigure}[b]{0.4\textwidth}
        \includegraphics[width=\linewidth]{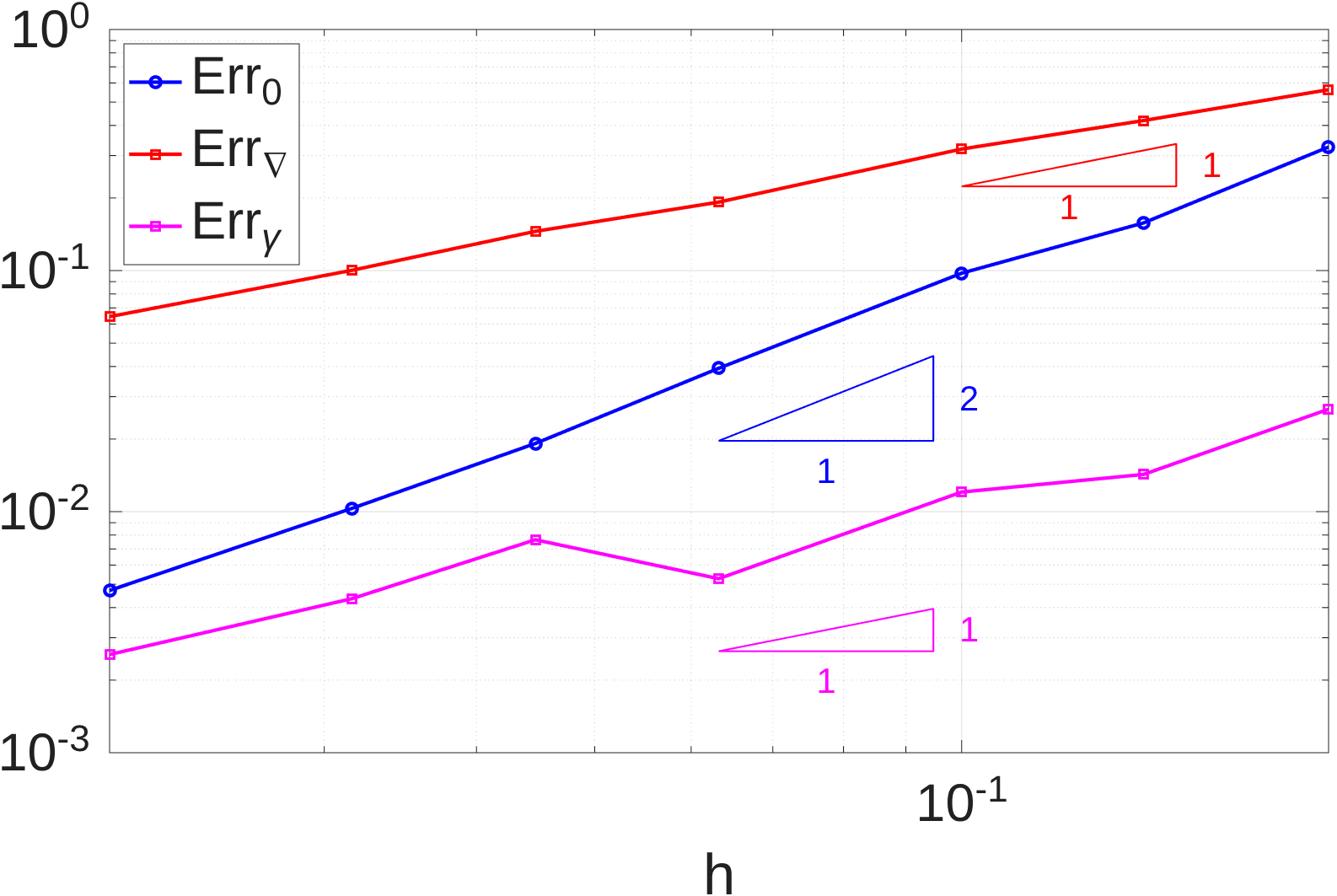}
        \caption{}
        \label{fig:convergence_curves_2}
    \end{subfigure}
    \hfill
    \begin{subfigure}[b]{0.4\textwidth}
        \includegraphics[width=\linewidth]{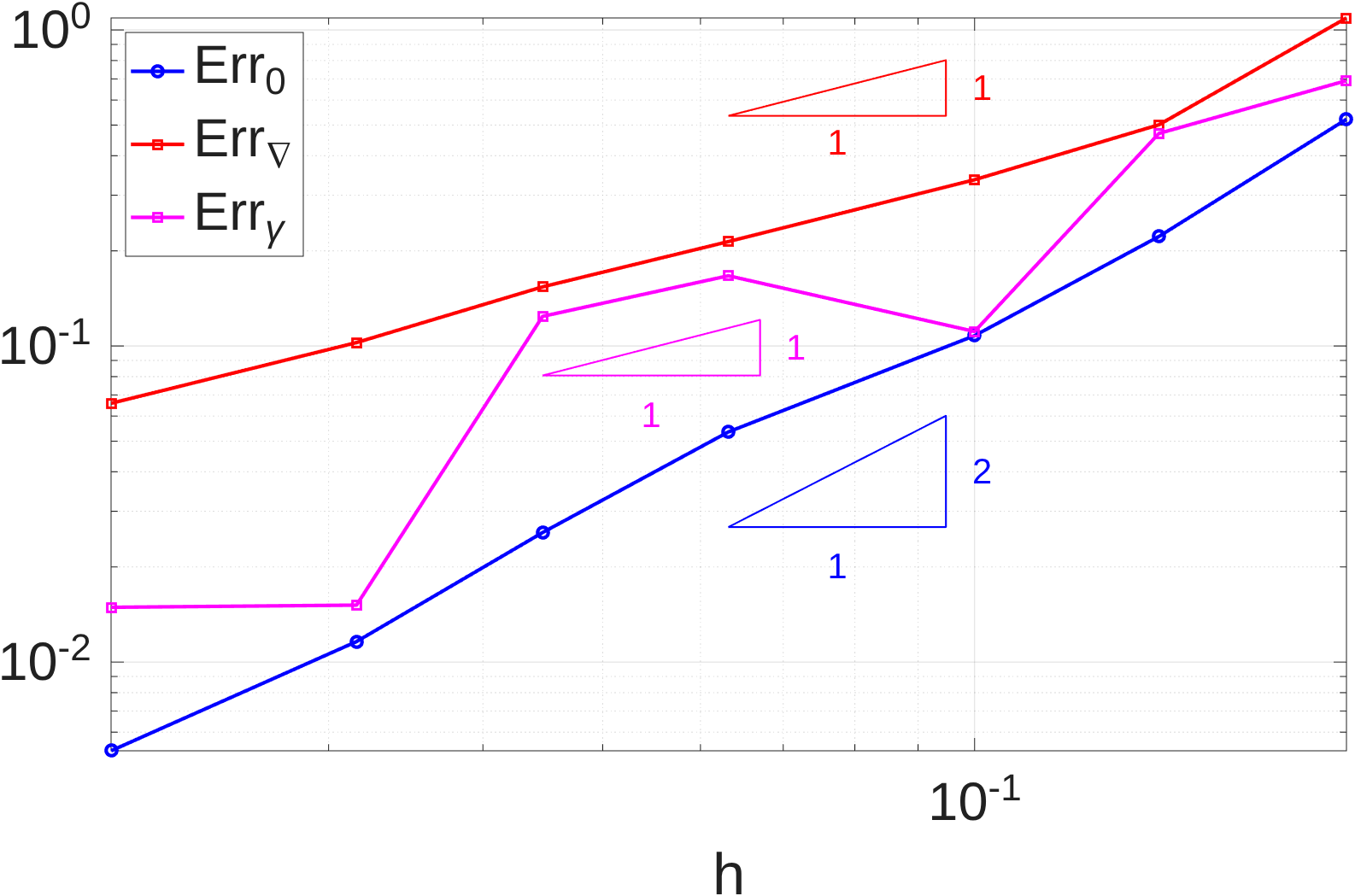}
        \caption{}
        \label{fig:convergence_curves_3}
    \end{subfigure}
    \caption{Behavior of the errors~\eqref{eq:Errors}--\eqref{eq:errorH-1/2} as functions of the mesh size $h$, computed using the fictitious domain formulation~\eqref{eq:DiscreteMixedProblem} of the test problem on quasi-uniform shape-regular meshes. Results are reported for three values of the parameter $a$: $a=0.25$ (a), $a=0.1$ (b), and $a=0.025$ (c).
}
    \label{fig:convergence_curves}
\end{figure}

As illustrated in Figure~\ref{fig:convergence_curves}, the numerical  errors converge with the rates established in  Theorem~\ref{theo:TeoremaConvergenzaH1}. Moreover, the errors associated with the solution $\tilde{u}$ are essentially unaffected by the variation of the parameter $a$. This is also confirmed by Figure~\ref{fig:ErroreH1VSa}, which shows that, across a range of meshes from the coarsest $M_1$ ($h = 0.2$) to the finest $M_4$ ($h = 0.02$), the $\mathrm{Err}_{0}$ error remains insensitive to $a$; the same behavior is observed for $\mathrm{Err}_{\nabla}$, which is therefore not reported. On the other hand, smaller values of $a$ slightly affect the error associated with the Lagrange multiplier, as can be observed by comparing the three $\mathrm{Err}_{\gamma}$ curves in Figure~\ref{fig:convergence_curves}: this is due to the fact that, for small values of $a$, the sub-triangles described in Section~\ref{sec:boundarydiscret} are larger, and elements of $\mathcal{S}_\eta$ of smaller length are accepted, leading to smaller values of the constant $C(a)$ appearing in Lemma~\ref{Lemma:LowerBound}, which, in turn, can pollute convergence.  As an example, 
Figure~\ref{fig:lagrange_multiplier}  shows the discrete Lagrange 
multiplier $\lambda_\eta$ computed for the three values of $a$: while 
for $a = 0.25$ the numerical solution remains close to the continuous 
solution $\lambda = 0$, decreasing values of $a$ introduce larger 
oscillations, particularly in the regions where elements of 
$\mathcal{S}_\eta$ of small length are present. Nevertheless, the 
convergence rate established in Theorem~\ref{theo:TeoremaConvergenzaH1} 
is preserved in all cases. Furthermore, numerical stability is always 
guaranteed, as discussed in the following subsection.

\begin{figure}
    \centering
    \includegraphics[width=0.5\linewidth]{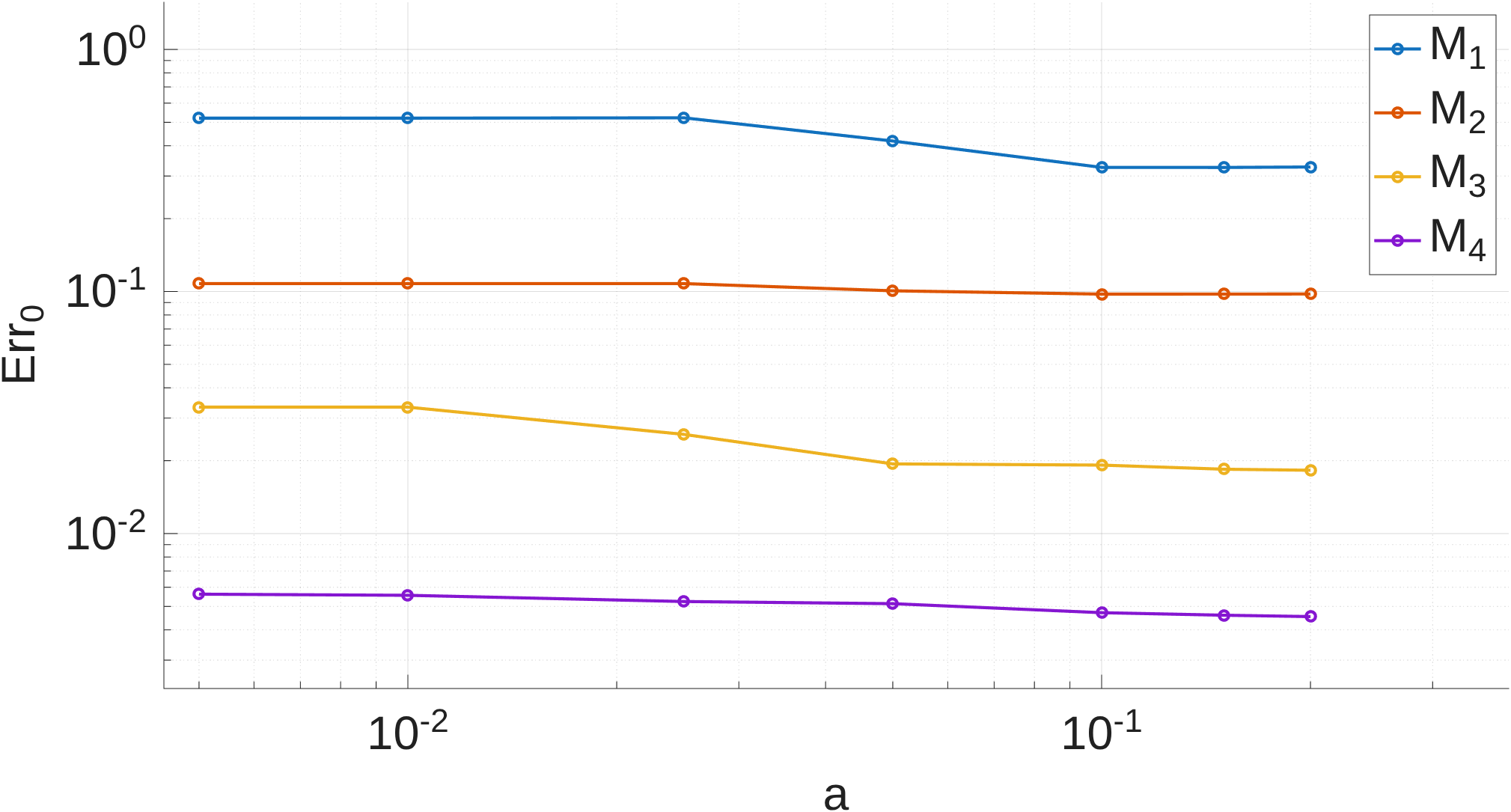}
    \caption{Behavior of the error $\mathrm{Err}_0$~\eqref{eq:Errors}, computed for the fictitious domain formulation~\eqref{eq:DiscreteMixedProblem} of the test problem, as the parameter $a$ decreases. Each curve corresponds to a different mesh size $h$ used to generate the triangulation $\mathcal{T}_h$. Here, $M_1,\dots,M_4$ denote quasi-uniform shape-regular triangulations ranging from the coarsest mesh $M_1$ ($h=0.2$) to the finest mesh $M_4$ ($h=0.02$).
}
    \label{fig:ErroreH1VSa}
\end{figure}

\begin{figure}
    \centering
    \begin{subfigure}[b]{0.4\textwidth}
        \includegraphics[width=\linewidth]{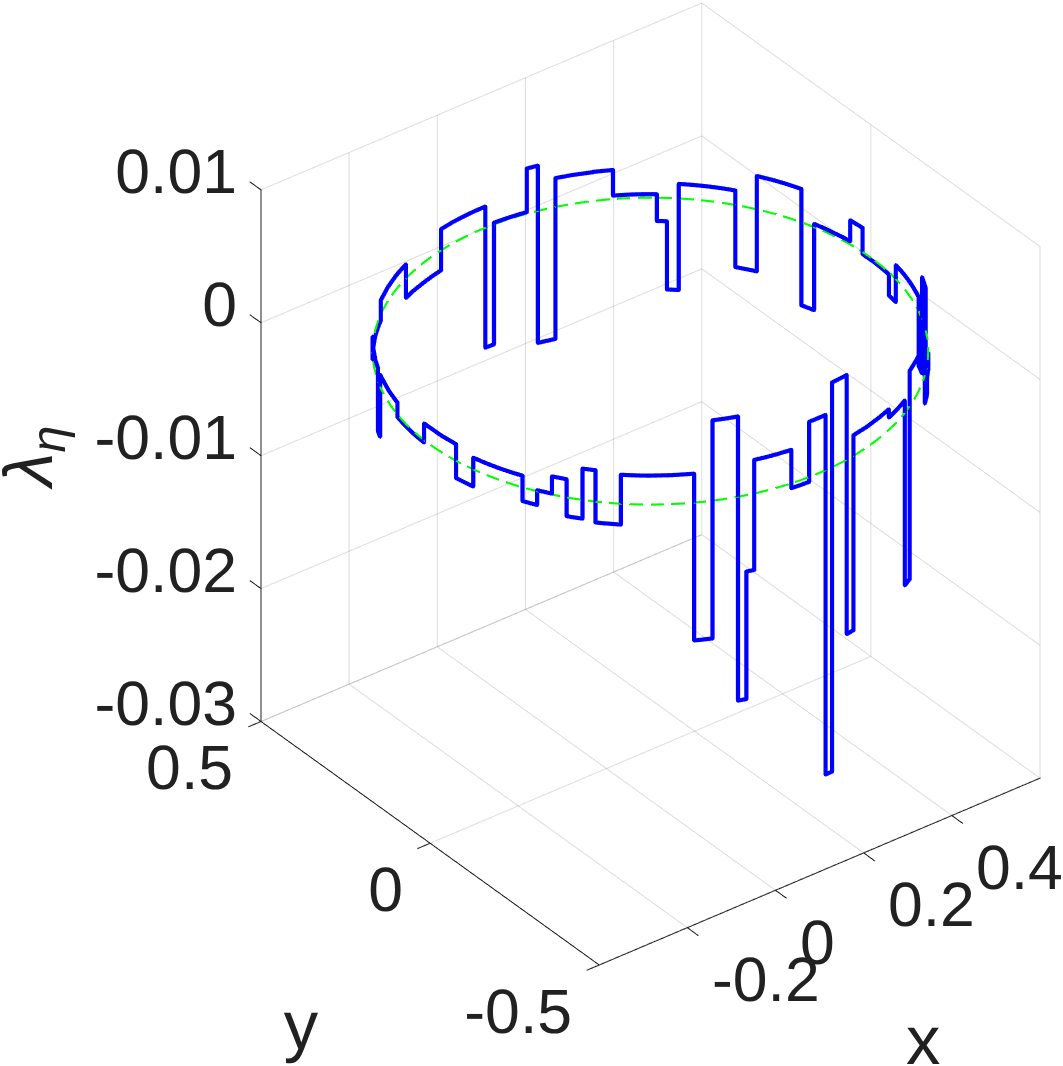}
        \caption{}
        \label{fig:lambda_1}
    \end{subfigure}
    \hspace{1cm}
    \begin{subfigure}[b]{0.4\textwidth}
        \includegraphics[width=\linewidth]{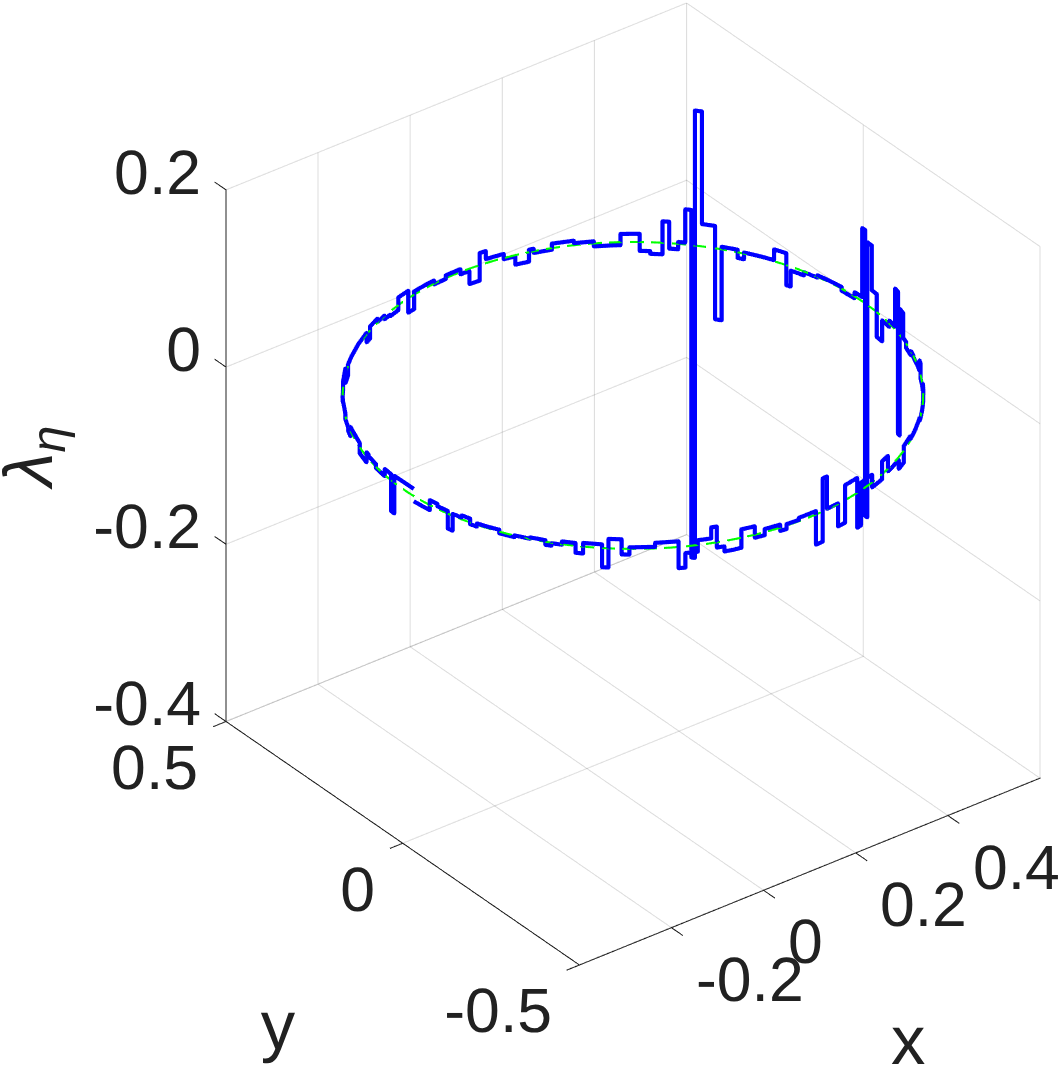}
        \caption{}
        \label{fig:lambda_2}
    \end{subfigure}
    \hfill
    \begin{subfigure}[b]{0.4\textwidth}
        \includegraphics[width=\linewidth]{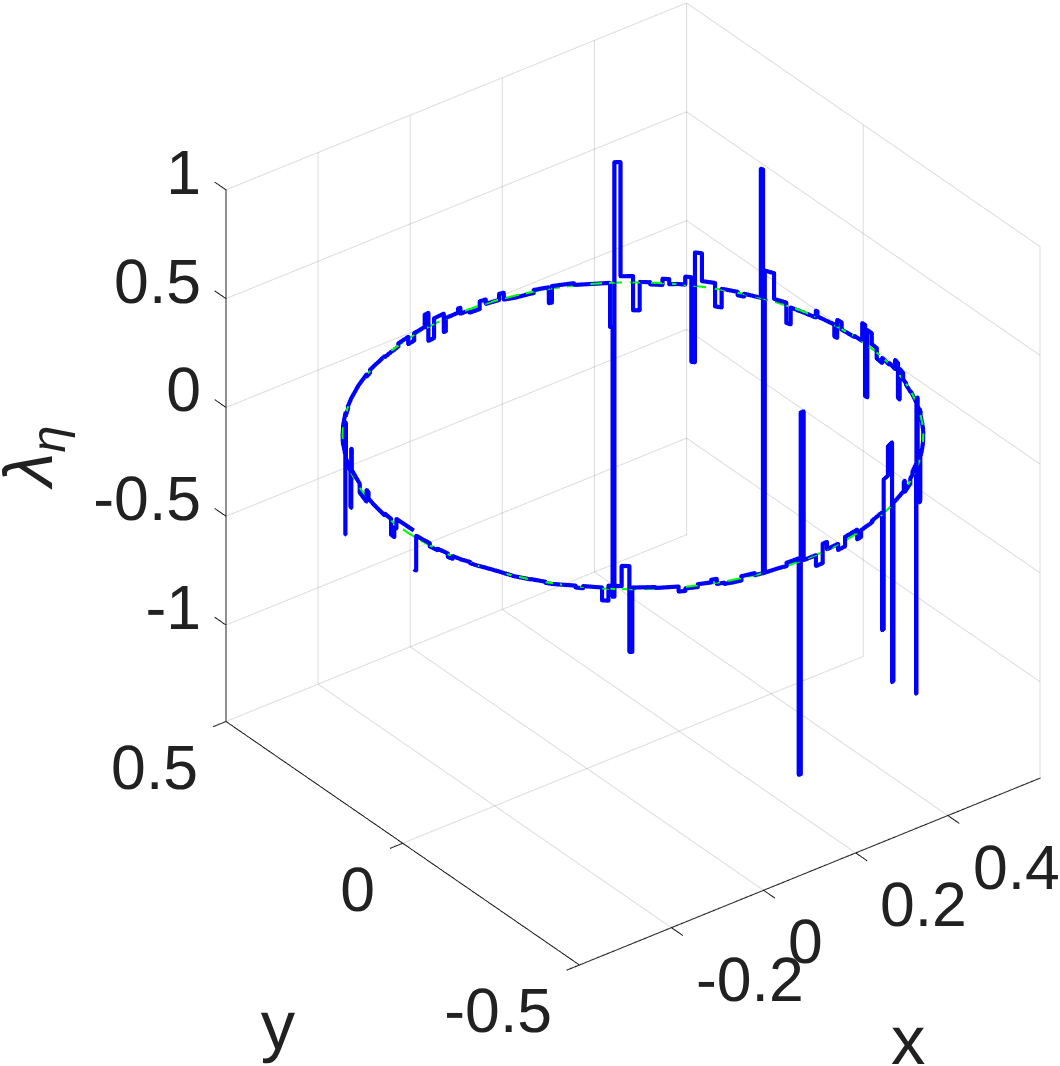}
        \caption{}
        \label{fig:lambda_3}
    \end{subfigure}
    \caption{Behavior of $\lambda_\eta$, computed for the fictitious domain formulation~\eqref{eq:DiscreteMixedProblem} of the test problem on a quasi-uniform, shape-regular mesh with $h=0.0316$ ($M_3$ of the meshes of Figure~\ref{fig:ErroreH1VSa}) for three values of the parameter $a$: $a=0.25$ (a), $a=0.1$ (b), and $a=0.025$ (c). The thin green dashed line denotes the continuous solution $\lambda = 0$.
}
    \label{fig:lagrange_multiplier}
\end{figure}

\subsection{Stability}
\label{subsec:StabilityNumericalResults}
Here we investigate the sensitivity of the system with respect to the variation of the parameter $a$ introduced in Section~\ref{sec:boundarydiscret}, highlighting how the absence of bubble functions, i.e. the lack of the enrichment framework proposed in this work, leads to a severe loss of numerical stability. Figure~\ref{fig:condition_curves} reports the condition number of the saddle-point system matrix as a function of the parameter $a$ for a sequence of meshes with fixed mesh sizes, ranging from the coarsest mesh $M_1$ ($h=0.2$) to the finest mesh $M_4$ ($h=0.02$); for each mesh $\mathcal{T}_h$ in this sequence, the corresponding boundary discretization $\mathcal{S}_\eta$ is constructed using the algorithm described in Section~\ref{sec:discretization}. The results clearly show that the inclusion of bubble functions significantly improves the conditioning of the resulting linear system, reflecting the enhanced stability properties of the proposed enriched formulation. Conversely, when bubble functions are not employed, the condition number deteriorates substantially, indicating a marked loss of stability and, in some cases, leading to a completely unstable discrete problem.

\begin{figure}
    \centering
    \includegraphics[width=0.4\linewidth]{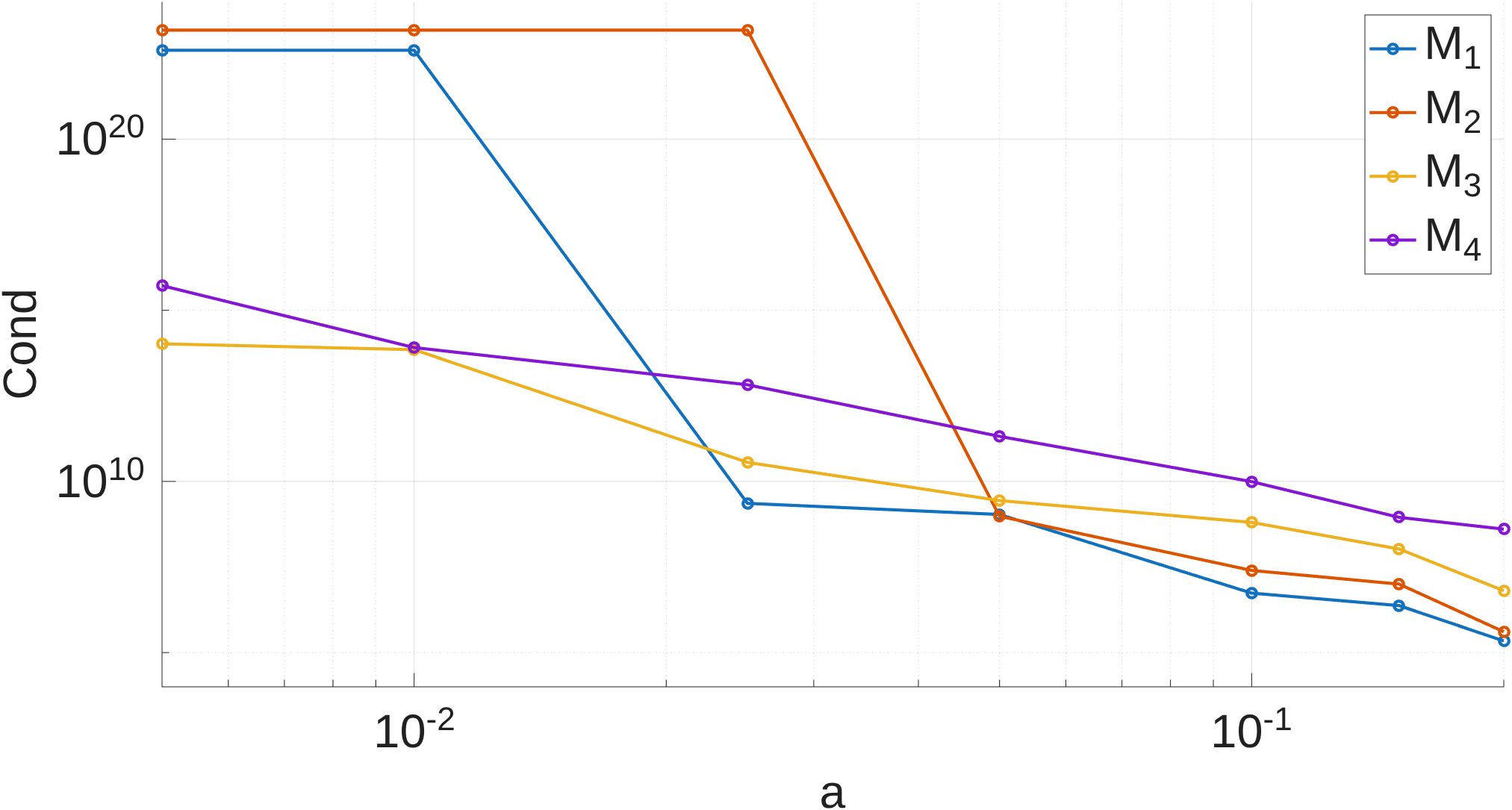}
    \hspace{1cm}\includegraphics[width=0.4\linewidth]{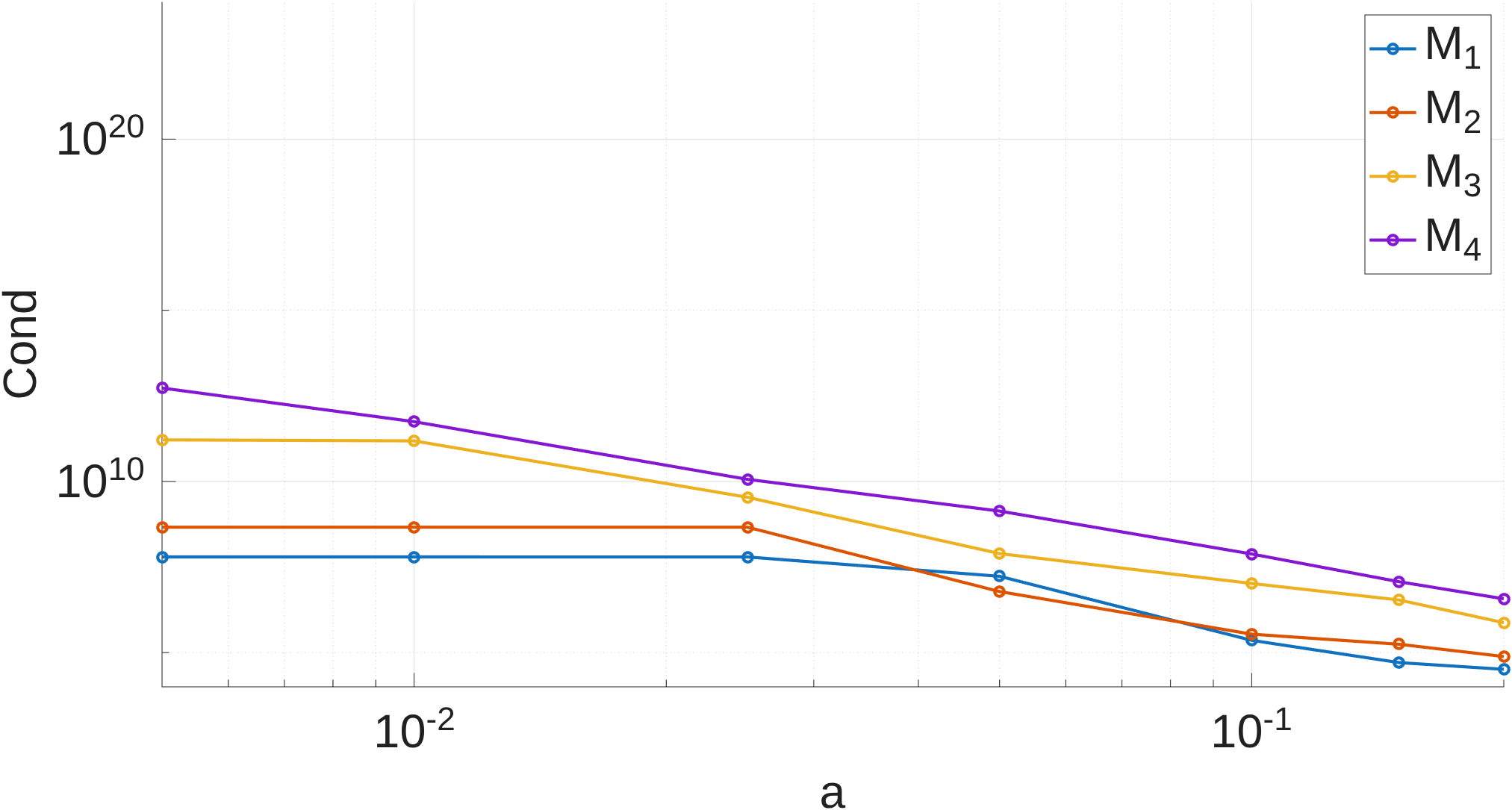}
    \caption{Behavior of the condition number of the matrix associated with the saddle-point problem~\eqref{eq:DiscreteMixedProblem} as $a$ decreases, computed for the test problem. Each curve corresponds to a different value of the parameter $h$ used to generate $\mathcal{T}_h$. Here $M_1, \dots, M_4$ denote quasi-uniform shape-regular triangulations ranging from the coarsest mesh $M_1$ ($h=0.2$) to the finest $M_4$ ($h=0.02$), and the corresponding boundary discretization $\mathcal{S}_\eta$ is constructed with the algorithm described in Section~\ref{sec:discretization}. The same meshes are used in both panels. Left: results obtained without enriching the $\Poly{1}{}$ finite element space with bubble functions. Right: results obtained using the proposed bubble-enriched finite element space.}
 
    \label{fig:condition_curves}
\end{figure}

\section{Conclusions}
\label{sec:conclusion}
In this work, we have presented and analyzed a stabilized fictitious domain method for the numerical approximation of elliptic problems with Dirichlet boundary conditions. The proposed approach successfully addresses and relaxes the restrictive mesh requirements traditionally associated with Lagrange multiplier-based unfitted methods. Specifically, by moving beyond the constraints established in \cite{Girault1995}, we have demonstrated that the method remains robust and well-posed on locally quasi-uniform background triangulations. Moreover, the proposed construction eliminates the need for two distinct meshes governed by two independent parameters. Indeed, the boundary mesh is derived from and controlled by the background mesh on $\Omega$. The numerical experiments confirm that the proposed method achieves the expected convergence rates. They also demonstrate that, in order to define a boundary discretization strategy entirely governed by the bulk mesh on $\Omega$, the enrichment of the classical $\mathbb{P}_1$ discrete space with bubble functions is essential to reach stability.

\section*{Acknowledgments}

The author S.B. kindly acknowledges partial financial support provided by European Union through project Next Generation EU, M4C2, PRIN 2022 PNRR project P2022BH5CB\_001 ``Polyhedral Galerkin methods for engineering applications to improve disaster risk forecast and management: stabilization-free operator-preserving methods and optimal stabilization methods'', and by PNRR M4C2 project of CN00000013 National Centre for HPC, Big Data and Quantum Computing (HPC) (CUP: E13C22000990001). The author L.N. acknowledges the financial support provided by INdAM-GNCS Project ``Metodi numerici politopali stabilization-free e neural-based per problemi accoppiati e non lineari'' (CUP: E53C25002010001). The author F.V. acknowledges the financial support by INdAM-research group GNCS, project title: ``Metodi numerici integrati per la simulazione e la prevenzione del dissesto idrogeologico'' (CUP: E53C25002010001).

S.B., L.N., S.S, and F.V. are members of the GNCS-INdAM Group.

\appendix 
\section*{Appendix}
\label{appendix}
We now present the proof of~\eqref{eq:stimexappendix}.  It is completely equivalent to work in the $(\hat{x}, \hat{y})$ coordinate system, since the action of the affine transformation does not affect the way in which $S$ intersects the subtriangle, being linear. The notation used in this proof is outlined in Figure~\ref{fig:NomenclTriangRef}. For the sake of simplicity, we restrict our focus to the case where element $S \in \mathcal{S}_\eta$ is entirely contained in $T$ (i.e. $S_T \equiv S$) and consists of at most two  segments. An exhaustive overview of all possible cases is showed in Figure~\ref{fig:TabellaGraficaTuttiCasi}. The treatment of alternative cases follows an identical procedure. 
To obtain a lower bound, we employ the properties of the bubble function: $\hat{\psi}_S|_{t_{\hat{T}}}$ has minimum values on the subtriangle vertices $D,E,F$, it is strictly concave in the reference triangle $\hat{T}$ and it is non-negative. Therefore, a lower bound for the integral is provided when $\hat{S}$ is  composed by two segments and lies totally in $\hat{T} \setminus\mathring{t}_{\hat{T}}$. In particular, if we compute explicitly the integral only over a segment $\overline{S} \subset \hat{S}$, we obtain a lower bound for $I_{\hat{S}} $, i.e.  $I_{\hat{S}} \geq I_{\overline{S}}$.  All the possible configurations of the segment $\overline{S}$ within $\hat{T}$ must be considered, as in Figure~\ref{fig:BrokenLineCasistiche}. In the following, the integrals $I_{\hat{S}}$ are evaluated as functions of the slope of $\hat{S}$, since $a$ is fixed.

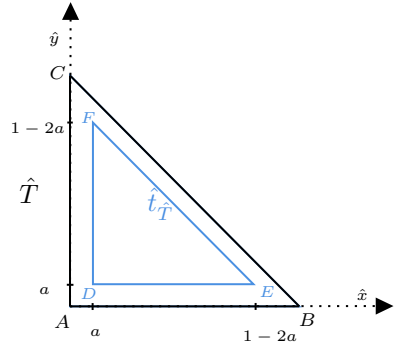
\begin{figure}
\centering

\tikzset{every picture/.style={line width=0.75pt}} %set default line width to 0.75pt        

\begin{tikzpicture}[x=0.75pt,y=0.75pt,yscale=-1,xscale=1]
%uncomment if require: \path (0,190); %set diagram left start at 0, and has height of 190

%Shape: Right Triangle [id:dp4347346177064999] 
\draw  [color={rgb, 255:red, 74; green, 144; blue, 226 }  ,draw opacity=1 ] (282.71,43.96) -- (397.79,159.84) -- (282.71,159.84) -- cycle ;
%Shape: Right Triangle [id:dp5720803968309729] 
\draw  [color={rgb, 255:red, 74; green, 144; blue, 226 }  ,draw opacity=1 ] (294.22,67.6) -- (374.77,148.72) -- (294.22,148.72) -- cycle ;

%Straight Lines [id:da44398161487153587] 
\draw  [dash pattern={on 0.84pt off 2.51pt}]  (282.71,159.84) -- (442.07,159.68) ;
\draw [shift={(445.07,159.68)}, rotate = 179.94] [fill={rgb, 255:red, 0; green, 0; blue, 0 }  ][line width=0.08]  [draw opacity=0] (8.93,-4.29) -- (0,0) -- (8.93,4.29) -- cycle    ;
%Straight Lines [id:da7496283815907522] 
\draw  [dash pattern={on 0.84pt off 2.51pt}]  (282.71,159.84) -- (283.01,10.14) ;
\draw [shift={(283.01,7.14)}, rotate = 90.11] [fill={rgb, 255:red, 0; green, 0; blue, 0 }  ][line width=0.08]  [draw opacity=0] (8.93,-4.29) -- (0,0) -- (8.93,4.29) -- cycle    ;
%Straight Lines [id:da7383668684042818] 
\draw    (294.06,161.67) -- (294.06,158.18) ;
%Straight Lines [id:da13745325329856095] 
\draw    (375.86,161.63) -- (375.86,158.13) ;
%Straight Lines [id:da08763780467367199] 
\draw    (281.1,149) -- (284.44,149.03) ;
%Straight Lines [id:da8373804852354975] 
\draw    (284.11,67.59) -- (281.35,67.58) ;
%Straight Lines [id:da11525121190340859] 
\draw    (282.71,43.96) -- (397.79,159.84) ;
%Straight Lines [id:da6574490464265177] 
\draw    (282.71,43.96) -- (282.71,159.84) ;
%Straight Lines [id:da86869886490759] 
\draw    (282.71,159.84) -- (397.79,159.84) ;

% Text Node
\draw (273.22,162.79) node [anchor=north west][inner sep=0.75pt]  [font=\scriptsize]  {$A$};
% Text Node
\draw (396.09,161.93) node [anchor=north west][inner sep=0.75pt]  [font=\scriptsize]  {$B$};
% Text Node
\draw (270.89,37.7) node [anchor=north west][inner sep=0.75pt]  [font=\scriptsize]  {$C$};
% Text Node
\draw (376.31,148.09) node [anchor=north west][inner sep=0.75pt]  [font=\tiny,color={rgb, 255:red, 74; green, 144; blue, 226 }  ,opacity=1 ]  {$E$};
% Text Node
\draw (286.5,149.41) node [anchor=north west][inner sep=0.75pt]  [font=\tiny,color={rgb, 255:red, 74; green, 144; blue, 226 }  ,opacity=1 ]  {$D$};
% Text Node
\draw (286.68,60.83) node [anchor=north west][inner sep=0.75pt]  [font=\tiny,color={rgb, 255:red, 74; green, 144; blue, 226 }  ,opacity=1 ]  {$F$};
% Text Node
\draw (255.83,93.85) node [anchor=north west][inner sep=0.75pt]  [font=\normalsize]  {$\hat{T}$};
% Text Node
\draw (320.09,98.86) node [anchor=north west][inner sep=0.75pt]  [font=\normalsize]  {$\textcolor[rgb]{0.29,0.56,0.89}{\hat{t}}\textcolor[rgb]{0.29,0.56,0.89}{_{\hat{T}}}$};
% Text Node
\draw (291.25,169.76) node [anchor=north west][inner sep=0.75pt]  [font=\tiny]  {$a$};
% Text Node
\draw (265.72,148.79) node [anchor=north west][inner sep=0.75pt]  [font=\tiny]  {$a$};
% Text Node
\draw (251.28,65.93) node [anchor=north west][inner sep=0.75pt]  [font=\tiny]  {$1-2a$};
% Text Node
\draw (425.11,149.47) node [anchor=north west][inner sep=0.75pt]  [font=\tiny]  {$\hat{x}$};
% Text Node
\draw (270.63,20.54) node [anchor=north west][inner sep=0.75pt]  [font=\tiny]  {$\hat{y}$};
% Text Node
\draw (367.88,170.33) node [anchor=north west][inner sep=0.75pt]  [font=\tiny]  {$1-2a$};

\end{tikzpicture}
    \caption{Reference triangle and subtriangle constructed on the reference triangle. The geometrical features associated with the subtriangle are highlighted in blue.}
    \label{fig:NomenclTriangRef}
\end{figure}

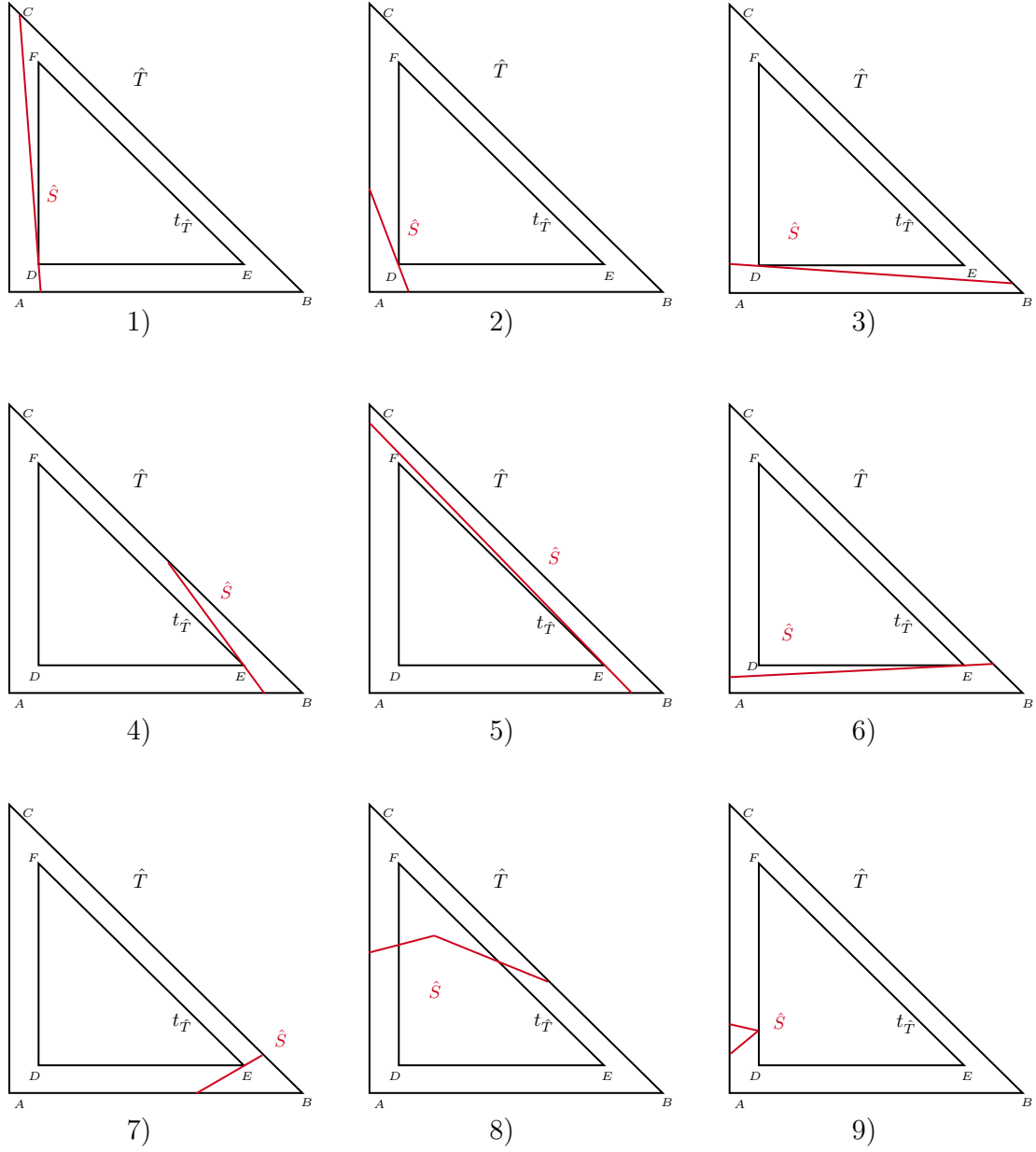
\begin{figure}
    \centering

\tikzset{every picture/.style={line width=0.75pt}} %set default line width to 0.75pt        

\begin{tikzpicture}[x=0.75pt,y=0.75pt,yscale=-1,xscale=1]
%uncomment if require: \path (0,641); %set diagram left start at 0, and has height of 641

%Shape: Right Triangle [id:dp7338714417700285] 
\draw   (0,0.6) -- (157.62,155.61) -- (0,155.61) -- cycle ;
%Shape: Right Triangle [id:dp5322793412315244] 
\draw   (15.76,32.22) -- (126.09,140.73) -- (15.76,140.73) -- cycle ;

%Shape: Right Triangle [id:dp6203593297437374] 
\draw   (387.32,1.17) -- (544.93,156.18) -- (387.32,156.18) -- cycle ;
%Shape: Right Triangle [id:dp03394821564212247] 
\draw   (403.08,32.79) -- (513.41,141.3) -- (403.08,141.3) -- cycle ;

%Shape: Right Triangle [id:dp8540990408939849] 
\draw   (193.7,216.21) -- (351.32,371.22) -- (193.7,371.22) -- cycle ;
%Shape: Right Triangle [id:dp6094673162798714] 
\draw   (209.47,247.83) -- (319.8,356.34) -- (209.47,356.34) -- cycle ;

%Shape: Right Triangle [id:dp9798253277085782] 
\draw   (0,216.21) -- (157.62,371.22) -- (0,371.22) -- cycle ;
%Shape: Right Triangle [id:dp7562579511018779] 
\draw   (15.76,247.83) -- (126.09,356.34) -- (15.76,356.34) -- cycle ;

%Shape: Right Triangle [id:dp2555530582568707] 
\draw   (387.32,216.21) -- (544.93,371.22) -- (387.32,371.22) -- cycle ;
%Shape: Right Triangle [id:dp8695425719542175] 
\draw   (403.08,247.83) -- (513.41,356.34) -- (403.08,356.34) -- cycle ;

%Shape: Right Triangle [id:dp80426319578562] 
\draw   (0,431.25) -- (157.62,586.26) -- (0,586.26) -- cycle ;
%Shape: Right Triangle [id:dp900329347038771] 
\draw   (15.76,462.87) -- (126.09,571.38) -- (15.76,571.38) -- cycle ;

%Shape: Right Triangle [id:dp5693156129666075] 
\draw   (193.7,431.25) -- (351.32,586.26) -- (193.7,586.26) -- cycle ;
%Shape: Right Triangle [id:dp050293723400847035] 
\draw   (209.47,462.87) -- (319.8,571.38) -- (209.47,571.38) -- cycle ;

%Shape: Right Triangle [id:dp8052434576040859] 
\draw   (387.32,431.25) -- (544.93,586.26) -- (387.32,586.26) -- cycle ;
%Shape: Right Triangle [id:dp4260927266831346] 
\draw   (403.08,462.87) -- (513.41,571.38) -- (403.08,571.38) -- cycle ;

%Shape: Right Triangle [id:dp1536652845811466] 
\draw   (193.7,0.6) -- (351.32,155.61) -- (193.7,155.61) -- cycle ;
%Shape: Right Triangle [id:dp3678078969319655] 
\draw   (209.47,32.22) -- (319.8,140.73) -- (209.47,140.73) -- cycle ;

%Straight Lines [id:da6402969920092466] 
\draw [color={rgb, 255:red, 208; green, 2; blue, 27 }  ,draw opacity=1 ]   (5.6,6.54) -- (16.93,155.88) ;
%Straight Lines [id:da36357316986277777] 
\draw [color={rgb, 255:red, 208; green, 2; blue, 27 }  ,draw opacity=1 ]   (193.6,99.96) -- (214.8,155.56) ;
%Straight Lines [id:da2816938207603994] 
\draw [color={rgb, 255:red, 208; green, 2; blue, 27 }  ,draw opacity=1 ]   (387.24,140.56) -- (539.24,150.96) ;
%Straight Lines [id:da6754382285218408] 
\draw [color={rgb, 255:red, 208; green, 2; blue, 27 }  ,draw opacity=1 ]   (137.24,371.52) -- (85.44,301.1) ;
%Straight Lines [id:da38679050420713745] 
\draw [color={rgb, 255:red, 208; green, 2; blue, 27 }  ,draw opacity=1 ]   (193.89,226.2) -- (334.74,371.23) ;
%Straight Lines [id:da9715495159525285] 
\draw [color={rgb, 255:red, 208; green, 2; blue, 27 }  ,draw opacity=1 ]   (387.31,362.71) -- (529.6,355.57) ;
%Straight Lines [id:da8958943486323538] 
\draw [color={rgb, 255:red, 208; green, 2; blue, 27 }  ,draw opacity=1 ]   (100.44,586.56) -- (136.04,565.96) ;
%Straight Lines [id:da7801634265719809] 
\draw [color={rgb, 255:red, 208; green, 2; blue, 27 }  ,draw opacity=1 ]   (193.4,510.8) -- (228.53,501.68) ;
%Straight Lines [id:da9510554589715262] 
\draw [color={rgb, 255:red, 208; green, 2; blue, 27 }  ,draw opacity=1 ]   (387,549.2) -- (402.8,552.8) ;
%Straight Lines [id:da9184091539683914] 
\draw [color={rgb, 255:red, 208; green, 2; blue, 27 }  ,draw opacity=1 ]   (387.8,565.2) -- (402.8,552.8) ;
%Straight Lines [id:da09068003176455208] 
\draw [color={rgb, 255:red, 208; green, 2; blue, 27 }  ,draw opacity=1 ]   (228.53,501.68) -- (289.73,526.48) ;

% Text Node
\draw (65.47,33.76) node [anchor=north west][inner sep=0.75pt]  [font=\normalsize,xscale=0.8,yscale=0.8] [align=left] {$\displaystyle \hat{T}$};
% Text Node
\draw (87.24,112.25) node [anchor=north west][inner sep=0.75pt]  [font=\normalsize,xscale=0.8,yscale=0.8]  {$t_{\hat{T}}$};
% Text Node
\draw (1.2,158.1) node [anchor=north west][inner sep=0.75pt]  [font=\tiny,xscale=0.8,yscale=0.8]  {$A$};
% Text Node
\draw (155.65,157.18) node [anchor=north west][inner sep=0.75pt]  [font=\tiny,xscale=0.8,yscale=0.8]  {$B$};
% Text Node
\draw (5.93,1.66) node [anchor=north west][inner sep=0.75pt]  [font=\tiny,xscale=0.8,yscale=0.8]  {$C$};
% Text Node
\draw (7.5,143.06) node [anchor=north west][inner sep=0.75pt]  [font=\tiny,xscale=0.8,yscale=0.8]  {$D$};
% Text Node
\draw (123.59,143.35) node [anchor=north west][inner sep=0.75pt]  [font=\tiny,xscale=0.8,yscale=0.8]  {$E$};
% Text Node
\draw (8.85,25.88) node [anchor=north west][inner sep=0.75pt]  [font=\tiny,xscale=0.8,yscale=0.8]  {$F$};
% Text Node
\draw (259.17,29.17) node [anchor=north west][inner sep=0.75pt]  [font=\normalsize,xscale=0.8,yscale=0.8] [align=left] {$\displaystyle \hat{T}$};
% Text Node
\draw (280.1,111.81) node [anchor=north west][inner sep=0.75pt]  [font=\normalsize,xscale=0.8,yscale=0.8]  {$t_{\hat{T}}$};
% Text Node
\draw (194.9,158.1) node [anchor=north west][inner sep=0.75pt]  [font=\tiny,xscale=0.8,yscale=0.8]  {$A$};
% Text Node
\draw (349.35,157.75) node [anchor=north west][inner sep=0.75pt]  [font=\tiny,xscale=0.8,yscale=0.8]  {$B$};
% Text Node
\draw (199.64,2.23) node [anchor=north west][inner sep=0.75pt]  [font=\tiny,xscale=0.8,yscale=0.8]  {$C$};
% Text Node
\draw (200.91,144.26) node [anchor=north west][inner sep=0.75pt]  [font=\tiny,xscale=0.8,yscale=0.8]  {$D$};
% Text Node
\draw (318.26,143.91) node [anchor=north west][inner sep=0.75pt]  [font=\tiny,xscale=0.8,yscale=0.8]  {$E$};
% Text Node
\draw (202.55,26.45) node [anchor=north west][inner sep=0.75pt]  [font=\tiny,xscale=0.8,yscale=0.8]  {$F$};
% Text Node
\draw (452.78,34.33) node [anchor=north west][inner sep=0.75pt]  [font=\normalsize,xscale=0.8,yscale=0.8] [align=left] {$\displaystyle \hat{T}$};
% Text Node
\draw (475.23,112.15) node [anchor=north west][inner sep=0.75pt]  [font=\normalsize,xscale=0.8,yscale=0.8]  {$t_{\hat{T}}$};
% Text Node
\draw (388.51,158.67) node [anchor=north west][inner sep=0.75pt]  [font=\tiny,xscale=0.8,yscale=0.8]  {$A$};
% Text Node
\draw (542.97,157.75) node [anchor=north west][inner sep=0.75pt]  [font=\tiny,xscale=0.8,yscale=0.8]  {$B$};
% Text Node
\draw (393.25,2.23) node [anchor=north west][inner sep=0.75pt]  [font=\tiny,xscale=0.8,yscale=0.8]  {$C$};
% Text Node
\draw (396.53,143.92) node [anchor=north west][inner sep=0.75pt]  [font=\tiny,xscale=0.8,yscale=0.8]  {$D$};
% Text Node
\draw (513.31,142.32) node [anchor=north west][inner sep=0.75pt]  [font=\tiny,xscale=0.8,yscale=0.8]  {$E$};
% Text Node
\draw (396.16,26.45) node [anchor=north west][inner sep=0.75pt]  [font=\tiny,xscale=0.8,yscale=0.8]  {$F$};
% Text Node
\draw (259.17,249.37) node [anchor=north west][inner sep=0.75pt]  [font=\normalsize,xscale=0.8,yscale=0.8] [align=left] {$\displaystyle \hat{T}$};
% Text Node
\draw (282.28,328.52) node [anchor=north west][inner sep=0.75pt]  [font=\normalsize,xscale=0.8,yscale=0.8]  {$t_{\hat{T}}$};
% Text Node
\draw (194.9,373.71) node [anchor=north west][inner sep=0.75pt]  [font=\tiny,xscale=0.8,yscale=0.8]  {$A$};
% Text Node
\draw (349.35,372.79) node [anchor=north west][inner sep=0.75pt]  [font=\tiny,xscale=0.8,yscale=0.8]  {$B$};
% Text Node
\draw (199.64,217.27) node [anchor=north west][inner sep=0.75pt]  [font=\tiny,xscale=0.8,yscale=0.8]  {$C$};
% Text Node
\draw (202.91,358.96) node [anchor=north west][inner sep=0.75pt]  [font=\tiny,xscale=0.8,yscale=0.8]  {$D$};
% Text Node
\draw (312.44,358.96) node [anchor=north west][inner sep=0.75pt]  [font=\tiny,xscale=0.8,yscale=0.8]  {$E$};
% Text Node
\draw (202.55,241.49) node [anchor=north west][inner sep=0.75pt]  [font=\tiny,xscale=0.8,yscale=0.8]  {$F$};
% Text Node
\draw (65.47,249.37) node [anchor=north west][inner sep=0.75pt]  [font=\normalsize,xscale=0.8,yscale=0.8] [align=left] {$\displaystyle \hat{T}$};
% Text Node
\draw (86.58,327.19) node [anchor=north west][inner sep=0.75pt]  [font=\normalsize,xscale=0.8,yscale=0.8]  {$t_{\hat{T}}$};
% Text Node
\draw (1.2,373.71) node [anchor=north west][inner sep=0.75pt]  [font=\tiny,xscale=0.8,yscale=0.8]  {$A$};
% Text Node
\draw (155.65,372.79) node [anchor=north west][inner sep=0.75pt]  [font=\tiny,xscale=0.8,yscale=0.8]  {$B$};
% Text Node
\draw (5.93,217.27) node [anchor=north west][inner sep=0.75pt]  [font=\tiny,xscale=0.8,yscale=0.8]  {$C$};
% Text Node
\draw (9.21,358.96) node [anchor=north west][inner sep=0.75pt]  [font=\tiny,xscale=0.8,yscale=0.8]  {$D$};
% Text Node
\draw (119.99,358.96) node [anchor=north west][inner sep=0.75pt]  [font=\tiny,xscale=0.8,yscale=0.8]  {$E$};
% Text Node
\draw (8.85,241.49) node [anchor=north west][inner sep=0.75pt]  [font=\tiny,xscale=0.8,yscale=0.8]  {$F$};
% Text Node
\draw (452.78,249.37) node [anchor=north west][inner sep=0.75pt]  [font=\normalsize,xscale=0.8,yscale=0.8] [align=left] {$\displaystyle \hat{T}$};
% Text Node
\draw (474.56,327.86) node [anchor=north west][inner sep=0.75pt]  [font=\normalsize,xscale=0.8,yscale=0.8]  {$t_{\hat{T}}$};
% Text Node
\draw (388.51,373.71) node [anchor=north west][inner sep=0.75pt]  [font=\tiny,xscale=0.8,yscale=0.8]  {$A$};
% Text Node
\draw (543.4,373.28) node [anchor=north west][inner sep=0.75pt]  [font=\tiny,xscale=0.8,yscale=0.8]  {$B$};
% Text Node
\draw (393.25,217.27) node [anchor=north west][inner sep=0.75pt]  [font=\tiny,xscale=0.8,yscale=0.8]  {$C$};
% Text Node
\draw (395.1,353.67) node [anchor=north west][inner sep=0.75pt]  [font=\tiny,xscale=0.8,yscale=0.8]  {$D$};
% Text Node
\draw (510.91,358.96) node [anchor=north west][inner sep=0.75pt]  [font=\tiny,xscale=0.8,yscale=0.8]  {$E$};
% Text Node
\draw (396.16,241.49) node [anchor=north west][inner sep=0.75pt]  [font=\tiny,xscale=0.8,yscale=0.8]  {$F$};
% Text Node
\draw (65.47,464.41) node [anchor=north west][inner sep=0.75pt]  [font=\normalsize,xscale=0.8,yscale=0.8] [align=left] {$\displaystyle \hat{T}$};
% Text Node
\draw (86.58,542.23) node [anchor=north west][inner sep=0.75pt]  [font=\normalsize,xscale=0.8,yscale=0.8]  {$t_{\hat{T}}$};
% Text Node
\draw (1.2,588.76) node [anchor=north west][inner sep=0.75pt]  [font=\tiny,xscale=0.8,yscale=0.8]  {$A$};
% Text Node
\draw (155.65,587.83) node [anchor=north west][inner sep=0.75pt]  [font=\tiny,xscale=0.8,yscale=0.8]  {$B$};
% Text Node
\draw (5.93,432.31) node [anchor=north west][inner sep=0.75pt]  [font=\tiny,xscale=0.8,yscale=0.8]  {$C$};
% Text Node
\draw (9.21,574) node [anchor=north west][inner sep=0.75pt]  [font=\tiny,xscale=0.8,yscale=0.8]  {$D$};
% Text Node
\draw (123.59,574) node [anchor=north west][inner sep=0.75pt]  [font=\tiny,xscale=0.8,yscale=0.8]  {$E$};
% Text Node
\draw (8.85,456.54) node [anchor=north west][inner sep=0.75pt]  [font=\tiny,xscale=0.8,yscale=0.8]  {$F$};
% Text Node
\draw (259.17,464.41) node [anchor=north west][inner sep=0.75pt]  [font=\normalsize,xscale=0.8,yscale=0.8] [align=left] {$\displaystyle \hat{T}$};
% Text Node
\draw (280.95,542.23) node [anchor=north west][inner sep=0.75pt]  [font=\normalsize,xscale=0.8,yscale=0.8]  {$t_{\hat{T}}$};
% Text Node
\draw (194.9,588.76) node [anchor=north west][inner sep=0.75pt]  [font=\tiny,xscale=0.8,yscale=0.8]  {$A$};
% Text Node
\draw (349.35,587.83) node [anchor=north west][inner sep=0.75pt]  [font=\tiny,xscale=0.8,yscale=0.8]  {$B$};
% Text Node
\draw (199.64,432.31) node [anchor=north west][inner sep=0.75pt]  [font=\tiny,xscale=0.8,yscale=0.8]  {$C$};
% Text Node
\draw (202.91,574) node [anchor=north west][inner sep=0.75pt]  [font=\tiny,xscale=0.8,yscale=0.8]  {$D$};
% Text Node
\draw (317.3,574) node [anchor=north west][inner sep=0.75pt]  [font=\tiny,xscale=0.8,yscale=0.8]  {$E$};
% Text Node
\draw (202.55,456.54) node [anchor=north west][inner sep=0.75pt]  [font=\tiny,xscale=0.8,yscale=0.8]  {$F$};
% Text Node
\draw (452.78,464.41) node [anchor=north west][inner sep=0.75pt]  [font=\normalsize,xscale=0.8,yscale=0.8] [align=left] {$\displaystyle \hat{T}$};
% Text Node
\draw (475.9,542.9) node [anchor=north west][inner sep=0.75pt]  [font=\normalsize,xscale=0.8,yscale=0.8]  {$t_{\hat{T}}$};
% Text Node
\draw (388.51,588.76) node [anchor=north west][inner sep=0.75pt]  [font=\tiny,xscale=0.8,yscale=0.8]  {$A$};
% Text Node
\draw (542.97,587.83) node [anchor=north west][inner sep=0.75pt]  [font=\tiny,xscale=0.8,yscale=0.8]  {$B$};
% Text Node
\draw (393.25,432.31) node [anchor=north west][inner sep=0.75pt]  [font=\tiny,xscale=0.8,yscale=0.8]  {$C$};
% Text Node
\draw (396.53,574) node [anchor=north west][inner sep=0.75pt]  [font=\tiny,xscale=0.8,yscale=0.8]  {$D$};
% Text Node
\draw (510.91,574) node [anchor=north west][inner sep=0.75pt]  [font=\tiny,xscale=0.8,yscale=0.8]  {$E$};
% Text Node
\draw (396.16,456.54) node [anchor=north west][inner sep=0.75pt]  [font=\tiny,xscale=0.8,yscale=0.8]  {$F$};
% Text Node
\draw (18.86,96.65) node [anchor=north west][inner sep=0.75pt]  [font=\small,color={rgb, 255:red, 208; green, 2; blue, 27 }  ,opacity=1 ,xscale=0.8,yscale=0.8]  {$\hat{S}$};
% Text Node
\draw (212.96,115.05) node [anchor=north west][inner sep=0.75pt]  [font=\small,color={rgb, 255:red, 208; green, 2; blue, 27 }  ,opacity=1 ,xscale=0.8,yscale=0.8]  {$\hat{S}$};
% Text Node
\draw (417.24,116.65) node [anchor=north west][inner sep=0.75pt]  [font=\small,color={rgb, 255:red, 208; green, 2; blue, 27 }  ,opacity=1 ,xscale=0.8,yscale=0.8]  {$\hat{S}$};
% Text Node
\draw (62,163.4) node [anchor=north west][inner sep=0.75pt]  [font=\Large,xscale=0.8,yscale=0.8]  {$1)$};
% Text Node
\draw (257,163.4) node [anchor=north west][inner sep=0.75pt]  [font=\Large,xscale=0.8,yscale=0.8]  {$2)$};
% Text Node
\draw (452,163.4) node [anchor=north west][inner sep=0.75pt]  [font=\Large,xscale=0.8,yscale=0.8]  {$3)$};
% Text Node
\draw (62,383.4) node [anchor=north west][inner sep=0.75pt]  [font=\Large,xscale=0.8,yscale=0.8]  {$4)$};
% Text Node
\draw (257,383.4) node [anchor=north west][inner sep=0.75pt]  [font=\Large,xscale=0.8,yscale=0.8]  {$5)$};
% Text Node
\draw (452,383.4) node [anchor=north west][inner sep=0.75pt]  [font=\Large,xscale=0.8,yscale=0.8]  {$6)$};
% Text Node
\draw (62,598.4) node [anchor=north west][inner sep=0.75pt]  [font=\Large,xscale=0.8,yscale=0.8]  {$7)$};
% Text Node
\draw (257,598.4) node [anchor=north west][inner sep=0.75pt]  [font=\Large,xscale=0.8,yscale=0.8]  {$8)$};
% Text Node
\draw (452,598.4) node [anchor=north west][inner sep=0.75pt]  [font=\Large,xscale=0.8,yscale=0.8]  {$9)$};
% Text Node
\draw (112.22,310.1) node [anchor=north west][inner sep=0.75pt]  [font=\small,color={rgb, 255:red, 208; green, 2; blue, 27 }  ,opacity=1 ,xscale=0.8,yscale=0.8]  {$\hat{S}$};
% Text Node
\draw (288.57,291.65) node [anchor=north west][inner sep=0.75pt]  [font=\small,color={rgb, 255:red, 208; green, 2; blue, 27 }  ,opacity=1 ,xscale=0.8,yscale=0.8]  {$\hat{S}$};
% Text Node
\draw (414,332.65) node [anchor=north west][inner sep=0.75pt]  [font=\small,color={rgb, 255:red, 208; green, 2; blue, 27 }  ,opacity=1 ,xscale=0.8,yscale=0.8]  {$\hat{S}$};
% Text Node
\draw (141.4,550.85) node [anchor=north west][inner sep=0.75pt]  [font=\small,color={rgb, 255:red, 208; green, 2; blue, 27 }  ,opacity=1 ,xscale=0.8,yscale=0.8]  {$\hat{S}$};
% Text Node
\draw (224.4,525.05) node [anchor=north west][inner sep=0.75pt]  [font=\small,color={rgb, 255:red, 208; green, 2; blue, 27 }  ,opacity=1 ,xscale=0.8,yscale=0.8]  {$\hat{S}$};
% Text Node
\draw (409.6,541.65) node [anchor=north west][inner sep=0.75pt]  [font=\small,color={rgb, 255:red, 208; green, 2; blue, 27 }  ,opacity=1 ,xscale=0.8,yscale=0.8]  {$\hat{S}$};

\end{tikzpicture}
    \caption{List of all possible cases that occurs when exists $T \in \mathcal{T}^I_h$ such that $S \cap \partial t_T \ne \emptyset$, represented onto the reference element. Cases 1-3 occurs when $\hat{S}$ is a segment and intersects $\partial t_{\hat{T}}$ in $D$; cases 4-7 occurs when $\hat{S}$ is a  segment and intersects $\partial t_{\hat{T}}$ in $E$ (or in $F$, due to symmetries); case 8 occurs when $\hat{S}$ is a segment or a sequence of two segments and intersects $\partial t_{\hat{T}}$ in two points; case 9 occurs when $\hat{S}$ is a sequence of two segments and $\hat{S} \cap \mathring{t}_{\hat{T}} = \emptyset$.}
\label{fig:TabellaGraficaTuttiCasi}
\end{figure}

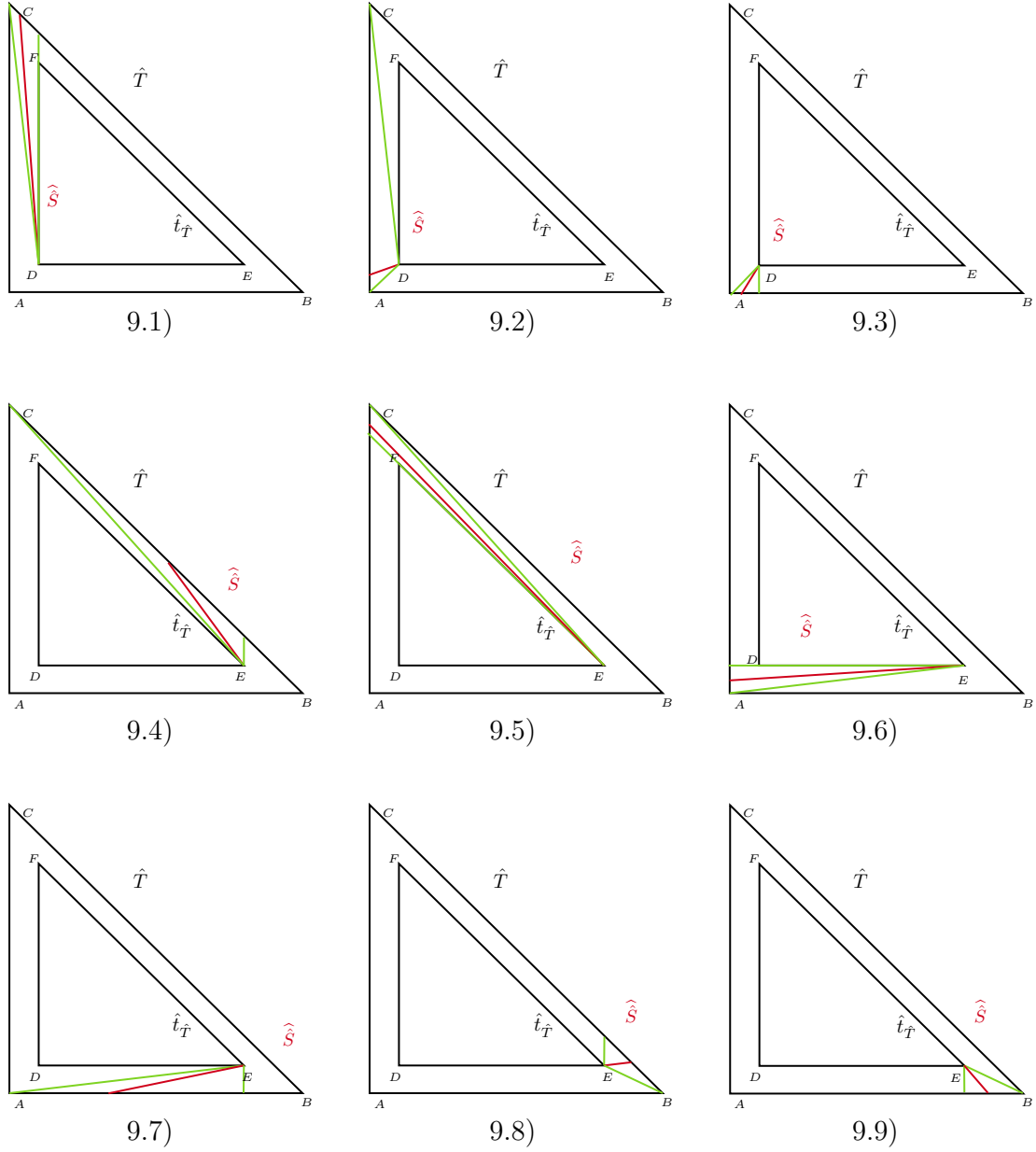
\begin{figure}
    \centering
    \tikzset{every picture/.style={line width=0.75pt}} %set default line width to 0.75pt        

\begin{tikzpicture}[x=0.75pt,y=0.75pt,yscale=-1,xscale=1]
%uncomment if require: \path (0,691); %set diagram left start at 0, and has height of 691

%Shape: Right Triangle [id:dp32390855616843695] 
\draw   (0,0.6) -- (157.62,155.61) -- (0,155.61) -- cycle ;
%Shape: Right Triangle [id:dp23673735575915567] 
\draw   (15.76,32.22) -- (126.09,140.73) -- (15.76,140.73) -- cycle ;

%Shape: Right Triangle [id:dp39877871411860233] 
\draw   (387.32,1.17) -- (544.93,156.18) -- (387.32,156.18) -- cycle ;
%Shape: Right Triangle [id:dp40411568634527684] 
\draw   (403.08,32.79) -- (513.41,141.3) -- (403.08,141.3) -- cycle ;

%Shape: Right Triangle [id:dp4026283857020786] 
\draw   (193.7,216.21) -- (351.32,371.22) -- (193.7,371.22) -- cycle ;
%Shape: Right Triangle [id:dp45184822676763525] 
\draw   (209.47,247.83) -- (319.8,356.34) -- (209.47,356.34) -- cycle ;

%Shape: Right Triangle [id:dp4902409234849392] 
\draw   (0,216.21) -- (157.62,371.22) -- (0,371.22) -- cycle ;
%Shape: Right Triangle [id:dp032285521352859115] 
\draw   (15.76,247.83) -- (126.09,356.34) -- (15.76,356.34) -- cycle ;

%Shape: Right Triangle [id:dp9994251059628556] 
\draw   (387.32,216.21) -- (544.93,371.22) -- (387.32,371.22) -- cycle ;
%Shape: Right Triangle [id:dp7937900661540899] 
\draw   (403.08,247.83) -- (513.41,356.34) -- (403.08,356.34) -- cycle ;

%Shape: Right Triangle [id:dp4887674933004805] 
\draw   (0,431.25) -- (157.62,586.26) -- (0,586.26) -- cycle ;
%Shape: Right Triangle [id:dp873227947957635] 
\draw   (15.76,462.87) -- (126.09,571.38) -- (15.76,571.38) -- cycle ;

%Shape: Right Triangle [id:dp28215051003603286] 
\draw   (193.7,431.25) -- (351.32,586.26) -- (193.7,586.26) -- cycle ;
%Shape: Right Triangle [id:dp2294978926146417] 
\draw   (209.47,462.87) -- (319.8,571.38) -- (209.47,571.38) -- cycle ;

%Shape: Right Triangle [id:dp8239346257636272] 
\draw   (387.54,431.47) -- (545.16,586.49) -- (387.54,586.49) -- cycle ;
%Shape: Right Triangle [id:dp6260362925730872] 
\draw   (403.3,463.1) -- (513.63,571.6) -- (403.3,571.6) -- cycle ;

%Shape: Right Triangle [id:dp1404910951206374] 
\draw   (193.7,0.6) -- (351.32,155.61) -- (193.7,155.61) -- cycle ;
%Shape: Right Triangle [id:dp5601383658554802] 
\draw   (209.47,32.22) -- (319.8,140.73) -- (209.47,140.73) -- cycle ;

%Straight Lines [id:da03922148355967214] 
\draw [color={rgb, 255:red, 208; green, 2; blue, 27 }  ,draw opacity=1 ]   (5.6,6.54) -- (15.76,140.73) ;
%Straight Lines [id:da7456521073111008] 
\draw [color={rgb, 255:red, 208; green, 2; blue, 27 }  ,draw opacity=1 ]   (193.19,146.52) -- (209.47,140.73) ;
%Straight Lines [id:da1874984032890149] 
\draw [color={rgb, 255:red, 208; green, 2; blue, 27 }  ,draw opacity=1 ]   (403.08,141.3) -- (393.41,156.97) ;
%Straight Lines [id:da44548283661340804] 
\draw [color={rgb, 255:red, 208; green, 2; blue, 27 }  ,draw opacity=1 ]   (126.09,356.34) -- (85.44,301.1) ;
%Straight Lines [id:da24233357794117305] 
\draw [color={rgb, 255:red, 126; green, 211; blue, 33 }  ,draw opacity=1 ]   (192.74,231.49) -- (319.8,356.34) ;
%Straight Lines [id:da4259771262555553] 
\draw [color={rgb, 255:red, 126; green, 211; blue, 33 }  ,draw opacity=1 ]   (386.73,356.37) -- (513.41,356.34) ;
%Straight Lines [id:da9583098470878206] 
\draw [color={rgb, 255:red, 126; green, 211; blue, 33 }  ,draw opacity=1 ]   (0,586.26) -- (126.09,571.38) ;
%Straight Lines [id:da6478578138727712] 
\draw [color={rgb, 255:red, 126; green, 211; blue, 33 }  ,draw opacity=1 ]   (319.94,556.01) -- (319.8,571.38) ;
%Straight Lines [id:da02315537137545176] 
\draw [color={rgb, 255:red, 126; green, 211; blue, 33 }  ,draw opacity=1 ]   (513.52,585.82) -- (513.41,571.38) ;
%Straight Lines [id:da20113425465855905] 
\draw [color={rgb, 255:red, 126; green, 211; blue, 33 }  ,draw opacity=1 ]   (126.09,356.34) -- (126.15,341.26) ;
%Straight Lines [id:da1836933177681097] 
\draw [color={rgb, 255:red, 126; green, 211; blue, 33 }  ,draw opacity=1 ]   (126.24,356.52) -- (0,216.21) ;
%Straight Lines [id:da5739878453799586] 
\draw [color={rgb, 255:red, 208; green, 2; blue, 27 }  ,draw opacity=1 ]   (193.26,226.6) -- (319.8,356.34) ;
%Straight Lines [id:da5148334269015501] 
\draw [color={rgb, 255:red, 208; green, 2; blue, 27 }  ,draw opacity=1 ]   (387.27,364.37) -- (513.41,356.34) ;
%Straight Lines [id:da8003065882445815] 
\draw [color={rgb, 255:red, 208; green, 2; blue, 27 }  ,draw opacity=1 ]   (53.27,586.37) -- (126.09,571.38) ;
%Straight Lines [id:da697326216737183] 
\draw [color={rgb, 255:red, 126; green, 211; blue, 33 }  ,draw opacity=1 ]   (125.93,586.15) -- (126.09,571.38) ;
%Straight Lines [id:da9972624058653946] 
\draw [color={rgb, 255:red, 208; green, 2; blue, 27 }  ,draw opacity=1 ]   (334.48,569.64) -- (319.8,571.38) ;
%Straight Lines [id:da6705118837764649] 
\draw [color={rgb, 255:red, 126; green, 211; blue, 33 }  ,draw opacity=1 ]   (351.32,586.26) -- (319.8,571.38) ;
%Straight Lines [id:da8832110904754708] 
\draw [color={rgb, 255:red, 126; green, 211; blue, 33 }  ,draw opacity=1 ]   (544.93,586.26) -- (513.41,571.38) ;
%Straight Lines [id:da1260278219848906] 
\draw [color={rgb, 255:red, 208; green, 2; blue, 27 }  ,draw opacity=1 ]   (526.24,586.19) -- (513.41,571.38) ;
%Straight Lines [id:da1106756866519728] 
\draw [color={rgb, 255:red, 126; green, 211; blue, 33 }  ,draw opacity=1 ]   (193.7,216.21) -- (319.8,356.34) ;
%Straight Lines [id:da9209511639849489] 
\draw [color={rgb, 255:red, 126; green, 211; blue, 33 }  ,draw opacity=1 ]   (387.32,371.22) -- (513.41,356.34) ;
%Straight Lines [id:da7503958992925296] 
\draw [color={rgb, 255:red, 126; green, 211; blue, 33 }  ,draw opacity=1 ]   (15.76,140.73) -- (-0.14,0.42) ;
%Straight Lines [id:da32253828862234635] 
\draw [color={rgb, 255:red, 126; green, 211; blue, 33 }  ,draw opacity=1 ]   (15.7,17.23) -- (15.76,140.73) ;
%Straight Lines [id:da5471805472061403] 
\draw [color={rgb, 255:red, 126; green, 211; blue, 33 }  ,draw opacity=1 ]   (209.47,140.73) -- (193.7,0.6) ;
%Straight Lines [id:da05592856630306564] 
\draw [color={rgb, 255:red, 126; green, 211; blue, 33 }  ,draw opacity=1 ]   (209.47,140.73) -- (193.7,155.61) ;
%Straight Lines [id:da25951226419191564] 
\draw [color={rgb, 255:red, 126; green, 211; blue, 33 }  ,draw opacity=1 ]   (387.78,157.23) -- (403.08,141.3) ;
%Straight Lines [id:da7482416628405111] 
\draw [color={rgb, 255:red, 126; green, 211; blue, 33 }  ,draw opacity=1 ]   (403.11,156.34) -- (403.08,141.3) ;

% Text Node
\draw (65.47,33.76) node [anchor=north west][inner sep=0.75pt]  [font=\normalsize,xscale=0.8,yscale=0.8] [align=left] {$\displaystyle \hat{T}$};
% Text Node
\draw (87.24,112.25) node [anchor=north west][inner sep=0.75pt]  [font=\normalsize,xscale=0.8,yscale=0.8]  {$\hat{t}_{\hat{T}}$};
% Text Node
\draw (1.2,158.1) node [anchor=north west][inner sep=0.75pt]  [font=\tiny,xscale=0.8,yscale=0.8]  {$A$};
% Text Node
\draw (155.65,157.18) node [anchor=north west][inner sep=0.75pt]  [font=\tiny,xscale=0.8,yscale=0.8]  {$B$};
% Text Node
\draw (5.93,1.66) node [anchor=north west][inner sep=0.75pt]  [font=\tiny,xscale=0.8,yscale=0.8]  {$C$};
% Text Node
\draw (7.5,143.06) node [anchor=north west][inner sep=0.75pt]  [font=\tiny,xscale=0.8,yscale=0.8]  {$D$};
% Text Node
\draw (123.59,143.35) node [anchor=north west][inner sep=0.75pt]  [font=\tiny,xscale=0.8,yscale=0.8]  {$E$};
% Text Node
\draw (8.85,25.88) node [anchor=north west][inner sep=0.75pt]  [font=\tiny,xscale=0.8,yscale=0.8]  {$F$};
% Text Node
\draw (259.17,29.17) node [anchor=north west][inner sep=0.75pt]  [font=\normalsize,xscale=0.8,yscale=0.8] [align=left] {$\displaystyle \hat{T}$};
% Text Node
\draw (280.1,111.81) node [anchor=north west][inner sep=0.75pt]  [font=\normalsize,xscale=0.8,yscale=0.8]  {$\hat{t}_{\hat{T}}$};
% Text Node
\draw (194.9,158.1) node [anchor=north west][inner sep=0.75pt]  [font=\tiny,xscale=0.8,yscale=0.8]  {$A$};
% Text Node
\draw (349.35,157.75) node [anchor=north west][inner sep=0.75pt]  [font=\tiny,xscale=0.8,yscale=0.8]  {$B$};
% Text Node
\draw (199.64,2.23) node [anchor=north west][inner sep=0.75pt]  [font=\tiny,xscale=0.8,yscale=0.8]  {$C$};
% Text Node
\draw (207.36,144.7) node [anchor=north west][inner sep=0.75pt]  [font=\tiny,xscale=0.8,yscale=0.8]  {$D$};
% Text Node
\draw (318.26,143.91) node [anchor=north west][inner sep=0.75pt]  [font=\tiny,xscale=0.8,yscale=0.8]  {$E$};
% Text Node
\draw (202.55,26.45) node [anchor=north west][inner sep=0.75pt]  [font=\tiny,xscale=0.8,yscale=0.8]  {$F$};
% Text Node
\draw (452.78,34.33) node [anchor=north west][inner sep=0.75pt]  [font=\normalsize,xscale=0.8,yscale=0.8] [align=left] {$\displaystyle \hat{T}$};
% Text Node
\draw (475.23,112.15) node [anchor=north west][inner sep=0.75pt]  [font=\normalsize,xscale=0.8,yscale=0.8]  {$\hat{t}_{\hat{T}}$};
% Text Node
\draw (388.51,158.67) node [anchor=north west][inner sep=0.75pt]  [font=\tiny,xscale=0.8,yscale=0.8]  {$A$};
% Text Node
\draw (542.97,157.75) node [anchor=north west][inner sep=0.75pt]  [font=\tiny,xscale=0.8,yscale=0.8]  {$B$};
% Text Node
\draw (393.25,2.23) node [anchor=north west][inner sep=0.75pt]  [font=\tiny,xscale=0.8,yscale=0.8]  {$C$};
% Text Node
\draw (405.08,144.7) node [anchor=north west][inner sep=0.75pt]  [font=\tiny,xscale=0.8,yscale=0.8]  {$D$};
% Text Node
\draw (513.31,142.32) node [anchor=north west][inner sep=0.75pt]  [font=\tiny,xscale=0.8,yscale=0.8]  {$E$};
% Text Node
\draw (396.16,26.45) node [anchor=north west][inner sep=0.75pt]  [font=\tiny,xscale=0.8,yscale=0.8]  {$F$};
% Text Node
\draw (259.17,249.37) node [anchor=north west][inner sep=0.75pt]  [font=\normalsize,xscale=0.8,yscale=0.8] [align=left] {$\displaystyle \hat{T}$};
% Text Node
\draw (282.28,328.52) node [anchor=north west][inner sep=0.75pt]  [font=\normalsize,xscale=0.8,yscale=0.8]  {$\hat{t}_{\hat{T}}$};
% Text Node
\draw (194.9,373.71) node [anchor=north west][inner sep=0.75pt]  [font=\tiny,xscale=0.8,yscale=0.8]  {$A$};
% Text Node
\draw (349.35,372.79) node [anchor=north west][inner sep=0.75pt]  [font=\tiny,xscale=0.8,yscale=0.8]  {$B$};
% Text Node
\draw (199.64,217.27) node [anchor=north west][inner sep=0.75pt]  [font=\tiny,xscale=0.8,yscale=0.8]  {$C$};
% Text Node
\draw (202.91,358.96) node [anchor=north west][inner sep=0.75pt]  [font=\tiny,xscale=0.8,yscale=0.8]  {$D$};
% Text Node
\draw (312.44,358.96) node [anchor=north west][inner sep=0.75pt]  [font=\tiny,xscale=0.8,yscale=0.8]  {$E$};
% Text Node
\draw (202.55,241.49) node [anchor=north west][inner sep=0.75pt]  [font=\tiny,xscale=0.8,yscale=0.8]  {$F$};
% Text Node
\draw (65.47,249.37) node [anchor=north west][inner sep=0.75pt]  [font=\normalsize,xscale=0.8,yscale=0.8] [align=left] {$\displaystyle \hat{T}$};
% Text Node
\draw (86.58,327.19) node [anchor=north west][inner sep=0.75pt]  [font=\normalsize,xscale=0.8,yscale=0.8]  {$\hat{t}_{\hat{T}}$};
% Text Node
\draw (1.2,373.71) node [anchor=north west][inner sep=0.75pt]  [font=\tiny,xscale=0.8,yscale=0.8]  {$A$};
% Text Node
\draw (155.65,372.79) node [anchor=north west][inner sep=0.75pt]  [font=\tiny,xscale=0.8,yscale=0.8]  {$B$};
% Text Node
\draw (5.93,217.27) node [anchor=north west][inner sep=0.75pt]  [font=\tiny,xscale=0.8,yscale=0.8]  {$C$};
% Text Node
\draw (9.21,358.96) node [anchor=north west][inner sep=0.75pt]  [font=\tiny,xscale=0.8,yscale=0.8]  {$D$};
% Text Node
\draw (119.99,358.96) node [anchor=north west][inner sep=0.75pt]  [font=\tiny,xscale=0.8,yscale=0.8]  {$E$};
% Text Node
\draw (8.85,241.49) node [anchor=north west][inner sep=0.75pt]  [font=\tiny,xscale=0.8,yscale=0.8]  {$F$};
% Text Node
\draw (452.78,249.37) node [anchor=north west][inner sep=0.75pt]  [font=\normalsize,xscale=0.8,yscale=0.8] [align=left] {$\displaystyle \hat{T}$};
% Text Node
\draw (474.56,327.86) node [anchor=north west][inner sep=0.75pt]  [font=\normalsize,xscale=0.8,yscale=0.8]  {$\hat{t}_{\hat{T}}$};
% Text Node
\draw (388.51,373.71) node [anchor=north west][inner sep=0.75pt]  [font=\tiny,xscale=0.8,yscale=0.8]  {$A$};
% Text Node
\draw (543.4,373.28) node [anchor=north west][inner sep=0.75pt]  [font=\tiny,xscale=0.8,yscale=0.8]  {$B$};
% Text Node
\draw (393.25,217.27) node [anchor=north west][inner sep=0.75pt]  [font=\tiny,xscale=0.8,yscale=0.8]  {$C$};
% Text Node
\draw (394.88,349.67) node [anchor=north west][inner sep=0.75pt]  [font=\tiny,xscale=0.8,yscale=0.8]  {$D$};
% Text Node
\draw (508.55,360.78) node [anchor=north west][inner sep=0.75pt]  [font=\tiny,xscale=0.8,yscale=0.8]  {$E$};
% Text Node
\draw (396.16,241.49) node [anchor=north west][inner sep=0.75pt]  [font=\tiny,xscale=0.8,yscale=0.8]  {$F$};
% Text Node
\draw (65.47,464.41) node [anchor=north west][inner sep=0.75pt]  [font=\normalsize,xscale=0.8,yscale=0.8] [align=left] {$\displaystyle \hat{T}$};
% Text Node
\draw (86.58,542.23) node [anchor=north west][inner sep=0.75pt]  [font=\normalsize,xscale=0.8,yscale=0.8]  {$\hat{t}_{\hat{T}}$};
% Text Node
\draw (1.2,588.76) node [anchor=north west][inner sep=0.75pt]  [font=\tiny,xscale=0.8,yscale=0.8]  {$A$};
% Text Node
\draw (155.65,587.83) node [anchor=north west][inner sep=0.75pt]  [font=\tiny,xscale=0.8,yscale=0.8]  {$B$};
% Text Node
\draw (5.93,432.31) node [anchor=north west][inner sep=0.75pt]  [font=\tiny,xscale=0.8,yscale=0.8]  {$C$};
% Text Node
\draw (9.21,574) node [anchor=north west][inner sep=0.75pt]  [font=\tiny,xscale=0.8,yscale=0.8]  {$D$};
% Text Node
\draw (123.59,574) node [anchor=north west][inner sep=0.75pt]  [font=\tiny,xscale=0.8,yscale=0.8]  {$E$};
% Text Node
\draw (8.85,456.54) node [anchor=north west][inner sep=0.75pt]  [font=\tiny,xscale=0.8,yscale=0.8]  {$F$};
% Text Node
\draw (259.17,464.41) node [anchor=north west][inner sep=0.75pt]  [font=\normalsize,xscale=0.8,yscale=0.8] [align=left] {$\displaystyle \hat{T}$};
% Text Node
\draw (280.95,542.23) node [anchor=north west][inner sep=0.75pt]  [font=\normalsize,xscale=0.8,yscale=0.8]  {$\hat{t}_{\hat{T}}$};
% Text Node
\draw (194.9,588.76) node [anchor=north west][inner sep=0.75pt]  [font=\tiny,xscale=0.8,yscale=0.8]  {$A$};
% Text Node
\draw (349.35,587.83) node [anchor=north west][inner sep=0.75pt]  [font=\tiny,xscale=0.8,yscale=0.8]  {$B$};
% Text Node
\draw (199.64,432.31) node [anchor=north west][inner sep=0.75pt]  [font=\tiny,xscale=0.8,yscale=0.8]  {$C$};
% Text Node
\draw (202.91,574) node [anchor=north west][inner sep=0.75pt]  [font=\tiny,xscale=0.8,yscale=0.8]  {$D$};
% Text Node
\draw (317.3,574) node [anchor=north west][inner sep=0.75pt]  [font=\tiny,xscale=0.8,yscale=0.8]  {$E$};
% Text Node
\draw (202.55,456.54) node [anchor=north west][inner sep=0.75pt]  [font=\tiny,xscale=0.8,yscale=0.8]  {$F$};
% Text Node
\draw (452.78,464.41) node [anchor=north west][inner sep=0.75pt]  [font=\normalsize,xscale=0.8,yscale=0.8] [align=left] {$\displaystyle \hat{T}$};
% Text Node
\draw (475.9,542.9) node [anchor=north west][inner sep=0.75pt]  [font=\normalsize,xscale=0.8,yscale=0.8]  {$\hat{t}_{\hat{T}}$};
% Text Node
\draw (388.51,588.76) node [anchor=north west][inner sep=0.75pt]  [font=\tiny,xscale=0.8,yscale=0.8]  {$A$};
% Text Node
\draw (542.97,587.83) node [anchor=north west][inner sep=0.75pt]  [font=\tiny,xscale=0.8,yscale=0.8]  {$B$};
% Text Node
\draw (393.25,432.31) node [anchor=north west][inner sep=0.75pt]  [font=\tiny,xscale=0.8,yscale=0.8]  {$C$};
% Text Node
\draw (396.53,574) node [anchor=north west][inner sep=0.75pt]  [font=\tiny,xscale=0.8,yscale=0.8]  {$D$};
% Text Node
\draw (504.55,574.91) node [anchor=north west][inner sep=0.75pt]  [font=\tiny,xscale=0.8,yscale=0.8]  {$E$};
% Text Node
\draw (396.16,456.54) node [anchor=north west][inner sep=0.75pt]  [font=\tiny,xscale=0.8,yscale=0.8]  {$F$};
% Text Node
\draw (18.86,96.65) node [anchor=north west][inner sep=0.75pt]  [font=\small,color={rgb, 255:red, 208; green, 2; blue, 27 }  ,opacity=1 ,xscale=0.8,yscale=0.8]  {$\widehat{\hat{S}}$};
% Text Node
\draw (62,163.4) node [anchor=north west][inner sep=0.75pt]  [font=\Large,xscale=0.8,yscale=0.8]  {$9.1)$};
% Text Node
\draw (257,163.4) node [anchor=north west][inner sep=0.75pt]  [font=\Large,xscale=0.8,yscale=0.8]  {$9.2)$};
% Text Node
\draw (452,163.4) node [anchor=north west][inner sep=0.75pt]  [font=\Large,xscale=0.8,yscale=0.8]  {$9.3)$};
% Text Node
\draw (62,383.4) node [anchor=north west][inner sep=0.75pt]  [font=\Large,xscale=0.8,yscale=0.8]  {$9.4)$};
% Text Node
\draw (257,383.4) node [anchor=north west][inner sep=0.75pt]  [font=\Large,xscale=0.8,yscale=0.8]  {$9.5)$};
% Text Node
\draw (452,383.4) node [anchor=north west][inner sep=0.75pt]  [font=\Large,xscale=0.8,yscale=0.8]  {$9.6)$};
% Text Node
\draw (62,598.4) node [anchor=north west][inner sep=0.75pt]  [font=\Large,xscale=0.8,yscale=0.8]  {$9.7)$};
% Text Node
\draw (257,598.4) node [anchor=north west][inner sep=0.75pt]  [font=\Large,xscale=0.8,yscale=0.8]  {$9.8)$};
% Text Node
\draw (452,598.4) node [anchor=north west][inner sep=0.75pt]  [font=\Large,xscale=0.8,yscale=0.8]  {$9.9)$};
% Text Node
\draw (215.41,110.65) node [anchor=north west][inner sep=0.75pt]  [font=\small,color={rgb, 255:red, 208; green, 2; blue, 27 }  ,opacity=1 ,xscale=0.8,yscale=0.8]  {$\widehat{\hat{S}}$};
% Text Node
\draw (409.26,114.88) node [anchor=north west][inner sep=0.75pt]  [font=\small,color={rgb, 255:red, 208; green, 2; blue, 27 }  ,opacity=1 ,xscale=0.8,yscale=0.8]  {$\widehat{\hat{S}}$};
% Text Node
\draw (116.15,302.48) node [anchor=north west][inner sep=0.75pt]  [font=\small,color={rgb, 255:red, 208; green, 2; blue, 27 }  ,opacity=1 ,xscale=0.8,yscale=0.8]  {$\widehat{\hat{S}}$};
% Text Node
\draw (300.52,288.03) node [anchor=north west][inner sep=0.75pt]  [font=\small,color={rgb, 255:red, 208; green, 2; blue, 27 }  ,opacity=1 ,xscale=0.8,yscale=0.8]  {$\widehat{\hat{S}}$};
% Text Node
\draw (423.93,327.81) node [anchor=north west][inner sep=0.75pt]  [font=\small,color={rgb, 255:red, 208; green, 2; blue, 27 }  ,opacity=1 ,xscale=0.8,yscale=0.8]  {$\widehat{\hat{S}}$};
% Text Node
\draw (329.96,534.81) node [anchor=north west][inner sep=0.75pt]  [font=\small,color={rgb, 255:red, 208; green, 2; blue, 27 }  ,opacity=1 ,xscale=0.8,yscale=0.8]  {$\widehat{\hat{S}}$};
% Text Node
\draw (517.7,534.92) node [anchor=north west][inner sep=0.75pt]  [font=\small,color={rgb, 255:red, 208; green, 2; blue, 27 }  ,opacity=1 ,xscale=0.8,yscale=0.8]  {$\widehat{\hat{S}}$};
% Text Node
\draw (145.78,547.59) node [anchor=north west][inner sep=0.75pt]  [font=\small,color={rgb, 255:red, 208; green, 2; blue, 27 }  ,opacity=1 ,xscale=0.8,yscale=0.8]  {$\widehat{\hat{S}}$};

\end{tikzpicture}
    \caption{The figure shows all possible positions of the portion $\overline{S}$ of the element 
$\hat{S}$. The limiting cases, highlighted in green, allow for a classification of all possible configurations.}
    \label{fig:BrokenLineCasistiche}
\end{figure}

\begin{enumerate}[label=9.\arabic*]

\item \label{item: Case 9.1} We consider $\overline{S}$ to be a portion of $\hat{S}$, terminating at the point $D$, with slope $m\in (-\infty,\frac{a-1}{a}]$. Then
    \begin{equation*}
    \begin{aligned}
    I_{\overline{S}}=\int_{\frac{a(m-1)+1}{m+1}}^{a} 27t[m(t-a)+a][1-t-m(t-a)-a]\sqrt{1+m^2}dt =\\
    = -\frac{9 (1-2a )^2 \sqrt{1 + m^2} \left[ m + 2a (1+a+m^2+am(4+m)) \right]}{4 (1 + m)^3}.
    \end{aligned}
\end{equation*}
Numerically, we can see that this integral reach a minimum in $m$ as $m \to -\infty $ or $m=\frac{a-1}{a}$, depending on the value of $a$. So,
\begin{equation}
\label{eq:LowerBoundCase9.1}
        \inf_{m\in (-\infty,\frac{a-1}{a}]} I_{\overline{S}} = \min_{a \in [0,\frac{1}{3}]} \left[\frac{9}{2} a(1+a)(1-2a)^2,   -\frac{9 \sqrt{a^2 + (a - 1)^2} \cdot (2a - 1)a(3a+1)}{4 }\right].
\end{equation}

\item \label{item: Case 9.2}  We consider $\overline{S}$ to be a portion of $\hat{S}$, terminating at the point $D$, with slope $m \in [\frac{a-1}{a},1]$. 
\begin{equation*}
    \begin{aligned}
        I_{\overline{S}}=\int_0^{a} 27t[m(t-a)+a][1-t-m(t-a)-a]\sqrt{1+m^2}dt =\\= -\frac{9}{4}a^3[2(-3+m)+a(10-5m+m^2)]\sqrt{m^2+1}.
    \end{aligned}
\end{equation*}
Numerically, we can see that this integral reach a minimum inside the interval of $m$. So we can perform its $m-$th derivative:
\begin{equation*}
    \frac{\partial I_{\overline{S}}}{\partial m} = \frac{a^3(m-1)(-2+4m+a(5+m(-7+3m)))}{\sqrt{1+m^2}}.
\end{equation*}
So, it is easy to see that the lower bound for $I_{\hat{S}}$ is obtained when $m= \frac{7a-4+ \sqrt{-11a^2-32a+16}}{6a}$.

\item \label{item: Case 9.3} Here, $\overline{S}$ is a segment connecting point $D$ with side $\overline{AB}$, so that it belongs to the family of lines with equation $\hat{y}-a=m(\hat{x}-a)$, with $m \in [1, +\infty]$. For symmetric reason, it is easy to see that this case is equivalent to the previous one (case~\ref{item: Case 9.2}) with $m \in [0,1]$.

\item \label{item: Case 9.4} We consider $\overline{S}$ to be a portion of $\hat{S}$, terminating at the point $E$, with slope $m\in (-\infty; \frac{1-a}{2a-1}]$. 
\begin{equation*}
\begin{aligned}
I_{\overline{S}}=\int_{\frac{1-a-2am+m}{m+1}}^{1-2a} 27t[m(t-(1-2a))+a][1-t-m(t-(1-2a))-a]\sqrt{1+m^2}dt =\\= \frac{9 a^3 \sqrt{1 + m^2} \left( -2  (1 + m)(4m + 3) + a(16m^2+25m+10) \right)}{4 (1 + m)^3}.
\end{aligned}
\end{equation*}
Numerically, we can see that this integral decreases as $m \to -\infty $:  
\begin{equation}
\label{eq:LowerBoundIntegralCase9.4}
\inf_{m\in (-\infty; \frac{1-a}{2a-1}]} I_{\overline{S}} = 18 a^3(1-2a).
\end{equation}

\item \label{item: Case 9.5} We consider $\overline{S}$ to be a portion of $\hat{S}$, terminating at the point $E$, with slope $m \in [\frac{1-a}{2a-1},-1]$.
\begin{equation*}
\begin{aligned}
I_{\overline{S}}= \int_0^{1-2a} 27t[m(t-(1-2a))+a][1-t-m(t-(1-2a))-a]\sqrt{1+m^2}dt =\\=  -\frac{9 (1-2a)^2\sqrt{1 + m^2} \left[ (1-2a)^2m^2+4(a-1)am-2a(a+1)+m \right]}{4}.
\end{aligned}
\end{equation*}
Given the range of values of \(a\) that will be considered later, namely \(0<a<\tfrac{1}{4}\), numerically we can see that this integral has a lower bound for $m=\frac{1-a}{2a-1}$. The value of the same integral with $m=\frac{1-a}{2a-1}$ was already considered in the previous case, so we can take~\eqref{eq:LowerBoundIntegralCase9.4} as lower bound.

\item \label{item: Case 9.6} We consider $\overline{S}$ to be a portion of $\hat{S}$, terminating at the point $E$, with slope $m \in [ 0, \frac{a}{1-2a}]$.
\begin{equation*}
\begin{aligned}
I_{\overline{S}}= \int_0^{1-2a} 27t[m(t-(1-2a))+a][1-t-m(t-(1-2a))-a]\sqrt{1+m^2}dt =\\=  \frac{9 (1-2a)^2\sqrt{1 + m^2} \left[ (1-2a)^2m^2+4(a-1)am-2a(a+1)+m \right]}{4}.
\end{aligned}
\end{equation*}
Given the range of values of \(a\) that will be considered later, namely \(0<a<\tfrac{1}{4}\), numerically we can see that this integral has a lower bound for $m=\frac{a}{1-2a}$. The value of the same integral with $m=\frac{1-a}{2a-1}$ will be addressed in the previous case, so we can take~\eqref{eq:lowerbound9.7} as lower bound.

\item \label{item: Case 9.7} We consider $\overline{S}$ to be a portion of $\hat{S}$, terminating at the point $E$, with slope 
$m \in [ \frac{a}{1-2a}, + \infty]$.
\begin{equation*}
\begin{aligned}
I_{\overline{S}}= \int_{1-2a-\frac{a}{m}}^{1-2a} 27t[m(t-(1-2a))+a][1-t-m(t-(1-2a))-a]\sqrt{1+m^2}dt =\\=  -\frac{9a^3\sqrt{1 + m^2} \left[ a(16m^2+7m+1)-2m(4m+1)\right]}{4m^3}
\end{aligned}
\end{equation*}
Given the range of values of \(a\) that will be considered later, namely \(0<a<\tfrac{1}{4}\), numerically we observe that we have a lower bound as $m \to + \infty$. So,
\begin{equation}
\label{eq:lowerbound9.7}
    \inf_{m \in [ \frac{a}{1-2a}, + \infty]} I_{\overline{S}} = 18 a^3(1-2a).
\end{equation}
Note that is the same as~\eqref{eq:LowerBoundIntegralCase9.4}.

\item \label{item: Case 9.8} 
Here, $\overline{S}$ is a segment connecting point $E$ with side $\overline{BC}$, so that it belongs to the family of lines with equation $\hat{y}-a=m(\hat{x}-(1-2a))$. As we can see from the Figure~\ref{fig:BrokenLineCasistiche}, the situation is the same of Case~\ref{item: Case 9.4}, with $m \in [ -\frac{1}{2}, + \infty]$. So the integral will be the same (but with a change of sign, since in this case $1-2a \leq \frac{1-a-2am-m}{m+1}$):
\begin{equation*}
\begin{aligned}
I_{\overline{S}}= \int_{1-2a}^{\frac{1-a-2am-m}{m+1}} 27t[m(t-(1-2a))+a][1-t-m(t-(1-2a))-a]\sqrt{1+m^2}dt =\\=  -\frac{9 a^3 \sqrt{1 + m^2} \left( -2  (1 + m)(4m + 3) + a(16m^2+25m+10) \right)}{4 (1 + m)^3}
\end{aligned}
\end{equation*}
Numerically, we can see that this integral reach a minimum inside the interval of $m$. So we can perform the $m-$th derivative:
\begin{equation*}
    \frac{\partial I_{\overline{S}} }{\partial m} = -\frac{9a^3\{ 4+6m-2m^2(4+5m)+a[-5+m(-8+14m+23m^2)]\}}{4(m+1)^4 \sqrt{m^2+1}}
\end{equation*}
We obtain numerically the minimum. Its values and the values of the corresponding integrals are listed in Table~\ref{Table:Case9.8}.

\begin{table}
\footnotesize
\centering
\begin{tabular}{l|c|c|c|r}
\toprule
\textbf{$a$} & \textbf{$m$} for 9.8 & \textbf{$I_{\hat{S}}$ } for 9.8 & \textbf{$m$ } for 9.9& \textbf{$I_{\hat{S}}$ } for 9.9\\
\midrule
0.0238 & 0.7471 & 0.0001 &  -3.5660 & 0.0002 \\
0.0476 & 0.7546  & 0.0011 & -3.7457 & 0.0017 \\
0.0714 & 0.7629& 0.0035 & -3.9624 & 0.0055 \\
0.0952 &  0.7721& 0.0079 &  -4.2295 & 0.0123 \\
0.1190 & 0.7823& 0.0146 & -4.5680 & 0.0227 \\
0.1429 & 0.7939 & 0.0239 & -5.0126 & 0.0369 \\
0.1667 & 0.8071& 0.0357 & -5.6257& 0.0548 \\
0.1905 & 0.8222 & 0.0500 & -6.5310 & 0.0762 \\
0.2143 & 0.8396 & 0.0666 & -8.0165 & 0.1005 \\
0.2381 & 0.8602 & 0.0849 & -10.9427& 0.1268 \\
\end{tabular} 
\caption{Analysis of the minimum value of the integral for Case~\ref{item: Case 9.8} (left) and for Case~\ref{item: Case 9.9} (right). For fixed $a$, the table reports the value of $m$ at which the integral attains its minimum, together with the corresponding minimum value.}
\label{Table:Case9.8}
\end{table}

\item \label{item: Case 9.9} 
Here, $\overline{S}$ is a segment connecting point $E$ with side $\overline{AB}$, so that it belongs to the family of lines with equation $\hat{y}-a=m(\hat{x}-(1-2a))$. As we can see from the Figure~\ref{fig:BrokenLineCasistiche}, the situation is the same of Case~\ref{item: Case 9.7}, with $m \in [ -\infty, -\frac{1}{2}]$. So the integral will be the same (but with a change of sign, since in this case $1-2a \leq 1-2a-\frac{a}{m}$):
\begin{equation*}
\begin{aligned} 
I_{\overline{S}}= \int_{1-2a}^{1-2a-\frac{a}{m}} 27t[m(t-(1-2a))+a][1-t-m(t-(1-2a))-a]\sqrt{1+m^2}dt =\\= \frac{9a^3\sqrt{1 + m^2} \left[ a(16m^2+7m+1)-2m(4m+1)\right]}{4m^3}
\end{aligned}
\end{equation*}
Numerically, we can see that this integral reach a minimum inside the interval of $m$, if $0 \leq a \leq \frac{1}{4}$. So we can perform the $m-$th derivative and obtain numerically a lower bound, in Table~\ref{Table:Case9.8}.

\end{enumerate}

\bibliographystyle{IEEEtranDOI}
\bibliography{biblio.bib}

@book{ciarlet1978finite,
  title={The Finite Element Method for Elliptic Problems},
  author={Ciarlet, P.G.},
  series={Studies in Mathematics and its Applications},
  year={1978},
  publisher={North Holland}
}

@article{berrone2006robust,
  title={Robust a posteriori error estimates for finite element discretizations of the heat equation with discontinuous coefficients},
  author={Berrone, Stefano},
  journal={ESAIM: Mod{\'e}lisation math{\'e}matique et analyse num{\'e}rique},
  volume={40},
  number={6},
  pages={991--1021},
  year={2006},
  doi={10.1051/m2an:2006034}
}

@article{CanutoNochetto2024,
      title={Adaptive Finite Element Methods}, 
      author={Bonito, Andrea and Canuto, Claudio and Nochetto, Ricardo H. and Veeser, Andreas},
      volume={33}, 
      doi={10.1017/S0962492924000011}, 
      journal={Acta Numerica}, 
      year={2024}, 
      pages={163–485}
}

@article{arnold1984stable,
  title={A stable finite element for the Stokes equations},
  author={Arnold, Douglas N and Brezzi, Franco and Fortin, Michel},
  journal={Calcolo},
  volume={21},
  number={4},
  pages={337--344},
  year={1984},
  publisher={Springer},
  doi={10.1007/BF02576171}
}

@article{Fortin1977,
  author  = {Fortin, Michel},
  title   = {An analysis of the convergence of mixed finite element methods},
  journal = {RAIRO. Analyse num\'erique},
  volume  = {11},
  number  = {4},
  pages   = {341--354},
  year    = {1977},
  doi     = {10.1051/m2an/1977110403411},
}

@book{BoffiBrezziFortin2013,
  author    = {Boffi, Daniele and Brezzi, Franco and Fortin, Michel},
  title     = {Mixed Finite Element Methods and Applications},
  series    = {Springer Series in Computational Mathematics},
  volume    = {44},
  publisher = {Springer Berlin Heidelberg},
  address   = {Berlin, Heidelberg},
  year      = {2013},
  pages     = {xiv + 685},
  isbn      = {978-3-642-36518-8},
  doi       = {10.1007/978-3-642-36519-5}
}

@article{chouly2026review,
  title={A Review on Some Discrete Variational Techniques for the Approximation of Essential Boundary Conditions},
  author={Chouly, Franz},
  journal={Vietnam Journal of Mathematics},
  volume={54},
  number={1},
  pages={73--115},
  year={2026},
  publisher={Springer},
  doi={10.1007/s10013-024-00702-1}
}

@article{Polydim,
title = {{POLYDIM: A C++ library for POLYtopal DIscretization Methods}},
journal = {Computer Physics Communications},
volume = {320},
pages = {109937},
year = {2026},
issn = {0010-4655},
doi = {10.1016/j.cpc.2025.109937},
author = {Stefano Berrone and Andrea Borio and Gioana Teora and Fabio Vicini},
}

@incollection{berrone2016adaptive,
  title={An adaptive fictitious domain method for elliptic problems},
  author={Berrone, Stefano and Bonito, Andrea and Verani, Marco},
  booktitle={Advances in Discretization Methods: Discontinuities, Virtual Elements, Fictitious Domain Methods},
  pages={229--244},
  year={2016},
  publisher={Springer},
  doi={10.1007/978-3-319-41246-7\_11}
}

@article{mommer2006smoothness,
  title={A smoothness preserving fictitious domain method for elliptic boundary-value problems},
  author={Mommer, Mario S},
  journal={IMA journal of numerical analysis},
  volume={26},
  number={3},
  pages={503--524},
  year={2006},
  publisher={Oxford University Press},
  doi={10.1093/imanum/dri045}
}

@article{berrone2019optimal,
  title={An optimal adaptive fictitious domain method},
  author={Berrone, Stefano and Bonito, Andrea and Stevenson, Rob and Verani, Marco},
  journal={Mathematics of Computation},
  volume={88},
  number={319},
  pages={2101--2134},
  year={2019},
   doi={10.1090/mcom/3414}
}

@article{Girault1995,
  author    = {Girault, V. and Glowinski, R.},
  title     = {Error analysis of a fictitious domain method applied to a Dirichlet problem},
  journal   = {Japan Journal of Industrial and Applied Mathematics},
  volume    = {12},
  number    = {3},
  pages     = {487--514},
  year      = {1995},
  doi       = {10.1007/BF03167240},
}

@article{GLOWINSKI1994283,
title = {A fictitious domain method for Dirichlet problem and applications},
journal = {Computer Methods in Applied Mechanics and Engineering},
volume = {111},
number = {3},
pages = {283-303},
year = {1994},
issn = {0045-7825},
doi = {10.1016/0045-7825(94)90135-X},
author = {Roland Glowinski and Tsorng-Whay Pan and Jacques Periaux}
}

@article{Babuska1972/73,
author = {Babuska, Ivo},
journal = {Numerische Mathematik},
pages = {179-192},
title = {The Finite Element Method with Lagrangian Multipliers.},
doi = {10.1007/BF01436561},
volume = {20},
year = {1973},
}

@article{burman2025cut,
  title={Cut finite element methods},
  author={Burman, Erik and Hansbo, Peter and Larson, Mats G and Zahedi, Sara},
  journal={Acta Numerica},
  volume={34},
  pages={1--121},
  year={2025},
  publisher={Cambridge University Press},
  doi={10.1017/S0962492925000017}
}

@article{burman2010fictitious,
  title={Fictitious domain finite element methods using cut elements: I. A stabilized Lagrange multiplier method},
  author={Burman, Erik and Hansbo, Peter},
  journal={Computer Methods in Applied Mechanics and Engineering},
  volume={199},
  number={41-44},
  pages={2680--2686},
  year={2010},
  publisher={Elsevier},
  doi={10.1016/j.cma.2010.05.011}
}

@article{burman2012fictitious,
  title={Fictitious domain finite element methods using cut elements: II. A stabilized Nitsche method},
  author={Burman, Erik and Hansbo, Peter},
  journal={Applied Numerical Mathematics},
  volume={62},
  number={4},
  pages={328--341},
  year={2012},
  publisher={Elsevier},
  doi={10.1016/j.apnum.2011.01.008}
}

@article{hansbo2002unfitted,
  title={An unfitted finite element method, based on Nitsche’s method, for elliptic interface problems},
  author={Hansbo, Anita and Hansbo, Peter},
  journal={Computer methods in applied mechanics and engineering},
  volume={191},
  number={47-48},
  pages={5537--5552},
  year={2002},
  publisher={Elsevier},
  doi={10.1016/S0045-7825(02)00524-8}
}

@article{clement1975approximation,
  title     = {Approximation by finite element functions using local regularization},
  author    = {Cl{\'e}ment, Ph.},
  journal   = {Revue fran{\c{c}}aise d'automatique, informatique, recherche op{\'e}rationnelle. Analyse num{\'e}rique},
  volume    = {9},
  number    = {R2},
  pages     = {77--84},
  year      = {1975},
  publisher = {EDP Sciences},
  doi       = {10.1051/m2an/197509R200771}
}

\end{document}